\documentclass{amsart}

\usepackage{Preset}
\usepackage[makeroom]{cancel}
\usepackage{nccmath}

\usepackage{mathrsfs}
\usepackage[latin1]{inputenc}
\usepackage[T1]{fontenc}
\usepackage{tikz}
\usepackage{tikz-cd}
\usepackage[colorlinks=true,linkcolor=blue!50!black,anchorcolor=red,citecolor=blue,filecolor=black,menucolor=black,runcolor=black,urlcolor=black]{hyperref}

\makeatletter
\def\hookrightarrowfill@{%
  \arrowfill@\relbar\relbar\hookrightarrow
}
\newcommand{\xhook}[2][]{%
  \ext@arrow 0359\hookrightarrowfill@{#1}{#2}%
}
\makeatother

\newcommand{\seq}[1]{\left\langle #1 \right\rangle}

\usetikzlibrary{decorations.pathmorphing}
\usetikzlibrary{shapes,arrows,graphs}
\setlist[enumerate,1]{label={\textnormal{(\arabic*)}}}

\title{The combinatorics of sector renormalization}
\author{Willie Rush Lim}
\address{Dept. of Mathematics, Brown University, RI 02912}
\email{willie\_rush\_lim@brown.edu}

\begin{document}

\begin{abstract}
    The goal of this paper is to systematically develop the fundamental arithmetic and combinatorial properties of the sector renormalization operation on rigid rotations.
    We employ the specific framework of modified continued fractions appropriate for sector renormalization and analyze their properties.
    By allowing infinite first return times, this framework yields a dynamical compactification of the space of irrationals called the parabolic compactification; we show that it is characterized by some universal properties. 
    We also discuss the corresponding natural extension and introduce the continuant group, the time semigroup, and topological cascades. 
    For example, we demonstrate how a bi-infinite tower of sector renormalizations of irrational rotations can be packaged within a single dynamical plane as a cascade of translations. 
    This will serve as a foundational combinatorial tool for studying the geometric properties of sector renormalizations of holomorphic maps with irrationally indifferent fixed points, particularly neutral quadratic polynomials.
\end{abstract}

\maketitle

\setcounter{tocdepth}{1}
\tableofcontents

\section{Introduction}

The dynamics of an irrational rotation 
\[
    \rotate_{\theta} : z \mapsto e^{2\pi i \theta}z
\]
are traditionally classified via the entries $a_n$ of its regular continued fraction expansion.
Various ideas behind such relation goes back to the classical work of Poincar\'e, Denjoy, Siegel, and Arnol'd; refer to \cite{dMvS93} for historical background.
The compactification of irrationals by allowing the entry $a_n$ to take the value $\infty$ was used in Lanford's program \cite{Lan88} on the renormalization theory of analytic critical circle maps.
This program was completed in the series of works by de Melo, de Faria, and Yampolsky; see \S\ref{sss:circle-map} for details.
A related renormalization program for neutral quadratic polynomials is one of the central topics in complex dynamics.

This paper explores a related symbolic coding arising from an operation called \emph{sector renormalization}. 
Let
\[
    \Theta := \left( -\frac{1}{2},\frac{1}{2} \right) \backslash \mathbb{Q}.
\]
For any irrational $\theta$ in $\Theta$, let $S$ be a fundamental sector of $\rotate_\theta$ inside the closed unit disk $\overline{\D}$ that is cut out by a pair of radial arcs subtending an angle of $2\pi|\theta|$ radians. 
The power map $\psi(z) = z^{1/|\theta|}$ glues the radial sides of $S$ together, mapping the sector back onto $\overline{\D}$. 
The first return map of $\rotate_\theta$ to $S$ is a piecewise continuous map consisting of a pair of iterates $(\rotate_\theta^{\abar}, \rotate_\theta^{\abar+1})$. Under $\psi$, this pair projects onto a new irrational rotation $\rotate_{\gauss(\theta)}: \overline{\D} \to \overline{\D}$, where the map $\gauss: \Theta \to \Theta$ is uniquely defined by the relation
\[
    \gauss(\theta) \equiv -\frac{1}{\theta} \quad (\textnormal{mod }1).
\]
Iterating this operation yields a symbolic representation of $\theta$ by an infinite sequence 
\[
\langle(\varepsilon_n, \abar_n)\rangle_{n \geq 1},
\]
where $\varepsilon_{n+1} \in \{-1, +1\}$ and $\abar_{n+1} \in \N_{\geq 2}$ track the orientation and first return time at renormalization level $n$.

This construction provides a combinatorial model for the classical sector renormalization of arbitrary holomorphic germs pioneered by Douady--Ghys \cite{D87} and Yoccoz \cite{Yoc95}.
It connects to the modified continued fractions commonly utilized in near-parabolic renormalization theory; see \S\ref{ss:motivation} for historical remarks. 

It has always been assumed that various dynamically motivated encodings of $\Irrat$ are ``routinely'' equivalent and lead to the same compactification of $\Irrat$.
We justify the appropriateness of the compactification $\Theta \leadsto \TheCpt$ by allowing entries $\abar_n$ to take the value $\infty$ by establishing two universal properties of $\TheCpt$. 
We also analyze the corresponding inverse limit $\TheBiCpt$ for applications to the unified theory of neutral renormalization.
Ultimately, our goal is to systematically describe:
\begin{enumerate}
    \item[(I)] the dynamical and arithmetic aspects of $\gauss: \Theta \to \Theta$;
    \item[(II)] the appropriate dynamical compactification of $(\Theta,\gauss)$;
    \item[(III)] the properties of the natural extension of $\gauss$.
\end{enumerate}

\subsection{Modified continued fractions}
\label{ss:intro-2}

The arithmetic properties of $\gauss$ boil down to a variant of the nearest integer continued fraction expansion.
It first appeared in the work of Hurwitz \cite{Hur} in 1889.
We take a more dynamical point of view and emphasize that the expansion is parameterized by the infinite sequence space
\[
    \The := \left\{ \seq{ (\varepsilon_n, \abar_n) }_{n\geq 1} \: : \: \varepsilon_n \in \{-1,+1\}, \abar_n \in \N_{\geq 2} \right\}.
\]
For any sequence of non-zero integers $x_1,x_2,x_3,\ldots$, denote
\begin{equation}
    [ x_1,x_2,x_3,\ldots ]_-  := \frac{1}{x_1- \frac{1}{x_2-\frac{1}{x_3 - \ldots}}}.
\end{equation}
One of the theorems in this paper is the following.

\begin{thmx}
\label{main-thm-1}
    The function $\mathfrak{X}: \The \to \Irrat$ given by
    \begin{align}
    \label{eqn:modified-ctd-frac}
        \mathfrak{X}( \seq{ (\varepsilon_n, \abar_n) }_{n\geq 1}) =  [\varepsilon_1 b_1, \varepsilon_2 b_2, \varepsilon_3 b_3,\ldots ]_- 
        \quad \text{ where } \quad
        b_n = \abar_n + \frac{1+\varepsilon_n \varepsilon_{n+1}}{2}
    \end{align}
    is a homeomorphism conjugating the standard shift map $\shift : \The \to \The$ with the map $\gauss: \Irrat \to \Irrat$.
\end{thmx}

This theorem is somewhat classical.
The proof relies on the study of Diophantine approximations related to the modified continued fraction above.
For $\sigma = \seq{ (\varepsilon_n, \abar_n) }_{n\geq 1} \in \The$ and $b_n$'s being as introduced in the theorem, we can write 
\[
    [ \varepsilon_1 b_1, \varepsilon_2 b_2, \ldots, \varepsilon_n b_n ]_-  = \frac{p_{[n]}}{q_{[n]}},
\] 
where $p_{[n]}=p_{[n]}(\sigma)$ and $q_{[n]}=q_{[n]}(\sigma)$ are co-prime integers and $q_{[n]} \geq 1$.
The $q_{[n]}$'s can be characterized in many other ways including the first return times of an iterated pre-renormalization (Proposition \ref{prop:q[n]-02}).
Theorem \ref{main-thm-1} follows from the arithmetic properties of these convergents.

In Section \ref{sec:continued-fraction}, we compare our approach with the regular continued fraction expansion (the standard $a_n$'s) and the nearest integer continued fraction expansion.
We establish a conversion from our expansion (\ref{eqn:modified-ctd-frac}) to the regular expansion, which will come in handy in subsequent applications.
The regular expansion is related to commuting pair renormalization in the dynamical theory of circle maps, whereas the nearest integer expansion is similar to (\ref{eqn:modified-ctd-frac}) and has also been used several times in complex dynamics; see \S\ref{ss:motivation} for a historical account.
Figure \ref{fig:diagram} summarizes their relationship.
Note that even though the $b_n$'s are more desirable in the number-theoretic point of view, they satisfy an intricate admissibility rule (depending on $\varepsilon_n \varepsilon_{n+1}$) and are not easily describable from the geometric procedure of sector renormalization.

\begin{figure}
    \centering
    \begin{tikzpicture}
        \node at (-4,0) {$\{a_n\}_{n\geq 1}$};
        \node at (1,0) {$\{(\varepsilon_n,\abar_n)\}_{n\geq 1}$};
        \node at (4,0) {$\{(\varepsilon_n,b_n)\}_{n\geq 1}$};
        \node at (2.5,0) {$\approx$};
        \node at (-4,-1.5) {$\{q_n\}_{n\geq 0}$};
        \node at (4,-1.5) {$\{q_{[n]}\}_{n\geq 0}$};
        \node at (-4,2) {commuting pair};
        \node at (-4,1.65) {renormalization};
        \node at (1,2) {sector};
        \node at (1,1.65) {renormalization};
        
        \draw[-latex] (-4,1.3) -- (-4,0.4);
        \draw[-latex] (-4,-0.4) -- (-4,-1.1);
        \draw[-latex] (1,1.3) -- (1,0.4);
        \draw[-latex] (4,-0.4) -- (4,-1.1);
        \draw[-latex] (-2.75,-1.5) -- (2.75,-1.5);
        \node at (0,-1.75) {\scalebox{0.85}{ignore \emph{negligible} levels}};
    \end{tikzpicture}
    \caption{Relationship between the different codings}
    \label{fig:diagram}
\end{figure}
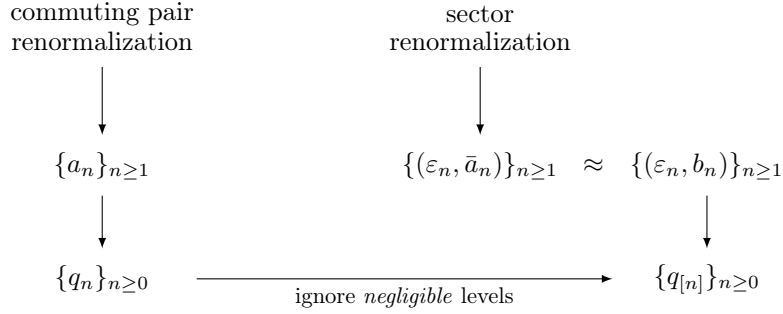

\subsection{The parabolic compactification}
\label{ss:intro-3}

Below are three embeddings of the space of irrationals $\Irrat$ into compact spaces that are of particular interest.
\begin{itemize}
    \item The modular embedding $\mu: \Irrat \to \T = \R/\Z$. \\
    This is the naive embedding given by $\mu(\theta) = \theta$ (mod $1$).
    \item The segment embedding $\eta: \Irrat \to [0,1]$. \\
    This is given by $\eta(\theta) = \theta$ if $\theta > 0$ and $\eta(\theta) = \theta+1$ if $\theta < 0$.
    It satisfies the following commutative diagram.
\begin{center}
\begin{tikzcd}
    \Irrat \arrow[rd, "\mu"'] \arrow[r, "\eta"] & {[0,1]} \arrow[d, "0 \sim 1"] \\
    & \T
\end{tikzcd}
\end{center}
    The map $\gauss$ naturally extends to a self-map of $\T$ that is continuous everywhere except at the origin, at which wild oscillations are produced on both the left and the right side.
    See Figure \ref{fig:gauss}.
    Unlike $\mu$, the embedding $\eta$ has the advantage of distinguishing $0^+$ and $0^-$.
    \item The dynamical embedding $\ord: \Irrat \to \Omega$. \\
    As the name suggests, this map has direct ties to the applications in neutral renormalization. (See \S\ref{ss:motivation}.)
    Geometrically, it encodes the cyclic order of forward orbits of successive pre-renormalizations of $\rotate_\theta$.
    Here is how it works.
    Given any two positive integers $k$ and $l$, let $\Omega_{k,l}$ denote the set of maps $f:\{0,1\ldots,k\}^{l+1} \to \T$ modulo post-composition with orientation-preserving self-homeomorphisms of $\T$.
    This set corresponds to the set of circular permutations of the blocks of a partition of $\{0,1,\ldots,k\}^{l+1}$ (coming from the fibers of $f$).
    The set $\Omega_{k,l}$ is finite and it will be equipped with the discrete topology.
    Consider the infinite product space
\[
    \Omega = \prod_{k, l \geq 1} \Omega_{k,l}.
\]
    Given an irrational $\theta \in \Irrat$, 
    we denote by $\ord_{k,l}(\theta) \in \Omega_{k,l}$ the equivalence class of the embedding 
\[
    (j_0,\ldots,j_{l}) \mapsto 
    \rotate_{\theta}^{j_0 q_{[0]}(\theta) + \ldots + j_l q_{[l]}(\theta)} (1).
\]
    The maps $\ord_{k,l}: \Irrat \to \Omega_{k,l}$ altogether produce the embedding
\[
    \ord: \Irrat \to \Omega, \qquad
    \ord(\theta) = \{\ord_{k,l}(\theta)\}_{k,l}.
\]
\end{itemize}

We define the \emph{parabolic compactification} of the space of irrationals to be the compact sequence space
    \[
        \TheCpt := \left\{ \seq{ (\varepsilon_n, \abar_n) }_{n\geq 1} \: : \: \varepsilon_n \in \{-1,+1\}, \abar_n \in \N_{\geq 2} \cup \{\infty\} \right\}.
    \]
Via the map $\mathfrak{X}$ from Theorem \ref{main-thm-1}, the shift map $\shift$ on $\TheCpt$ is a compactification of $(\Irrat, \gauss)$.
The three embeddings listed above extend to continuous maps 
\[
    \bar{\mu}: \TheCpt \to \T, \qquad
    \bar{\eta}: \TheCpt \to [0,1], \qquad
    \overline{\ord}: \TheCpt \to \Omega
\]
respectively.
We prove:

\begin{thmx}[Universal property of $\TheCpt$]
\label{main-thm-2}
    The compactification $(\TheCpt, \shift)$ of $(\Irrat,\gauss)$ is uniquely characterized by one of the following properties.
    \begin{enumerate}
        \item $(\TheCpt, \shift)$ is the smallest compactification of $(\Irrat,\gauss)$ that respects $\eta$, and it has the universal property that every other compactification of $(\Irrat,\gauss)$ respecting $\eta$ factors to $(\TheCpt, \shift)$.
        \item $(\TheCpt, \shift)$ is the smallest compactification of $(\Irrat,\gauss)$ that respects $\ord$, and it has the universal property that every other compactification of $(\Irrat,\gauss)$ respecting $\ord$ factors to $(\TheCpt, \shift)$.
    \end{enumerate}
    Consequently, every compactification of $(\Irrat, \gauss)$ that respects $\eta$ also respects $\ord$, and vice versa.
\end{thmx}

A more precise formulation can be found in \S\ref{ss:universality}.
In spirit, it is the dynamical opposite of Stone-\v{C}ech compactification in topology.
The minimality property fails for the modular embedding (Proposition \ref{prop:multiplier-failure}).

Every element $\ttheta$ of $\TheCpt$ induces an ordered additive abelian group called the \emph{continuant group} $(\Kont_{\ttheta}, <)$ of $\ttheta$ whose generators are a generalization of the denominator continuants $q_{[n]}$'s discussed above.
The positive cone $\Time_{\ttheta} := \{p \in \Kont_{\ttheta} \: : \: p>0 \}$, which we call the \emph{time semigroup}, will be important in applications; see \S\ref{sss:sector-renorm}.

\subsection{The continuant group and topological cascades}

In Section \ref{sec:natural-extension}, we study the natural extension $\shift: \TheCpt \to \TheCpt$, which is the bi-infinite shift space 
\[
\shift: \TheBiCpt \to \TheBiCpt
\]
obtained by taking the inverse limit.
We show that each element $\tttheta$ of $\TheBiCpt$ induces a group $\Kbold_{\tttheta}$ with a countably infinite number of generators $Q_{[n]}$, $n \in \Z$ mimicking the properties of denominator continuants.
It admits two natural orders, namely the \emph{time order} $<$ and the \emph{space order} $\lhd$.

The time order $<$ is characterized by $Q_{[n]} < Q_{[n+1]}$ for all $n$ and an additional natural condition when $\abar_{n+1} = \infty$ (cf. Proposition \ref{prop:chronological-order-02}).
The \emph{time semigroup} 
$\Tbold_{\tttheta} := \{P \in \Kbold_{\tttheta} \: : \: P>0 \}$
is useful in applications; see \S\ref{sss:sector-renorm}.

The space order $\lhd$ is characterized by $Q_{[n]} \lhd 0$ if $\varepsilon_{n+1} = +$ and $0 \lhd Q_{[n]}$ if $\varepsilon_{n+1} = -$, and an additional natural condition when $\abar_{n+1} = \infty$ (cf. Proposition \ref{prop:left-right-order} and Lemma \ref{lem:further-space-orderings}).
The latter leads to the following theorem.

\begin{thmx}
\label{main-thm-3}
    For every $\tttheta \in \TheBiCpt$, there exists an invertible dynamical system
    \[
        A_{\tttheta} = \{ A^P : \R \to \R \}_{P \in \Kbold_{\tttheta}}
    \]
    parametrized by the continuant group $\Kbold_{\tttheta}$ with the following properties.
    \begin{enumerate}
        \item Each map $A^P_{\tttheta}$ is a monotonically increasing self-homeomorphism of $\R$.
        \item The map $(\Kbold_{\tttheta}, \lhd) \to (\R,<)$, $P \mapsto A^P_{\tttheta}(0)$ is an order-preserving injective map with a dense image.
    \end{enumerate}
    Moreover,
    \begin{enumerate}[start=3]
        \item $A_{\tttheta}$ is unique up to topological conjugacy.
        \item If $\tttheta$ is irrational, then each $A^P_{\tttheta}$ can be chosen to be a translation.
        Such a group of translations $A_{\tttheta}$ is unique up to linear conjugacy.
    \end{enumerate}
\end{thmx}

Such a system $A_{\tttheta}$ is called an \emph{invertible topological cascade} with combinatorics $\tttheta$.
Relative to $<$, the action of $A_{\tttheta}$ on $\R$ is properly discontinuous and induces a dynamical tiling of $\R$ analogous to renormalization tilings from irrational rotations.

Throwing out $\infty$ gives us the shift-invariant subspace $\TheBi$ of $\TheBiCpt$.
Every element of $\TheBi$ induces a unique bi-infinite tower of sector renormalizations of irrational rotations.
In \S\ref{ss:cascade}, we describe how this tower can be fit into a single dynamical plane as a cascade of translations.
Translational cascades with periodic combinatorics first appeared in \cite[Section 2]{DL23}.

\subsection{Motivation and historical remarks}
\label{ss:motivation}

This paper provides the background on the combinatorics of neutral renormalization in holomorphic dynamics.

\subsubsection{On sector renormalization}
\label{sss:sector-renorm}

The use of Diophantine approximation in studying linearizability of neutral holomorphic germs dates back to the work of Siegel \cite{Si42}.
As previously mentioned, sector renormalization for neutral holomorphic germs was initially studied by Douady-Ghys, and Yoccoz.
Variants of sector renormalization for neutral quadratic polynomials $f_\theta(z) = e^{2\pi i \theta}z + z^2$ have been studied under some arithmetic conditions.
For bounded-type irrationals $\theta$, Dudko, Lyubich, and Selinger \cite{DLS,DL23} studied Siegel-Pacman renormalization to prove parameter universality and local connectivity of the Mandelbrot set near Siegel parameters.
For high-type irrationals $\theta$, Inou and Shishikura studied near-parabolic renormalization \cite{IS} which resulted in a number of significant applications, e.g. \cite{BC12, CC15, Che13, Che19, SY24, Che25}. 
A variant of the modified continued fraction expansion presented in Theorem \ref{main-thm-1} has appeared a number of times in holomorphic dynamics, e.g. \cite{Yoc95,IS,CC15,AC17,Che25}; see \S\ref{ss:nearest-integer}.

In \cite{DL26}, Dudko and Lyubich proved that sector renormalizations of quadratic polynomials $f_\theta(z) = e^{2\pi i \theta}z + z^2$ with irrational rotation numbers $\theta$ can be constructed in such a way that the resulting maps $\renorm_{\textnormal{sec}}^n f_{\theta}$ are pre-compact and capture the critical orbit well.
Their result is based on a prior work \cite{DL22} on \emph{pseudo-Siegel disks} where the regular expansion $\{a_n\}_{n\geq 1}$ was used.
In \cite{DL26}, rotation numbers are relabeled using $\{ (\varepsilon_n, \abar_n) \}_{n\geq 1}$ to ensure compatibility with sector renormalization (see Figure \ref{fig:diagram}).
Exploring the subtleties between the two symbolic codings is one of the motivations for this very paper.

Every element in the closure of the space of forward renormalization orbits $\{ \seq{ \renorm_{\textnormal{sec}}^{n+m} f_{\theta} }_{n \geq 0} \: : m \geq 0, \theta \in \Irrat\}$
has combinatorics that can be uniquely encoded by an element $\ttheta = \seq{ (\varepsilon_n, \abar_n) }_{n\geq 1}$ of the compactification $\TheCpt$.
It also comes with an associated full Lavaurs-Epstein renormalization tower that can be parametrized by the time semigroup $\Time_{\ttheta}$ of $\ttheta$.
Any term $\abar_{n+1} = \infty$ indicates the occurrence of simple parabolic bifurcation at renormalization level $n$.
The extension $\bar{\eta}$ reflects the necessity of distinguishing top versus bottom Lavaurs parabolic enrichments, while the fiber structure of $\bar{\mu}$ reflects the necessity of performing parabolic blow-ups on the Mandelbrot set to define a geometric topology where Julia sets are continuous. 
Meanwhile, the extension $\overline{\text{ord}}$ ensures that the cyclic arrangement of critical orbits under successive pre-renormalizations varies continuously. 
Together, these properties provide the framework necessary for further applications, including the uniform continuity of \emph{Mother Hedgehogs} \cite{DL26}, the zero area of the postcritical set \cite{Lim26}, and the rigidity of the attractor of renormalization of neutral quadratic polynomials \cite{DLL}.

Let us elaborate on the last point above.
Elements of the attractor are bi-infinite towers $\seq{ f_n }_{n\in\Z}$ of sector renormalization and their combinatorial data are elements of the natural extension $\TheBiCpt$. 
With Dudko and Lyubich, we prove:

\begin{theorem}[\cite{DLL}]
\label{thm:tower-rigidity}
    The full renormalization attractor for neutral quadratic polynomials is combinatorially rigid. In particular, modulo conformal conjugacy, the renormalization attractor is conjugate to the shift map $\shift: \TheBiCpt \to \TheBiCpt$.
\end{theorem}

In the proof, we bring a bi-infinite tower of renormalizations $\seq{ f_n }_{n\in\Z}$ to a single dynamical plane of a semigroup $\mathbf{F}$ of global transcendental holomorphic maps onto $\C$ called a \emph{neutral cascade} which can parametrized by the time semigroup $\Tbold_{\tttheta}$ of $\tttheta$.
(For periodic combinatorics, this cascade was studied in \cite{McM98,DL23}.)
On the exterior of its Mother Hedgehog, we pass to external coordinates on which the neutral cascade acts on the upper half-plane and extends continuously to a topological cascade (c.f. Theorem \ref{thm:topological-cascade}).
We then establish various geometric bounds in external coordinates to construct an affine conjugacy between any two combinatorially equivalent neutral cascades.
Under some arithmetic condition (called the Brjuno condition), $\mathbf{F}$ admits an invariant topological half-plane (called the Siegel set) in which $\mathbf{F}$ is analytically conjugate to a cascade of translations (c.f. \S\ref{ss:cascade}).

\subsubsection{On commuting pair renormalization}
\label{sss:circle-map}

Lastly, let us comment on the connection between the renormalization theory of neutral quadratic polynomials and that of real-analytic (uni-)critical circle maps with irrational rotation number.

In the theory of critical circle maps, renormalization is defined in terms of commuting pairs supported on real intervals and the combinatorics can be described in terms of the regular continued fraction expansion $\{a_n\}_{n\geq 1}$ together with the classical Gauss map $G(\theta) = \{ \frac{1}{\theta} \}$.
It was initially observed by Lanford \cite{Lan88} that the study of commuting pair renormalization would explain the rigidity and universality phenomenon for critical circle maps. 
This has now been justified by de Faria, de Melo, and Yampolsky \cite{dF99,dFdM2,Y99,Y01,Y02,Y03}.

The analog of Theorem \ref{thm:tower-rigidity} was proven by Yampolsky \cite{Y01} where the corresponding renormalization horseshoe was shown to be conjugate to the shift map on $\{1,2,3,\ldots,\infty\}^{\Z}$.
The term $\infty$ again encodes the degeneration coming from simple parabolic bifurcation, but, unlike sector renormalization, real-symmetry prevents any possible ambiguity in orientation.
Every element of the attractor also admits the structure of a cascade of real-symmetric $\sigma$-proper maps onto $\C$.
(For periodic combinatorics, this cascade was studied in \cite{Lim24}.)
In \cite{Y02,Y03}, the commuting pair renormalization was upgraded to cylinder renormalization, a gluing operation that is geometrically similar to sector renormalization.
Within this new framework of cylinder renormalization, Yampolsky \cite{Y02,Y03} successfully brought Lanford's Program to completion.

\subsection{Acknowledgments}

I would like to thank Dzmitry Dudko and Mikhail Lyubich for numerous discussions on neutral renormalization.
A significant part of this paper was written during my visit to the Banach Center during the Simons Semester on Continued Fractions, Fractals, Ergodic theory and Dynamics in May 2026. 
This work was partially supported by the Simons Foundation grant (award no. SFI-MPS-T-Institutes-00010825) and from State Treasury funds as part of a task commissioned by the Minister of Science and Higher Education under the project ``Organization of the Simons Semesters at the Banach Center -- New Energies in 2026--2028'' (agreement no. MNiSW/2025/DAP/491).

\section{Renormalization of irrational rotations}

\subsection{Sector renormalization}

For any irrational $\theta \in \Irrat$, denote
\[
    \varepsilon(\theta) := \textnormal{sgn}(\theta) \in \{-1,+1\} \qquad \text{ and } \qquad 
    \abar(\theta) := \left\lfloor \frac{1}{|\theta|} \right\rfloor \in \{2,3,4,\ldots\}.
\]
The sign $\varepsilon=\varepsilon(\theta)$ tells us whether or not $\rotate_\theta$ rotates clockwise or anticlockwise.
The number $\abar=\abar(\theta)$ is the unique positive integer such that
\[
    \abar|\theta| < 1 < (\abar+1)|\theta|.
\]

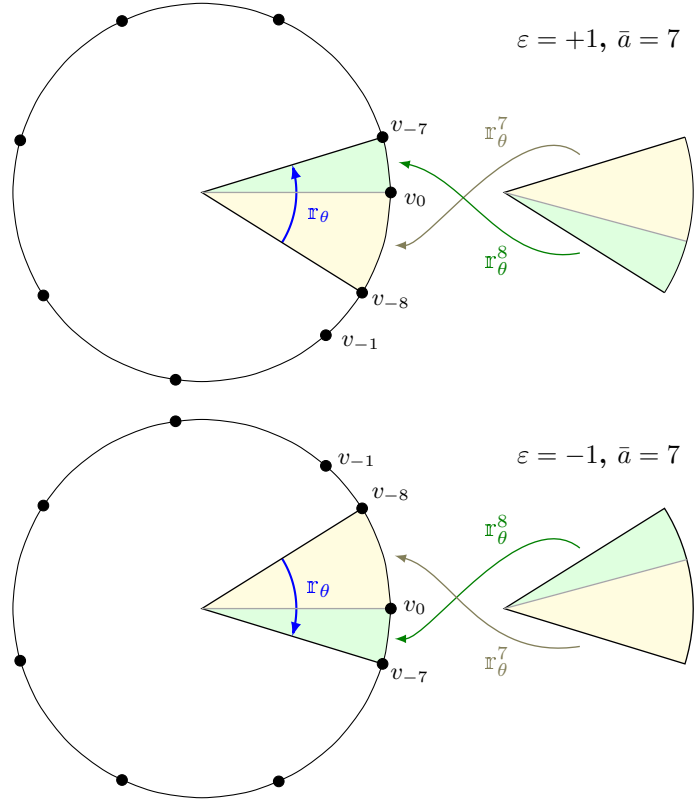
\begin{figure}    
    \centering
    \begin{tikzpicture}[scale=1]
        \filldraw[white,fill opacity=0.5,fill=green!25!white] (-3.5,0) -- (-1,0) arc (0:17:2.5cm) -- cycle;        
        \filldraw[white,fill opacity=0.5,fill=yellow!30!white] (-3.5,0) -- ({2.5*cos(-32)-3.5},{2.5*sin(-32)}) arc (-32:0:2.5cm) -- cycle;
        \draw [black,domain=0:360,smooth] plot ({2.5*cos(\x)-3.5}, {2.5*sin(\x)});
        \draw [blue,thick,domain=-32:17,-latex] plot ({1.25*cos(\x)-3.5}, {1.25*sin(\x)});
        \draw[black,line width=0.5pt] ({2.5*cos(-32)-3.5},{2.5*sin(-32)}) -- (-3.5,0) -- ({2.5*cos(17)-3.5},{2.5*sin(17)}); 
        \draw[gray!70!white,line width=0.5pt] (-3.5,0) -- (-1,0);

        \filldraw[white,fill opacity=0.5,fill=yellow!30!white] (0.5,0) -- ({2.5*cos(-15)+0.5},{2.5*sin(-15)}) arc (-15:17:2.5cm) -- cycle;   
        \filldraw[white,fill opacity=0.5,fill=green!25!white] (0.5,0) -- ({2.5*cos(-32)+0.5},{2.5*sin(-32)}) arc (-32:-15:2.5cm) -- cycle;
        \draw[black,line width=0.5pt] ({2.5*cos(17)+0.5},{2.5*sin(17)}) -- (0.5,0) -- ({2.5*cos(-32)+0.5},{2.5*sin(-32)});        
        \draw[gray!70!white,line width=0.5pt] ({2.5*cos(-15)+0.5},{2.5*sin(-15)}) -- (0.5,0);
        \draw [black,domain=-32:17] plot ({2.5*cos(\x)+0.5}, {2.5*sin(\x)});

        \filldraw (-1,0) circle (2pt);
        \filldraw ({2.5*cos(-49)-3.5},{2.5*sin(-49)}) circle (2pt);
        \filldraw ({2.5*cos(-98)-3.5},{2.5*sin(-98)}) circle (2pt);
        \filldraw ({2.5*cos(-147)-3.5},{2.5*sin(-147)}) circle (2pt);
        \filldraw ({2.5*cos(-196)-3.5},{2.5*sin(-196)}) circle (2pt);
        \filldraw ({2.5*cos(-245)-3.5},{2.5*sin(-245)}) circle (2pt);
        \filldraw ({2.5*cos(-294)-3.5},{2.5*sin(-294)}) circle (2pt);
        \filldraw ({2.5*cos(-343)-3.5},{2.5*sin(-343)}) circle (2pt);
        \filldraw ({2.5*cos(-392)-3.5},{2.5*sin(-392)}) circle (2pt);

        \node [black, font=\bfseries] at (1.75,2) {$\varepsilon =+1$, $\abar= 7$};
        \node [blue, font=\bfseries] at (-1.95,-0.3) {\small $\rotate_{\theta}$};
        \draw[yellow!40!black,line width=0.5pt,-latex] (1.5,0.5).. controls (0.75,1.1) and (-0.5,-0.7) .. (-0.95,-0.7) ;
        \draw[green!50!black,line width=0.5pt,-latex] (1.5,-0.8) .. controls (0.5,-1) and (0,0.2) .. (-0.9,0.4);
        \node [green!50!black, font=\bfseries] at (0.4,-0.9) {$\rotate_{\theta}^8$};
        \node [yellow!40!black, font=\bfseries] at (0.4,0.8) {$\rotate_{\theta}^7$};
        \node [black, font=\bfseries] at (-0.67,-0.1) {\small $v_{0}$};
        \node [black, font=\bfseries] at (-1.4,-2) {\small $v_{-1}$};
        \node [black, font=\bfseries] at (-0.75,0.85) {\small $v_{-7}$};
        \node [black, font=\bfseries] at (-1,-1.45) {\small $v_{-8}$};

        \filldraw[white,fill opacity=0.5,fill=green!25!white] (-3.5,-5.5) -- (-1,-5.5) arc (0:-17:2.5cm) -- cycle;        
        \filldraw[white,fill opacity=0.5,fill=yellow!30!white] (-3.5,-5.5) -- (-1,-5.5) arc (0:32:2.5cm) -- cycle;
        \draw [black,domain=0:360,smooth] plot ({2.5*cos(\x)-3.5}, {2.5*sin(\x)-5.5});        
        \draw [blue,thick,domain=32:-17,-latex] plot ({1.25*cos(\x)-3.5}, {1.25*sin(\x)-5.5});
        \draw[black,line width=0.5pt] ({2.5*cos(32)-3.5},{2.5*sin(32)-5.5}) -- (-3.5,-5.5) -- ({2.5*cos(-17)-3.5},{2.5*sin(-17)-5.5});        
        \draw[gray!70!white,line width=0.5pt] ({2.5*cos(0)-3.5},{2.5*sin(0)-5.5}) -- (-3.5,-5.5);

        \filldraw[white,fill opacity=0.5,fill=yellow!30!white] (0.5,-5.5) -- ({2.5*cos(-17)+0.5},{2.5*sin(-17)-5.5}) arc (-17:15:2.5cm) -- cycle;        
        \filldraw[white,fill opacity=0.5,fill=green!25!white] (0.5,-5.5) -- ({2.5*cos(15)+0.5},{2.5*sin(15)-5.5}) arc (15:32:2.5cm) -- cycle;
        \draw[black,line width=0.5pt] (0.5,-5.5) -- ({2.5*cos(-17)+0.5},{2.5*sin(-17)-5.5}) arc (-17:32:2.5cm) -- cycle;        
        \draw[gray!70!white,line width=0.5pt] ({2.5*cos(15)+0.5},{2.5*sin(15)-5.5}) -- (0.5,-5.5);
        
        \filldraw (-1,-5.5) circle (2pt);
        \filldraw ({2.5*cos(49)-3.5},{2.5*sin(49)-5.5}) circle (2pt);
        \filldraw ({2.5*cos(98)-3.5},{2.5*sin(98)-5.5}) circle (2pt);
        \filldraw ({2.5*cos(147)-3.5},{2.5*sin(147)-5.5}) circle (2pt);
        \filldraw ({2.5*cos(196)-3.5},{2.5*sin(196)-5.5}) circle (2pt);
        \filldraw ({2.5*cos(245)-3.5},{2.5*sin(245)-5.5}) circle (2pt);
        \filldraw ({2.5*cos(294)-3.5},{2.5*sin(294)-5.5}) circle (2pt);
        \filldraw ({2.5*cos(343)-3.5},{2.5*sin(343)-5.5}) circle (2pt);
        \filldraw ({2.5*cos(392)-3.5},{2.5*sin(392)-5.5}) circle (2pt);

        \node [black, font=\bfseries] at (1.75,-3.5) {$\varepsilon=-1$, $\abar = 7$};
        \node [blue, font=\bfseries] at (-1.95,-5.25) {$\rotate_{\theta}$};
        \draw[green!50!black,line width=0.5pt,-latex] (1.5,-4.7).. controls (0.75,-4.1) and (-0.5,-5.9) .. (-0.95,-5.9) ;
        \draw[yellow!40!black,line width=0.5pt,-latex] (1.5,-6) .. controls (0.2,-6.2) and (-0.2,-5) .. (-0.95,-4.8);
        \node [green!50!black, font=\bfseries] at (0.4,-4.45) {$\rotate_{\theta}^8$};
        \node [yellow!40!black, font=\bfseries] at (0.4,-6.2) {$\rotate_{\theta}^7$};
        \node [black, font=\bfseries] at (-0.68,-5.48) {\small $v_{0}$};
        \node [black, font=\bfseries] at (-1.45,-3.55) {\small $v_{-1}$};
        \node [black, font=\bfseries] at (-0.75,-6.4) {\small $v_{-7}$};
        \node [black, font=\bfseries] at (-1,-4) {\small $v_{-8}$};
\end{tikzpicture}
    \caption{Two examples of the first return map on the shaded sector $S_{\theta}$.}
    \label{fig:sector-rotation}
\end{figure}

Let us first fix $\theta \in \Irrat$.
For $k \in \Z$, denote 
\[
v_k = v_k(\theta) := \rotate_\theta^k(1).
\]
For any two distinct integers $k,l \in \Z$, we denote by $\triangle_{\theta}(k,l)$ the closed radial sector bounded by the radial line segment from $0$ to $v_k$, the radial line segment from $0$ to $v_l$, and the shortest circular arc in $\partial \D$ joining $v_k$ and $v_l$.
We will in particular consider the sector
\[
    S_\theta := \triangle_{\theta}(-\abar-1,-\abar).
\]
It contains the point $v_0 = 1$. The first return map of $\rotate_{\theta}$ back to $S_\theta$ is the piecewise linear map
\[
    \textnormal{FRM}_{\theta}(z) = \begin{cases}
        \rotate_{\theta}^{\abar+1}(z) & \text{ if } z \in \triangle_{\theta}(-\abar-1,-2\abar-1) , \\
        \rotate_{\theta}^{\abar}(z) & \text{ if } z \in \triangle_{\theta}(-2\abar-1,-\abar) .
    \end{cases}
\]
See Figure \ref{fig:sector-rotation} for an illustration.
The power map 
\[
    \psi_{\theta}: S_\theta \to \overline{\D}, \quad \psi_\theta(z) = z^{1/|\theta|}
\]
sends the interior of $S_\theta$ conformally onto the unit disk minus the radial slit $\gamma_\theta = \{\arg z = |\theta|^{-1} \arg v_{-\abar} \}$, and it glues the two radial edges of $S_\theta$ together onto the slit.
Under $\psi_\theta$, $\textnormal{FRM}_{\theta}$ projects to a new rigid rotation $\rotate_{\gauss(\theta)} : \overline{\D} \to \overline{\D}$ for some irrational $\gauss(\theta) \in \Irrat$. The slit $\gamma_\theta$ is actually equal to $\{ \arg z = - 2\pi\, \gauss(\theta) \}$.
The map
\[
    \gauss: \Irrat \to \Irrat
\]
should be thought as a modified Gauss map.
We call $\rotate_{\gauss(\theta)}$ the \emph{sector renormalization} of $\rotate_\theta$.

\begin{lemma}
\label{lem:r-shift}
    For $\theta \in \Irrat$, $\gauss(\theta)$ is the unique irrational in $\Irrat$ such that 
    \[
    \gauss(\theta) \equiv -\frac{1}{\theta} \quad (\text{mod }1).
    \]
    More precisely, we have
    \[
        \gauss(\theta) = - \frac{1}{\theta} + \varepsilon(\theta) \,b(\theta) \quad \text{ where }
        b(\theta) = \abar(\theta) + \frac{1+ \varepsilon(\theta) \, \varepsilon(\gauss(\theta))}{2}.
    \]
    The map $\gauss: \Irrat \to \Irrat$ is an infinite-to-one surjective continuous function.
\end{lemma}

The graph of $\gauss$ is illustrated in Figure \ref{fig:gauss}.

\begin{figure}
    \centering
    \includegraphics[width=0.5\linewidth]{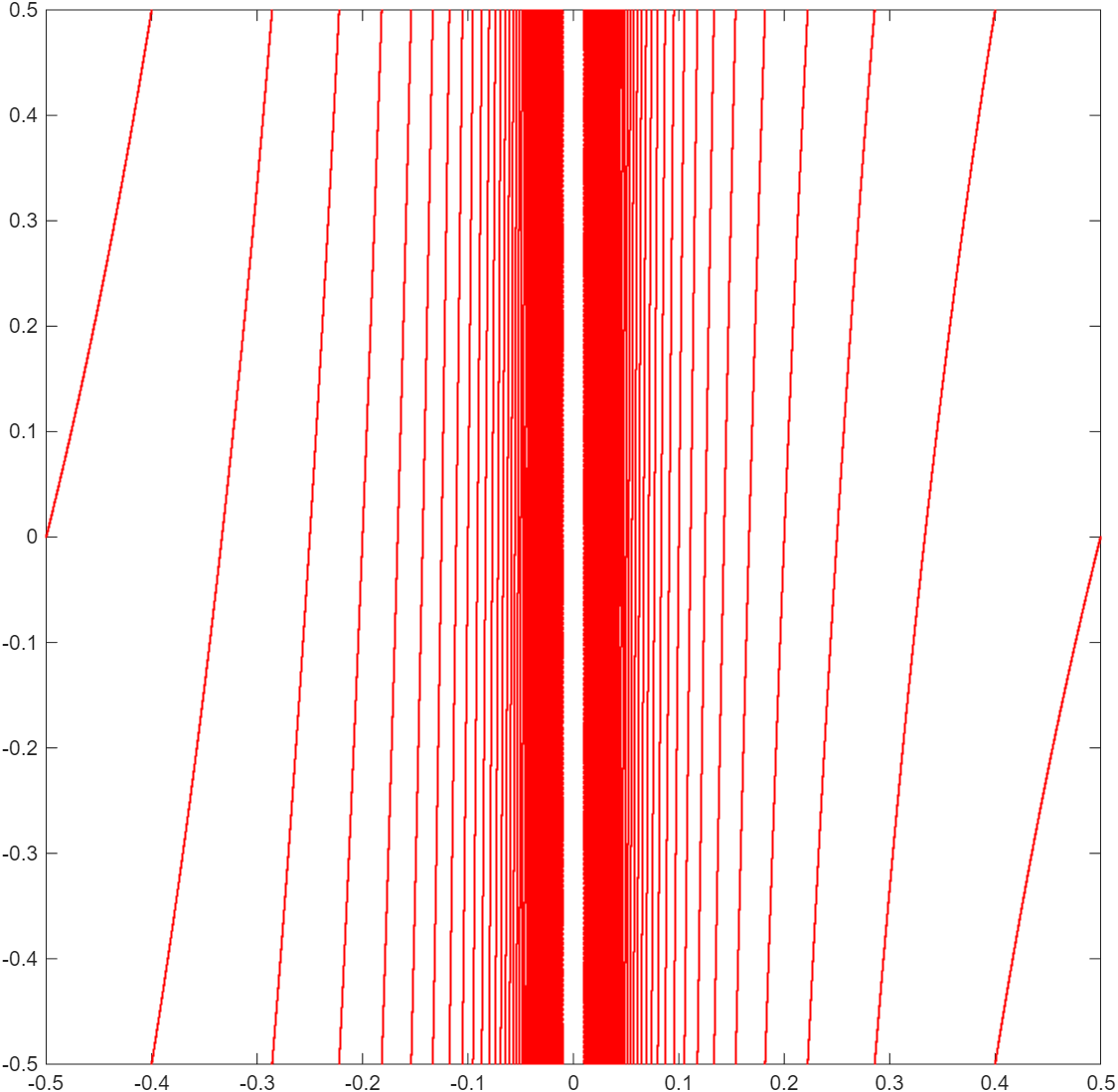}
    \caption{The graph of $\gauss$}
    \label{fig:gauss}
\end{figure}

\begin{proof}
    Denote $\varepsilon_1 = \varepsilon(\theta)$, $\varepsilon_2 = \varepsilon(\gauss(\theta))$, and $\abar=\abar(\theta)$. 
    The first return map back to $S_\theta$ can be written as the pair of maps $\rotate_{\abar\theta-\varepsilon_1}$ and $\rotate_{(\abar+1)\theta-\varepsilon_1}$ where both $\abar\theta - \varepsilon_1$ and $(\abar+1)\theta - \varepsilon_1$ are contained in $(-\frac{1}{2},\frac{1}{2})$.
    Therefore, for $z \in \overline{\D}$, there is some $k=k(z) \in \{\abar,\abar+1\}$ such that
    \begin{align*}
        \rotate_{\gauss(\theta)}(z) &= \psi_\theta \circ \rotate_{k \theta-\varepsilon_1} \circ \psi_\theta^{-1}(z) = e^{2\pi i(k\theta - \varepsilon_1)/|\theta|} z = e^{2\pi i (\varepsilon_1 k - \frac{1}{\theta})} z,
    \end{align*}
    so then $\gauss(\theta) \equiv - \frac{1}{\theta}$ mod $1$.
    
    Let us write $\gauss(\theta) = -\frac{1}{\theta} + x$ where $x=x(\theta)$ is the unique integer closest to $\frac{1}{\theta}$.
    Since $\abar < \frac{\varepsilon_1}{\theta} < \abar+1$, then either $x = \varepsilon_1 \abar$ or $x = \varepsilon_1 (\abar+1)$.
    If $x = \varepsilon_1 \abar$, then
    \[
        \varepsilon_1 \gauss(\theta) = -\frac{\varepsilon_1}{\theta} + \varepsilon_1x = - \frac{\varepsilon_1}{\theta} + \abar < 0.
    \]
    which implies that $\varepsilon_1 \varepsilon_2 = -1$. 
    Else, if $x= \varepsilon_1 (\abar+1)$, then
    \[
        \varepsilon_1 \gauss(\theta) = -\frac{\varepsilon_1}{\theta}+  \varepsilon_1x =  -\frac{\varepsilon_1}{\theta} + \abar + 1 > 0
    \]
    and so $\varepsilon_1 \varepsilon_2 = +1$.
    In both cases, we do have $x = \varepsilon_1 b(\theta) $.

    The function $\gauss$ is continuous since $-\frac{1}{\theta}$ and $x(\theta)$ are continuous in $\theta \in \Irrat$. It is clear that $\gauss$ is infinite-to-one and surjective.
\end{proof}

\subsection{More notation}
\label{ss:more-notation}

Denote $\N = \{1,2,3\ldots\}$ and $\Sigma := \{-1,+1\} \times \N_{\geq 2}$. Consider the infinite sequence space
\[
    \The = \{ \seq{ (\varepsilon_n,\abar_n) }_{n\geq 1} \: : \: (\varepsilon_n, \abar_n) \in \Sigma \text{ for all } n\geq 1\},
\]
equipped with the infinite product topology.
For every $\sigma = \seq{ (\varepsilon_n, \abar_n ) }_{n\geq 1} \in \The$, we denote
\[
b_n = b_n(\sigma) := \abar_n + \frac{1 + \varepsilon_n \varepsilon_{n+1}}{2} \qquad \text{ for } n \geq 1.
\]
Associated to $\sigma$ are the pair of sequences $\{p_{[n]} = p_{[n]}(\sigma)\}_{n \geq 0}$ and $\{q_{[n]}=q_{[n]}(\sigma)\}_{n\geq 0}$ where
\[
    p_{[0]} = 0, \quad q_{[0]} = 1, \qquad p_{[1]} = \varepsilon_1, \quad q_{[1]} = b_1, 
\]
and recursively for $n\geq 2$,
\begin{equation}
\label{eqn:recurrence-relation}
    p_{[n]} = b_n p_{[n-1]} - \varepsilon_{n-1} \varepsilon_n  p_{[n-2]} \qquad \text{and} \qquad q_{[n]} = b_n q_{[n-1]} - \varepsilon_{n-1} \varepsilon_n  q_{[n-2]}.
\end{equation}

We have a well-defined function
\[
    \mathfrak{Y} : \Irrat \to \The, \quad \theta \mapsto \seq{ (\varepsilon(\gauss^{n-1}(\theta),\abar(\gauss^{n-1}(\theta)) }_{n\geq 1}.
\]
Throughout the rest of this section, we will fix an irrational number $\theta \in \Irrat$. 
We denote 
\[
    \seq{ (\varepsilon_n, \abar_n) }_{n\geq 1} = \mathfrak{Y}(\theta) \qquad \text{ and } \qquad b_n = b_n(\mathfrak{Y}(\theta)) \:\text{ for all } n \geq 1.
\]
For $n \geq 0$, we will denote
\[
    p_{[n]} \equiv p_{[n]}(\theta) = p_{[n]}(\mathfrak{Y}(\theta)), \quad
    q_{[n]} \equiv q_{[n]}(\theta) = q_{[n]}(\mathfrak{Y}(\theta)), \quad  
    \rotate_\theta^{[n]} \equiv \rotate_\theta^{q_{[n]}},
\]
and
\[
\theta_n = \gauss^n(\theta).
\]
By Lemma \ref{lem:r-shift}, we have
\begin{equation}
    \label{eqn:b-induction}
        \theta_{n} = - \frac{1}{\theta_{n-1}} + \varepsilon_{n} b_{n}
        \qquad \text{ for all } n \geq 1.
\end{equation}

\subsection{Iterated sector renormalizations}

For any $m \geq 0$ and $n \in \Z$, denote
\[
    S_m^1 = S_{\theta_m}, \quad \psi_{\theta_m} = \psi_m, \quad \text{and} \quad I^n_m := \{ z\in \overline{\D} \: : \: \arg z = \arg \rotate_{\theta_m}^n(1) \}.
\]
Observe that the radial slit $\gamma_{\theta_m}$ discussed previously is equal to $I^{-1}_{m+1}$ and is always disjoint from the sector $S_{m+1,1}$.
Hence, the sector $S_{m+1}^1$ can be lifted under the gluing map $\psi_m$ to a sector $S_{m}^2$ contained in $S^1_m$.
By repeating this process, we obtain the nest of sectors 
\[
S_{m}^1 \supset S_{m}^2 \supset S_{m}^3 \supset S_{m}^4 \supset \ldots
\]
where
\[
    S_{m}^n := \psi_{\theta_m}^{-1} (S_{m+1}^{n-1})
    \quad \text{ for all } m \geq 0 \text{ and } n\geq 2.
\]
For every $n \geq 1$, the angular width of $S_{0}^n$ is equal to
\[
    l_{[n-1]} := |\theta_0 \theta_1 \ldots \theta_{n-1}|.
\]
For convenience, we will also set $l_{[-1]} = 1$.

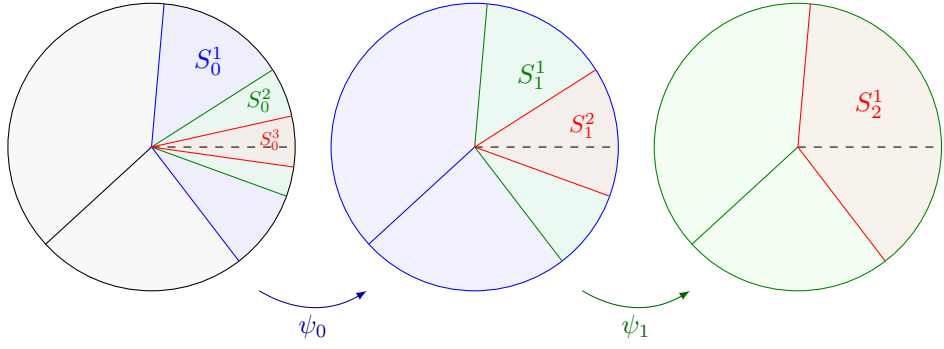
\begin{figure}
    \centering
    \begin{tikzpicture}[scale=0.95]
        \filldraw[white,fill opacity=0.5,fill=green!10!white] (4.5,0) circle (2cm);
        \filldraw[white,fill opacity=0.5,fill=red!10!white] ({4.5+2*cos(-52.5)}, {2*sin(-52.5)}) arc (-52.5:85:2cm) -- (4.5,0) -- cycle;
        \draw [green!50!black] (4.5,0) circle (2cm);
        \draw [green!50!black] ({4.5 + 2*cos(-137.5)}, {2*sin(-137.5)}) -- (4.5,0);
        \draw [black,dashed] (4.5,0) -- (6.5,0);
        \draw [red] ({4.5+2*cos(-52.5)}, {2*sin(-52.5)}) -- (4.5,0) -- ({4.5+2*cos(85)}, {2*sin(85)});
        
        \filldraw[white,fill opacity=0.5,fill=blue!10!white] (0,0) circle (2cm);
        \filldraw[white,fill opacity=0.5,fill=green!10!white] ({2*cos(-52.5)}, {2*sin(-52.5)}) arc (-52.5:85:2cm) -- (0,0) -- cycle;
        \filldraw[white,fill opacity=0.5,fill=red!10!white] ({2*cos(-20)}, {2*sin(-20)}) arc (-20:32.5:2cm) -- (0,0) -- cycle;
        \draw [blue] (0,0) circle (2cm);
        \draw [blue] ({2*cos(-137.5)}, {2*sin(-137.5)}) -- (0,0);
        \draw [black,dashed] (0,0) -- (2,0);
        \draw [green!50!black] ({2*cos(-52.5)}, {2*sin(-52.5)}) -- (0,0) -- ({2*cos(85)}, {2*sin(85)});
        \draw [red] ({2*cos(-20)}, {2*sin(-20)}) -- (0,0) -- ({2*cos(32.5)}, {2*sin(32.5)});
        
        \filldraw[white,fill opacity=0.5,fill=gray!10!white] (-4.5,0) circle (2cm);
        \filldraw[white,fill opacity=0.5,fill=blue!10!white] ({-4.5+2*cos(-52.5)}, {2*sin(-52.5)}) arc (-52.5:85:2cm) -- (-4.5,0) -- cycle;
        \filldraw[white,fill opacity=0.5,fill=green!10!white] ({-4.5+2*cos(-20)}, {2*sin(-20)}) arc (-20:32.5:2cm) -- (-4.5,0) -- cycle;
        \filldraw[white,fill opacity=0.5,fill=red!10!white] ({-4.5+2*cos(-8)}, {2*sin(-8)}) arc (-8:12.5:2cm) -- (-4.5,0) -- cycle;
        \draw [black] (-4.5,0) circle (2cm);
        \draw [black] ({-4.5 + 2*cos(222.5)}, {2*sin(222.5)}) -- (-4.5,0);
        \draw [black,dashed] (-4.5,0) -- (-2.5,0);
        \draw [blue] ({-4.5+2*cos(-52.5)}, {2*sin(-52.5)}) -- (-4.5,0) -- ({-4.5+2*cos(85)}, {2*sin(85)});
        \draw [green!50!black] ({-4.5+2*cos(-20)}, {2*sin(-20)}) -- (-4.5,0) -- ({-4.5+2*cos(32.5)}, {2*sin(32.5)});
        \draw [red] ({-4.5+2*cos(-8)}, {2*sin(-8)}) -- (-4.5,0) -- ({-4.5+2*cos(12.5)}, {2*sin(12.5)});
    
        \node [blue] at (-3.7,1.2) {$S_0^1$};
        \node [green!50!black] at (-3,0.65) {\scalebox{0.85}{$S_0^2$}};
        \node [red] at (-2.85,0.1) {\scalebox{0.7}{$S_0^3$}};
        \node [green!50!black] at (0.8,1) {$S_1^1$};
        \node [red] at (1.5,0.3) {\scalebox{0.9}{$S_1^2$}};
        \node [red] at (5.5,0.6) {$S_2^1$};
        \draw [blue!60!black,-latex] (-3,-2) .. controls (-2.5,-2.3) and (-2,-2.3) .. (-1.5,-2);
        \node [blue!60!black] at (-2.25,-2.5) {$\psi_0$};
        \draw [green!40!black,-latex] (1.5,-2) .. controls (2,-2.3) and (2.5,-2.3) .. (3,-2);
        \node [green!40!black] at (2.25,-2.5) {$\psi_1$};
    \end{tikzpicture}
    \caption{A nest of renormalization sectors}
    \label{fig:nesting}
\end{figure}

\begin{proposition}
\label{prop:ln}
    For all $n\geq 1$,
    \begin{enumerate}[label=\textnormal{(\arabic*}$_n$\textnormal{)}]
        \item $l_{[n-2]} = b_{n} l_{[n-1]} - \varepsilon_{n} \varepsilon_{n+1} l_{[n]}$, \vspace{0.03in}
        
        \item $ 1 = q_{[n]} l_{[n-1]} - \varepsilon_n \varepsilon_{n+1} q_{[n-1]} l_{[n]}$, \vspace{0.03in}
        
        \item $\varepsilon_{n+1} l_{[n]}= q_{[n]} \theta - p_{[n]}$.
    \end{enumerate}
\end{proposition}

\begin{proof}
    We can rewrite (\ref{eqn:b-induction}) as
    \begin{equation}
        \varepsilon_n = b_n \theta_{n-1} - \varepsilon_{n+1} |\theta_{n-1} \theta_n|.
    \end{equation}
    Multiplying this equation by $\varepsilon_n l_{[n-2]}$ gives us (1$_n$).
    Observe that when $n=1$, the equation above gives us (2$_1$) and (3$_1$) as well.
    For $n\geq 2$, (2$_n$) follows directly from (2$_{n-1}$) and the recurrence relation (\ref{eqn:recurrence-relation}).
    Observe also that (3$_0$) holds when the left hand side is to be taken to be $\theta$.
    For $n\geq 2$, (3$_n$) follows from (1$_n$), (3$_{n-1}$), (3$_{n-2}$), and the recurrence relations (\ref{eqn:recurrence-relation}) as follows:
    \begin{align*}
        \varepsilon_{n+1} l_{[n]} 
        &= b_n \varepsilon_{n} l_{[n-1]} - \varepsilon_n l_{[n-2]} \\
        &= b_n ( q_{[n-1]} \theta - p_{[n-1]}) - \varepsilon_{n-1} \varepsilon_n (q_{[n-2]} \theta - p_{[n-2]}) \\
        &= q_{[n]} \theta - p_{[n]}. \qedhere
    \end{align*}
\end{proof}

These $q_{[n]}$'s are known as the first return times of $\theta$.

\begin{proposition}[Dynamical characterization of $q_{[n]}$'s]
\label{prop:q[n]-02}
    For all $n \geq 1$,
\begin{enumerate}
    \item $\rotate_\theta^{[n]}$ is the rigid rotation by angle $\varepsilon_{n+1} l_{[n]} = |\theta_0\ldots\theta_{n-1}|\theta_n$;
    \item the sector $S_{0}^n$ is bounded by radial segments $I^{-q_{[n]} + \varepsilon_n \varepsilon_{n+1} q_{[n-1]}}_0$ and $I^{-q_{[n]}}_0$;
    \item $q_{[n]}$ is the unique smallest positive integer that satisfies 
    \[
    |[1, \rotate_\theta^{[n]}(1)]| \leq \frac{1}{2} |[1, \rotate_\theta^{[n-1]}(1)]|
    \]
    where $|\cdot|$ refers to the normalized angular measure on the unit circle.
\end{enumerate}
\end{proposition}

In neutral renormalization theory in holomorphic dynamics, the inequality in (3) gives uniform \emph{improvement of domain}.

We will take a dynamical approach to prove the proposition.
We will make use of a nest of dynamical triangulations $\{\Delta_m^n\}_{n \geq 0}$ of the closed unit disk for all $m \geq 0$. 
Each $\Delta_m^n$ is a collection of closed sectors in $\overline{\D}$ with vertex at $0$ such that their union is $\overline{\D}$ and that they have pairwise disjoint interiors.
By a nest, we mean that for every $n$, every sector in $\Delta_m^n$ is a union of finitely many sectors in $\Delta_m^{n+1}$. Here is how these triangulations are constructed.

The level $0$ triangulation $\Delta_m^0$ is defined by the union of the two closed sectors formed by cutting $\overline{\D}$ along $I^0_m$ and $I^{-1}_m$. 
We will denote by $A_{m,0}$ the sector in $\Delta^0_{m}$ with angular width $1-|\theta_m|$ and by $A_{m,1}$ the sector in $\Delta_{m,0}$ with angular width $|\theta_m|$.

For all $m \geq 0$ and $n \geq 1$, the sector $S_{m}^n$ is the union of two triangles $A^n_{m,0}$ and $A^n_{m,1}$ which are the lifts of $A_{m+n,0}$ and $A_{m+n,1}$ under the map $\psi_{m+n-1} \circ \ldots \circ \psi_{m}$.
The angular widths of $A^n_{m,0}$ and $A^n_{m,1}$ are $l_{[n-1]}(\theta_m) - l_{[n]}(\theta_m)$ and $l_{[n]}(\theta_m)$ respectively.

For $m \geq 0$ and $j \in \{0,1\}$, the first time $k \geq 1$ for which $\rotate_{\theta_m}^{-k}(A^1_{m,j})$ intersects the interior of $S_m^1$ is $k=b_m + j \varepsilon_m$. 
In fact, $\rotate_{\theta_m}^{-b_m - j \varepsilon_m}(A^1_{m,j})$ is contained in $S_m^1$.
In general, for $n \geq 1$, 
the level $n$ triangulation $\Delta_m^n$ consists of sectors with angular widths $l_{[n]}(\theta_m)$ and $l_{[n-1]}(\theta_m)-l_{[n]}(\theta_m)$. 
It is defined inductively in increasing $n$ as follows.
Suppose $\Delta_m^{n-1}$ has been defined for all $m \geq 0$.
Pull back $\Delta_{m+1}^{n-1}$ under $\psi_m$ to obtain a triangulation $\breve{\Delta}_m^n$ of $S_{m}^1$.
For every sector $B \in \breve{\Delta}_m^n$, denote $j_B \in \{0,1\}$ such that $B$ is contained in $A^{1}_{m,j_B}$.
Then,
\[
    \Delta_{m,n} := \left\{ \rotate_{\theta_m}^{-k} B\right\}_{B \in \in \breve{\Delta}_m^n, \; k = 0,1,\ldots, b_m + j_B \varepsilon_m-1}.
\]

\begin{proof}
    Item (1) follows directly from Proposition \ref{prop:ln} (3$_n$).
    Let us prove (2).
    
    Observe that the triangulation $\Delta_0^n$ is obtained out of cutting $\overline{\D}$ along the radial arcs $I^0_0, I^{-1}_0, \ldots I^{-N+1}_0$ where $N=N(n)$ is the first positive integer such that $I^{-N}_0$ intersects the interior of the sector $S_0^n$. 
    It is clear that the intersection of $A_{0,0}^n$ and $A_{0,1}^n$ is $I^0_0$.
    For $j \in \{0,1\}$, we just need to find the value of $k_j \in \{ 1,2,\ldots,N-1\}$ such that $A_{0,j}^n$ is bounded by $I^0_0$ and $I^{-k_j}_0$.
    Since $A_{0,1}^n$ has angular width equal to $l_{[n]}=l_{[n]}(\theta)$, then by (1), $k_1$ has to be equal to $q_{[n]}$.
    By (1) again, $\rotate_\theta^{ \varepsilon_n \varepsilon_{n+1} q_{[n-1]}}$ is the rotation by angle
    $\varepsilon_{n+1} l_{[n-1]}$ which has the same sign as $\rotate_\theta^{[n]}$.
    As such, $\rotate_\theta^{q_{[n]} - \varepsilon_n \varepsilon_{n+1} q_{[n-1]}}$ rotates by $l_{[n-1]} - l_{[n]}$ in the direction opposite to $\rotate_\theta^{[n]}$. 
    This implies that $k_0 = q_{[n]} - \varepsilon_n \varepsilon_{n+1} q_{[n-1]}$.

    The inequality in (3) follows from (1) and the fact that $|\theta_n|< \frac{1}{2}$.
    We will use (2) and the triangulation $\Delta_0^n$ to prove the rest of (3). 
    From the previous paragraph, we have $N \geq q_{[n]}$. 
    Let $A_{0,2}^n \in \Delta_0^n$ be the sector next to $A_{0,1}^n$ that is not equal to $A_{0,0}^n$, and let $A_{0,3}^n \in \Delta_0^n$ be the sector next to $A_{0,2}^n$ that is not equal to $A_{0,1}^n$.
    The argument is split into two cases.
    \begin{itemize}
        \item Suppose $\varepsilon_n \varepsilon_{n+1} = +1$. 
        Then, $A_{0,2}^n$ is bounded by $I^{-q_{[n]}}_0$ and $I^{-q_{[n-1]}}_0$ and its angular length is $l_{[n-1]}-l_{[n]}$.
        For any integer $k$ with $1 \leq k < q_{[n]}$, the radial segment $I^{-k}_0$ is disjoint from the interior of $S_0^n \cup A_{0,2}^n$ which implies that the angular distance between $I^{-k}_0$ and $I^0_0$ is bounded from below by $l_{[n-1]}-l_{[n]} > \frac{1}{2} l_{[n]}$. 
        \item Suppose $\varepsilon_n \varepsilon_{n+1} = -1$.
        Then, $A_{0,3}^n$ is bounded by $I^{-q_{[n]}}_0$ and $I_0^{-2q_{[n]}}$ and $A_{0,3}^n$ is bounded by $I_0^{-q_{[n]}+q_{[n-1]}}$.
        In this case, for any integer $k$ with $1 \leq k < q_{[n]}$, the radial segment $I_0^{-k}$ is disjoint from the interior of $S_0^n \cup A_{0,2}^n \cup A_{0,3}^n$ which implies that the angular distance between $I_0^{-k}$ and $I_0^0$ is again at least $l_{[n-1]}-l_{[n]} > \frac{1}{2} l_{[n]}$. 
    \end{itemize}
    In both cases, we have proven the optimality of $q_{[n]}$ described in item (3).
\end{proof}

From the proposition above, for any $n \geq 1$, the first return time of $\rotate_\theta$ at any point in the interior of $S_0^n$ is either $q_{[n]}$ or $q_{[n]} - \varepsilon_n \varepsilon_{n+1} q_{[n-1]}$. 
In other words, $\rotate_\theta^{[n]}|_{S_0^n}$ is the $n$\textsuperscript{th} pre-renormalization of $\rotate_\theta$:
\[
    \rotate_{\theta_n} = \rotate_\theta^{[n]}|_{S_0^n} \Big/ \rotate_\theta^{[n-1]}.
\]
The gluing map $\psi_{n-1} \circ \ldots \circ \psi_0$ projects the map $\rotate_\theta^{[n]}|_{S_0^n}$, modulo $\rotate_\theta^{[n-1]}$, to the rotation $\rotate_{\theta_n}$.

\section{Diophantine approximation}
\label{sec:continued-fraction}

In this section, we explore the arithmetic properties of irrationals relative to the modified Gauss map $\gauss: \Irrat \to \Irrat$ defined in the previous section.
Most of the results in this section are elementary and similar to those in the study of regular continued fractions. (E.g. compare with \cite{Khin}.)
Experts in Diophantine approximation should be familiar with most of the content in this section.
We provide the details for the sake of completeness.

\subsection{The convergents}
Given a finite sequence of non-zero integers $x_1,x_2,\ldots,x_n$, we will denote
\[
    [ x_1,x_2,\ldots, x_n ]_-  := 1/(x_1 - 1/(x_2 - \ldots - 1/(x_{n-1} - 1/x_n) \ldots )).
\]

\begin{example}
    The golden mean irrational 
    \[
    \theta_{\textnormal{gm}} = \frac{3 - \sqrt{5}}{2} = 0.381966\ldots
    \]
    is a fixed point of the map $x \mapsto \frac{1}{3-x}$ and so it can be written as
    \[
        \theta_{\textnormal{gm}} = [ 3,3,3,3,3,\ldots ]_- .
    \]
\end{example}

\begin{lemma}
\label{lem:p[n]-q[n]}
    Consider $\sigma  = \seq{ (\varepsilon_n, \abar_n) }_{n\geq 1} \in \The$ and the associated sequences of integers $b_n = b_n(\sigma)$, $p_{[n]} = p_{[n]}(\sigma)$ and $q_{[n]} = q_{[n]}(\sigma)$ described in \S\ref{ss:more-notation}. For all $n \geq 1$, we have
\begin{enumerate}[label=\textnormal{(\arabic*}$_n$\textnormal{)}]
    \item $p_{[n]}q_{[n-1]} - p_{[n-1]} q_{[n]} = \varepsilon_n$,
    \item $\displaystyle
        \frac{p_{[n]}}{q_{[n]}} = \sum_{k=1}^n \frac{\varepsilon_k}{ q_{[k-1]} q_{[k]}},
        $
    \item $\displaystyle 
    [ \varepsilon_1 b_1, \varepsilon_2 b_2, \ldots, \varepsilon_n b_n ]_-  = \frac{p_{[n]}}{q_{[n]}}$,
    \item $\displaystyle -\frac{1}{2} \leq \frac{p_{[n]}}{q_{[n]}} \leq \frac{1}{2}$,
    \item $q_{[n-1]} < (1-\theta_{\textnormal{gm}}) q_{[n]}$.
\end{enumerate}
\end{lemma}

\begin{proof}
    For $n=1$, (1$_{1}$) holds since $p_{[0]} = 0$, $q_{[0]}=1$, and $p_{[1]} = \varepsilon_1$.
    In general, the recurrence relations (\ref{eqn:recurrence-relation}) and (1$_{n-1}$) imply (1$_{n}$) as shown below.
    \begin{align*}
        & p_{[n]}q_{[n-1]} - p_{[n-1]} q_{[n]} \\
        = & (b_n p_{[n-1]}- \varepsilon_n \varepsilon_{n-1} p_{[n-2]}) q_{[n-1]} - p_{[n-1]} (b_n q_{[n-1]}- \varepsilon_n \varepsilon_{n-1} q_{[n-2]}) \\
        = & \varepsilon_n \varepsilon_{n-1} (p_{[n-1]}q_{[n-2]} - p_{[n-2]} q_{[n-1]}) \\
        = & \varepsilon_n \varepsilon_{n-1} \varepsilon_{n-1} = \varepsilon_n.
    \end{align*}
    
    Let us prove (2$_n$). The base case (2$_n$) is clear. For $n \geq 2$, (2$_n$) follows from (2$_{n-1}$) and (1$_n$) since
    \[
        \frac{p_{[n]}}{q_{[n]}} - \frac{p_{[n-1]}}{q_{[n-1]}} = \frac{p_{[n]} q_{[n-1]} - p_{[n-1]}q_{[n]} }{q_{[n-1]} q_{[n]}} = \frac{\varepsilon_n}{q_{[n-1]} q_{[n]}} .
    \]

    By elementary calculation, we have (3$_{1}$) and (3$_{2}$):
    \[
        [\varepsilon_1 b_1]_-  = \frac{\varepsilon_1}{b_1} = \frac{p_{[1]}}{q_{[1]}}, \qquad
        [\varepsilon_1 b_1, \varepsilon_2 b_2]_-  = \frac{\varepsilon_1 b_2}{b_2 b_1 - \varepsilon_2 \varepsilon_1} =
        \frac{p_{[2]}}{q_{[2]}}.
    \]
    Suppose we have (3$_n$) for all $n \in \{1,\ldots,k\}$ for some integer $k \geq 2$.
    Let $x = b_{k} - \frac{\varepsilon_{k} \varepsilon_{k+1}}{b_{k+1}}$. Then, by the recurrence relations,
    \begin{align*}
        [ \varepsilon_1 b_1, \ldots, \varepsilon_{k-1} b_{k-1}, \varepsilon_k b_k, \varepsilon_{k+1} b_{k+1}]_- 
        = & [ \varepsilon_1 b_1, \ldots, \varepsilon_{k-1} b_{k-1}, \varepsilon_k x ]_-  \\
        = & \frac{x p_{[k-1]}- \varepsilon_{k-1} \varepsilon_k p_{[k-2]}}{x q_{[k-1]}- \varepsilon_{k-1} \varepsilon_k q_{[k-2]}} \\
        = & \frac{p_{[k]} - \varepsilon_k \varepsilon_{k+1} p_{[k-1]}/b_{k+1}}{q_{[k]} - \varepsilon_k \varepsilon_{k+1} q_{[k-1]}/b_{k+1}} = \frac{p_{[k+1]}}{q_{[k+1]}}.
    \end{align*}
    Hence, (3$_n$) holds for all $n$.

    Next, we will use (3$_n$) to prove (4$_n$). 
    The base case (4$_1$) is clear. 
    Suppose (4$_{n-1}$) holds for some $n\geq 2$.
    In particular, $y := \varepsilon_2 [ \varepsilon_2 b_2, \ldots, \varepsilon_n b_n ]_- $ is contained in the interval $[0, \frac{1}{2}]$.
    Then, we can write
    \[
        \left| \frac{p_{[n]}}{q_{[n]}} \right| = \left| \frac{1}{\varepsilon_1 b_1 - \varepsilon_2 y} \right| = \frac{1}{b_1 - \varepsilon_1 \varepsilon_2 y} = \begin{cases}
            \frac{1}{\abar_1 +1 - y} & \text{ if } \varepsilon_1 \varepsilon_2 = +1, \\
            \frac{1}{\abar_1 + y} & \text{ if } \varepsilon_1 \varepsilon_2 = -1.
        \end{cases}
    \]
    Since $\abar_1 \geq 2$ and $0 \leq y \leq \frac{1}{2}$, the equation above gives us (4$_n$).

    To prove (5$_n$), we will show that for all $n \geq 1$, we have
    \begin{equation}
    \label{eqn:theta-gm-bound}
        \theta_{\textnormal{gm}} - 1 < \varepsilon_{n} \varepsilon_{n+1} \frac{q_{[n-1]}}{q_{[n]}} < \theta_{\textnormal{gm}}.
        \tag{$\spadesuit_n$}
    \end{equation}
    The base case ($\spadesuit_1$) follows from $\frac{1}{3} < \theta_{\textnormal{gm}}<\frac{1}{2} < 1- \theta_{\textnormal{gm}}$ and the following observation:
    \[
        \varepsilon_1 \varepsilon_2 \frac{q_{[0]}}{q_{[1]}} = \frac{\varepsilon_1 \varepsilon_2}{b_1} =
        \begin{cases}
            \frac{1}{\abar_1 + 1} \in (0,\frac{1}{3}] & \text{ if } \varepsilon_1 \varepsilon_2 = +1,\\
            -\frac{1}{\abar_1} \in [-\frac{1}{2},0) & \text{ if } \varepsilon_1 \varepsilon_2 = -1.
        \end{cases}
    \]
    In general, for $n \geq 2$, the recurrence relation \ref{eqn:recurrence-relation} gives us:
    \begin{align*}
        \varepsilon_{n} \varepsilon_{n+1} \frac{q_{[n-1]}}{q_{[n]}}
        = \frac{\varepsilon_{n} \varepsilon_{n+1}}{b_n - \varepsilon_n \varepsilon_{n-1} \frac{q_{[n-2]}}{q_{{n-1}}}} 
        = \begin{cases}
            \cfrac{1}{\abar_n + 1- \varepsilon_n \varepsilon_{n-1} \frac{q_{[n-2]}}{q_{{n-1}}}} & \text{ if } \varepsilon_n \varepsilon_{n+1} = +1,\\
            -\cfrac{1}{\abar_n - \varepsilon_n \varepsilon_{n-1} \frac{q_{[n-2]}}{q_{{n-1}}}} & \text{ if } \varepsilon_1 \varepsilon_{n+1} = -1.
        \end{cases}
    \end{align*}
    Then, ($\spadesuit_n$) follows from ($\spadesuit_{n-1}$):
    \begin{align*}
        \varepsilon_{n} \varepsilon_{n+1} \frac{q_{[n-1]}}{q_{[n]}}
        &\in \begin{cases}
            \left(0, \frac{1}{3 - \theta_{\textnormal{gm}}} \right) \enspace = (0, \theta_{\textnormal{gm}}) & \text{ if } \varepsilon_n \varepsilon_{n+1} = +1,\\
            \left( - \frac{1}{2 - \theta_{\textnormal{gm}}}, 0 \right) = (\theta_{\textnormal{gm}}-1,0) & \text{ if } \varepsilon_1 \varepsilon_{n+1} = -1.
        \end{cases} \qedhere
    \end{align*}
\end{proof}

\begin{proposition}
    Consider $\sigma  = \seq{ (\varepsilon_n, \abar_n ) }_{n\geq 1} \in \The$ and the associated sequences $b_n = b_n(\sigma)$, $p_{[n]} = p_{[n]}(\sigma)$, and $q_{[n]} = q_{[n]}(\sigma)$. The limit $\mathfrak{X}(\sigma) = \displaystyle \lim_{n\to \infty} \frac{p_{[n]}}{q_{[n]}}$ converges to an irrational in $\Irrat$ equal to
    \[
        \mathfrak{X}(\sigma) = \sum_{k=1}^\infty \frac{\varepsilon_k}{q_{[k-1]}q_{[k]}}
    \]
    Moreover, for every $n \geq 1$,
    \[
        \frac{\theta_{\textnormal{gm}}}{q_{[n]} q_{[n+1]}} \leq \left| \mathfrak{X}(\sigma) - \frac{p_{[n]}}{q_{[n]}} \right| \leq \frac{2-\theta_{\textnormal{gm}}}{q_{[n]} q_{[n+1]}}.
    \]
\end{proposition}

\begin{proof}
    For $k \geq 1$, denote $x_k = \cfrac{\varepsilon_k}{q_{[k-1]} q_{[k]}}$.
    By Lemma \ref{lem:p[n]-q[n]} (5), we have
    \begin{equation}
    \label{eqn:x_k-ratio}
        \left| \frac{x_{k+1}}{x_{k}} \right| = \frac{q_{[k-1]}}{q_{[k+1]}} < (1-\theta_{\textnormal{gm}})^2 = \theta_{\textnormal{gm}} < 1.
    \end{equation}
    Therefore, the series $\sum_{k\geq 1} x_k$ converges to an irrational number $\theta$.
    By items (2), (3), and (4) of Lemma \ref{lem:p[n]-q[n]}, $\theta$ is in $\Irrat$, is equal to $\mathfrak{X}(\sigma)$, and has an infinite modified continued fraction expansion $\theta = 
    [ \varepsilon_1 b_1, \varepsilon_2 b_2, \varepsilon_3 b_3, \ldots ]_-$.
    Moreover, for $n \geq 1$, the inequality (\ref{eqn:x_k-ratio}) gives us
    \begin{align*}
        \left| \theta - \frac{p_{[n]}}{q_{[n]}} \right| 
        &= \left| \sum_{k = n+1}^\infty x_n \right| \leq \frac{|x_{n+1}|}{1- \theta_{\textnormal{gm}}} = \frac{1}{(1-\theta_{\textnormal{gm}}) q_{[n]} q_{[n+1]}} = \frac{2-\theta_{\textnormal{gm}}}{q_{[n]} q_{[n+1]}}
    \end{align*}
    and
    \begin{align*}
        \left| \theta - \frac{p_{[n]}}{q_{[n]}} \right| 
        \geq |x_{n+1}| - \left| \sum_{k = n+2}^\infty x_n \right|
        & \geq \frac{1}{q_{[n]}q_{[n+1]}} - \frac{2-\theta_{\textnormal{gm}}}{q_{[n+1]}q_{[n+2]}}
        \\
        & \geq \frac{ 1 - (2-\theta_{\textnormal{gm}})(1-\theta_{\textnormal{gm}})^2 }{q_{[n]}q_{[n+1]}} 
        = \frac{\theta_{\textnormal{gm}}}{q_{[n]}q_{[n+1]}}.
    \end{align*}
    Note that we have used Lemma \ref{lem:p[n]-q[n]} (5) on the last line above.
\end{proof}

The proposition above gives us a well-defined function 
\[
\mathfrak{X} : \The \to \Irrat.
\]
Recall the functions $\gauss: \Irrat \to \Irrat$ and $\mathfrak{Y}: \Irrat \to \The$ defined in the previous section.

\begin{theorem}
\label{prop:projection-to-irrationals}
    The function $\mathfrak{X}: \The \to \Irrat$ is a homeomorphism with inverse
    $\mathfrak{Y}: \Irrat \to \The$.
    Moreover, $\mathfrak{X}$ is a conjugacy between the standard shift map $\shift: \The \to \The$ and the map $\gauss: \Irrat \to \Irrat$.
\end{theorem}

This theorem settles Theorem \ref{main-thm-1}.

\begin{proof}
    Let $\sigma = \seq{ (\varepsilon_n,\abar_n) }_{n\geq 1} \in \The$ and let $b_n = b_n(\sigma)$ for $n \geq 1$.
    We have 
    \[
        \mathfrak{X}(\shift (\sigma)) = [\varepsilon_2 b_2, \varepsilon_3 b_3, \ldots]_- = -\frac{1}{\mathfrak{X}(\sigma)} + \varepsilon_1 b_1 = \gauss ( \mathfrak{X}(\sigma)).
    \]
    Let $\theta = \mathfrak{X}(\sigma)$.
    Then, $\varepsilon(\theta) = \varepsilon_1$ and $\abar(\theta) = \abar_1$.
    Since $\gauss^{n-1}(\theta) = \mathfrak{X}(\shift^{n-1}(\sigma))$, we also have that $\varepsilon(\gauss^{n-1}(\theta)) = \varepsilon_n$ and $\abar(\gauss^{n-1}(\theta)) = \abar_n$.
    Therefore, $\mathfrak{Y} \circ \mathfrak{X}$ is equal to the identity map on $\The$.

    We have established that $\mathfrak{Y}$ is surjective. Now, we will show the injectivity of $\mathfrak{Y}$.
    Fix $\sigma = \seq{ (\varepsilon_n,\abar_n) }_{n\geq 1}$.
    For $k \geq 1$, denote
    \[
        E_k(\sigma) = \left\{ \theta \in \Irrat \: : \: \varepsilon(\gauss^{n-1}(\theta)) = \varepsilon_n \text{ and } \abar(\gauss^{n-1}(\theta)) = \abar_n \text{ for } 1\leq n \leq k \right\}.
    \]
    The preimage $\mathfrak{Y}^{-1}(\sigma)$ is non-empty and it is contained in $\cap_{n\geq 1} E_k(\sigma)$.
    The set $E_1(\sigma)$ is the open subset of $\Irrat$ consisting of irrationals between $\frac{1}{\varepsilon_1\abar_1}$ and $\frac{1}{\varepsilon_1(\abar_1+1)}$; since $\abar \geq 2$, this open set has Euclidean length bounded above by $6^{-1}$.
    In general, $\gauss^{k-1}$ sends $E_k(\sigma)$ homeomorphically onto $E_1(\shift^{k-1}(\sigma))$.
    Observe that $\gauss$ is uniformly expanding: $\gauss' \geq 4$.
    This implies that $E_k(\sigma)$ is equal to $\Irrat \cap J$ for some open interval $J$ of Euclidean length at most $6^{-1} \cdot 4^{1-k}$.
    In particular, $\mathfrak{Y}^{-1}(\sigma)$ is a singleton and thus $\mathfrak{Y}$ is injective.

    Pick $\sigma \in \The$ and let $\theta = \mathfrak{X}(\sigma)$.
    The sets $E_k(\sigma)$, $k \geq 1$ described above form a basis of open neighborhoods of $\theta$. The images $\mathfrak{Y}(E_k(\sigma))$, $k \geq 1$ are precisely the cylinder sets about $\sigma$ which form a basis of open neighborhoods of $\sigma$ with respect to the topology of $\The$.
    Therefore, both $\mathfrak{X}$ and $\mathfrak{Y}$ are continuous.
\end{proof}

\begin{example}
    For the golden mean irrationals in $\Irrat$ are $\theta_{\textnormal{gm}} = \frac{3-\sqrt{5}}{2}$ and $-\theta_{\textnormal{gm}} = \frac{\sqrt{5}-3}{2}$, we have
    \begin{align*}
        \mathfrak{Y}(\theta_{\textnormal{gm}}) &= \seq{ (+,2), (+,2), (+,2), (+,2),\ldots },\\
        \mathfrak{Y}(-\theta_{\textnormal{gm}}) &= \seq{ (-,2), (-,2), (-,2), (-,2), \ldots }.
    \end{align*}
    The corresponding modified continued fraction expansions are
    \[
        \theta_{\textnormal{gm}} = [ 3,3,3,\ldots ]_-  =\cfrac{1}{3 - \frac{1}{3 - \frac{1}{3 - \ldots}}}, \qquad
        -\theta_{\textnormal{gm}} = [ -3,-3,-3,\ldots]_-  = \cfrac{1}{-3 - \frac{1}{-3 - \frac{1}{-3 - \ldots}}}.
    \]
    Both $\theta_{\textnormal{gm}}$ and $-\theta_{\textnormal{gm}}$ have the same first return times, namely
    \[
        q_{[0]} = 1, \quad q_{[1]} = 3, \quad q_{[2]}=8, \quad q_{[3]} = 21, \quad q_{[4]} = 55, \quad q_{[5]} = 144, \quad \ldots.
    \]
    Each $q_{[n]}$ is the $(2n+2)$\textsuperscript{th} Fibonacci number.
\end{example}

\begin{example}
    If $\theta = \sqrt{2}-1$, then 
    \[
        \theta = \cfrac{1}{2+\frac{1}{2+\frac{1}{2+\ldots}}} \quad \text{ and } \quad \mathfrak{Y}(\theta) = \seq{ (+,2), (-,2), (+,2),(-,2),\ldots }.
    \]
\end{example}

\subsection{The nearest integer continued fraction expansion}
\label{ss:nearest-integer}

Consider the nearest integer continued fraction map
\[
    f_{1/2}: \Irrat \to \Irrat, \qquad
    f_{1/2}(\theta) = \frac{1}{|\theta|} - \left\lfloor \frac{1}{|\theta|} + \frac{1}{2}\right\rfloor.
\]
For $n \geq 0$, we define the sequences $\{ \tilde{a}_n \}_{n\geq 1}$ in $\N_{\geq 2}$, $\{ \tilde{\theta}_n \}_{n\geq 0}$ in $(0,\frac{1}{2})\backslash \Q$, and $\{ \tilde{\varepsilon}_n\}_{n \geq 1}$ in $\{-1,+1\}$ inductively as follows:
\begin{itemize}
    \item $\tilde{\varepsilon}_{n+1}$ is the sign of $\tilde{\theta}_n$;
    \item $\tilde{a}_{n+1} = \left\lfloor |\tilde{\theta}_n|^{-1} + \frac{1}{2}\right\rfloor$ is the integer closest to $|\tilde{\theta}_n|^{-1}$;
    \item $\tilde{\theta}_{n+1} = f_{1/2}(\tilde{\theta}_n) = |\tilde{\theta}_n|^{-1} - \tilde{a}_{n+1}$.
\end{itemize}
Then, we can write
\begin{equation}
    \theta = \cfrac{\tilde{\varepsilon}_1}{\tilde{a}_1 + \cfrac{\tilde{\varepsilon}_2}{\tilde{a}_2 + \frac{\tilde{\varepsilon}_3}{\tilde{a}_3 + \ldots}}}. 
    \tag{$\ddagger$}
\end{equation}
The map $f_{1/2}$ is an even function and it is related to $\gauss$ by 
\[
f_{1/2}(\theta) = \begin{cases}
    \gauss(\theta) & \text{ if } \theta < 0, \\
    -\gauss(\theta) & \text{ if } \theta > 0.
\end{cases} 
\]
It produces the same convergents $\tilde{p}_{[n]}/\tilde{q}_{[n]}$ as those $p_{[n]}/q_{[n]}$ of $\gauss$ up to a sign change.
The terms $\tilde{a}_n$'s coincide with the $b_n$'s described in \S\ref{ss:more-notation}. 
Our modified Gauss map $\gauss$ has two primary advantages over $f_{1/2}$.
\begin{itemize}
    \item The map $\gauss$ locally preserves orientation everywhere, whereas $f_{1/2}$ reverses orientation the positive irrationals.
    \item While $f_{1/2}$ still acts as a shift on $(\tilde{\varepsilon}_n,\tilde{a}_n)$'s, these $\tilde{a}_n$'s satisfy the awkward admissibility rule that $\tilde{a}_n \geq 3$ whenever $\tilde{\varepsilon}_n \tilde{\varepsilon}_{n+1} = -1$.
\end{itemize}

Nearest integer continued fraction was initially introduced by Bernhard Minnigerode \cite{Mi1873} in 1873.
The notation $f_{1/2}$ was adapted from Nakada \cite{Na81} who introduced the one-parameter family 
\[
f_{\alpha}: [ \alpha -1, \alpha) \to [ \alpha -1, \alpha), \qquad \alpha \in \left[ \frac{1}{2},1 \right]
\]
of continued fractions map, where $f_1$ is the classical Gauss map.

In complex dynamics, the expansion ($\ddagger$) first appeared in Yoccoz's work \cite{Yoc95} where he introduced sector renormalization.
The adaptation of ($\ddagger$) was also evident in the work of Inou-Shishikura \cite{IS} on near-parabolic renormalization as well as its numerous applications, for example \cite{CC15,AC17,Che25}.
In those works, the authors are primarily working with the sequence of positive irrationals 
\[
\alpha_n = |\tilde{\theta}_n|, \qquad n \geq 0.
\]

\subsection{The regular continued fraction expansion}
\label{ss:conversion}

Every irrational $\theta \in \Irrat$ admits a unique continued fraction expansion
\[
    \theta = \varepsilon \, [a_1,a_2,a_3,\ldots]_+ := \cfrac{\varepsilon}{a_1+\cfrac{1}{a_2 + \frac{1}{a_3 + \ldots}}}
\]
where $\varepsilon \in \{-1,+1\}$, $a_1 \in \N_{\geq 2}$, and $a_n \in \N$ for all $n \geq 2$.
For $n \geq 1$, the $n$\textsuperscript{th} rational approximation of $\theta$ is 
\[
\frac{p_n}{q_n} = \varepsilon \, [a_1,a_2,\ldots,a_n]_+
\]
where $p_n$ and $q_n$ are co-prime integers and $q_n > 0$. These $p_n$'s and $q_n$'s start with
\[
    p_n = 0, \; q_0 = 1, \qquad p_1 = \varepsilon, \; q_1 = a_1,
\]
and the rest follow the recurrence relation
\[
    p_n = a_n p_{n-1} + p_{n-2}, \qquad q_{n} = a_n q_{n-1} + q_{n-2} \quad  \text{ for } n \geq 2.
\]
Note that the $q_n$'s are independent of the sign $\varepsilon$.
The numbers $q_n \theta - p_n$ alternate in sign, that is 
\[
(q_n \theta - p_n)(q_{n+1} \theta - p_{n+1}) < 0 \qquad \text{ for all } n \geq 0.
\]

There are two advantages of the regular continued fraction expansion, namely the simpler recursive formula for the rational approximations and the alternating nature of points of closest return.
The main disadvantage is that the regular expansion does not get transformed as nicely under sector renormalization:
\[
    \gauss\left( \varepsilon [a_1,a_2,\ldots]_+ \right) =
    \begin{cases}
        -\varepsilon \, [a_2,a_3,\ldots]_+ & \text{ if } a_2 \geq 2, \\
        \varepsilon \, [a_3+1, a_4, \ldots]_+ & \text{ if } a_2 = 1.
    \end{cases}
\]
Arithmetically, the regular expansion is more suited to studying the action of the traditional Gauss map $G(\theta) = \{ \frac{1}{\theta} \}$ acting on irrationals on the unit interval $(0,1)$.

There is a singularization procedure to convert the regular continued fraction expansion $\theta = \varepsilon \, [a_1, a_2,\ldots]_+$ to the modified one with terms governed by a sequence $\seq{ (\varepsilon_n,\abar_n) }_{n\geq 1}$ in $\The$.
First, set $(\varepsilon_1, \varepsilon) = (\varepsilon, a_1)$. To get $(\varepsilon_2, \abar_2)$, there are $2$ cases.
\begin{itemize}
    \item If $a_2 \geq 2$, then $\varepsilon_2 = - \varepsilon_1$, and $\abar_2 = a_2$.
    \item Else if $a_2 = 1$, then $\varepsilon_2 = \varepsilon_1$, and $\abar_2 = a_3+1$.
\end{itemize}
In general, it is convenient to use the increasing sequence of integers
\[
    0 = t_0 < t_1 < t_2 < t_3 < \ldots
\]
defined as follows. For $n \geq 0$,
\begin{itemize}
    \item if $a_{t_n + 2} \geq 2$, then $t_{n+1} = t_{n} + 1$, $\varepsilon_{n+2} = - \varepsilon_{n+1}$ and $\abar_{n+2} = a_{t_n + 2}$;
    \item if $a_{t_n + 2} = 1$, then $t_{n+1} = t_n + 2$, $\varepsilon_{n+2} = \varepsilon_{n+1}$ and $\abar_{n+2} = a_{t_n + 3} + 1$.
\end{itemize}
With this convention, we have that 
\[
    q_{[n]} = q_{t_n} \qquad \text{ for all } n\geq 0.
\]
Geometrically, here is what happens.
Consider the lengths 
\[
l_n := |q_n \theta - p_n|, \qquad n \geq 1
\]
which form strictly decreasing sequence.
The rate of decrease is generally non-uniform and the ratio $l_{n+1}/l_n$ can be arbitrarily close to $1$.
According to Proposition \ref{prop:q[n]-02}, the subsequence $\{t_n\}_{n \geq 0}$ records the moments in which we have definite contraction by one half:
\[
    l_{t_{n+1}} < \frac{l_{t_n}}{2}.
\]
The moments that are not of the form $t_n$ are deemed as ``negligible'' levels and can be thrown out.

\section{The parabolic compactification}

In this section, we will discuss a natural compactification $\TheCpt$ of $\Irrat$ that respects the dynamical behavior of $\gauss$.
In \S\ref{ss:universality}, we show that $\TheCpt$ admits universal properties in the sense of Theorem \ref{main-thm-2}.
In \S\ref{ss:continuant-1}, we define the continuant group $\Kont_{\ttheta}$ associated to every element $\ttheta$ of $\TheCpt$.

\subsection{The space $\TheCpt$}

\begin{definition}
    Consider a continuous dynamical system $f: X \to X$ on a topological space $X$.
    A \emph{compactification} of the dynamical system $(X,f)$ is a triple $(j,\bar{X},\bar{f})$ where
    \begin{enumerate}[label=\textnormal{(\alph*)}]
        \item $j: X \to \bar{X}$ is an embedding with dense image,
        \item $\bar{X}$ is a compact space, and
        \item $\bar{f}: \bar{X} \to \bar{X}$ is a continuous map such that $\bar{f} \circ j = j \circ f$.
    \end{enumerate}
    We say that a compactification $(j,\bar{X},\bar{f})$ of $(X,f)$ \emph{respects} an embedding $h: X \to Y$ onto a compact space $Y$ if there exists a continuous map $\bar{h}: \bar{X} \to Y$, called a \emph{$j$-extension} of $h$, such that $\bar{h} \circ j = h$.
\end{definition}

In the definition above, both $\bar{f}$ and $\bar{h}$ are unique because of the density of $j(X)$ in $\bar{X}$.
Theorem \ref{main-thm-1} gives us a natural compactification of $(\Irrat,\gauss)$.

\begin{definition}
    Let $\overline{\N} = \N \cup \{\infty\}$, where the addition of the point $\displaystyle \infty = \lim_{n \to \infty} n$ makes it a compactification of the discrete space $\N$.
    Let 
\[
\overline{\Sigma} := \{-1,+1\} \times \overline{\N}_{\geq 2} \qquad \text{ and } \qquad
    \TheCpt := \overline{\Sigma}^{\N}. 
\]
    Given an element $\ttheta = \seq{(\varepsilon_n, \abar_n)}_{n \geq 1}$ of $\TheCpt$, any integer $n \geq 0$ such that $\abar_{n+1} = \infty$ is called a \emph{parabolic level} for $\ttheta$.
    We say that $\ttheta$ is called an \emph{enriched rational} if parabolic levels exist, an \emph{irrational} if otherwise.
    For $n \geq 1$, we will denote
    \[
        b_n(\ttheta) := \begin{cases}
            \abar_n + \frac{1+\varepsilon_n \varepsilon_{n+1}}{2} & \text{ if } \abar_n < \infty, \\
            \infty & \text{ if } \abar_n = \infty.
        \end{cases}
    \]
\end{definition}
    
The sequence space $\TheCpt$ is topologically a Cantor set. 
The inclusion map $\iota : \Sigma^{\N} \to \TheCpt$ induces a dense embedding 
    \[
    \bar{\iota}:= \iota \circ \mathfrak{Y} : \Irrat \to \TheCpt
    \]
and the modified Gauss map $\gauss$ naturally admits an $\bar{\iota}$-extension
\[
    \bar{\gauss}: \TheCpt \to \TheCpt
\]
given by the standard shift map $\bar{\gauss} \equiv \shift$.
Hence, $(\bar{\iota}, \TheCpt, \bar{\gauss})$ is a compactification of $(\Irrat, \gauss)$.

There are three compact embeddings of $\Irrat$ of high interest, each of which admits an $\bar{\iota}$-extension. 

\subsubsection{The modular embedding}

The first one is the naive embedding
\[
\mu : \Irrat \to \T, \qquad
\theta \mapsto \theta \textnormal{ (mod }1\textnormal{)} 
\]
into the circle $\T = \R/\Z$, called the \emph{modular embedding}.

\begin{proposition}
    The modular embedding $\mu$ admits an $\bar{\iota}$-extension $\overline{\mu} : \TheCpt \to \T$.
    For every rational number $\theta$ in $\mathbb{T}$, $\overline{\mu}^{-1}(\theta)$ is homeomorphic to $\TheCpt$ itself.
\end{proposition}

\begin{proof}
    For an enriched rational $\ttheta = \seq{ (\varepsilon_n,\abar_n) }_{n\geq 1} \in \TheCpt$ with smallest parabolic level $N \geq 0$,
    we set 
    \[
    \bar{\mu}(\ttheta) = \begin{cases}
        0 & \text{ if } N = 0, \\
        [\varepsilon_1 b_1,\ldots, \varepsilon_{N} b_{N}]_- & \text{ if } N \geq 1.
    \end{cases}
    \]
    The continuity of $\overline{\mu} : \TheCpt \to \mathbb{T}$ follows from the proof of Proposition \ref{prop:projection-to-irrationals}.
    
    For the second claim, observe that 
    \[
    \overline{\mu}^{-1}(0) = \left\{ 
    \abar_1=\infty \right\}, \quad \overline{\mu}^{-1}\left(\pm \frac{1}{2}\right) = \left\{ 
    \varepsilon_1 \varepsilon_2 = -,  \,\abar_1=2, \, \abar_2=\infty \right\},
    \]
    and for $m \geq 3$ and $\varepsilon \in \{-1,+1\}$, we have 
    \[
    \overline{\mu}^{-1}\left(\frac{\varepsilon}{m}\right) = \left\{ 
    \varepsilon_1 = \varepsilon, \abar_1 + \frac{1+\varepsilon \varepsilon_2}{2} =m, \abar_2=\infty \right\}.
    \]
    All of the above are homeomorphic to $\TheCpt$.
    For other rational numbers $\theta$, let $k \geq 2$ be the first moment such that $\gauss(\theta) = 0$, then $\varepsilon_j(\theta)$ and $\abar_j(\theta)$ are well-defined and finite for $j \leq k-1$, $\abar_{k+1}= \infty$, the sequences $\{\varepsilon_j\}_{j \geq k+1}$ and $\{\abar_j\}_{j\geq k+2}$ are completely arbitrary, and the pair $(\varepsilon_k, \abar_k)$ is completely determined by $\varepsilon_{k+1}$ in the way that is similar to the $\pm \frac{1}{m}$ case.
\end{proof}

This proposition tells us that, with respect to $\overline{\mu}$, the space $\TheCpt$ can be viewed as a way of blowing up the circle $\mathbb{T}$ at its rational points.
In light of applications in holomorphic dynamics, the structure of the fibers of a rational point on the circle sheds light on the necessity to blow up the Mandelbrot set at parabolic parameters in order to claim continuity of dynamical sets (e.g. the Mother Hedgehog, the (filled) Julia set, etc).

\subsubsection{The segment embedding}

The second embedding to consider is 
\[
    \eta: \Irrat \to [0,1], \qquad \eta(\theta) = 
    \begin{cases}
        \theta & \text{ if } \theta > 0, \\
        \theta + 1 & \text{ if } \theta <0,
    \end{cases}
\]
which we call the \emph{segment embedding}.
This map is more natural to consider rather than $\mu$ for the following reason.
The modified Gauss map $\gauss: \Irrat \to \Irrat$ extends to a self-map of $\mathbb{T}$ that is continuous (in fact, $C^1$ smooth) everywhere except at the origin. 
See Figure \ref{fig:gauss} for reference.
Since the origin is singular from both the left and right, it is natural to distinguish between the left side of $0$ and the right side of $0$.
The segment embedding precisely does this as it identifies the ends $-\frac{1}{2}$ and $\frac{1}{2}$ and splits the left and right sides of the origin to two separate ends $0$ and $1$.
The diagram below is commutative.
\begin{center}
\begin{tikzcd}
    \Irrat \arrow[rd, "\mu"'] \arrow[r, "\eta"] & {[0,1]} \arrow[d, "0 \sim 1"] \\
    & \T
\end{tikzcd}
\end{center}

\begin{proposition}
\label{prop:compact-eta}
    The segment embedding $\eta$ admits an $\bar{\iota}$-extension $\bar{\eta}: \TheCpt \to [0,1]$, which is surjective.
\end{proposition}

\begin{proof}
    On the irrationals, $\eta$ can be symbolically written as
\[
    \eta \circ \mathfrak{X}( \seq{ (\varepsilon_n,\abar_n) }_{n\geq 1} ) = 
    \frac{1-\varepsilon_1}{2} + [\varepsilon_1 b_1,\varepsilon_2 b_2, \ldots]_-.
\]
    Define the extension $\bar{\eta}$ symbolically in the same way as $\eta \circ \mathfrak{X}$.
    For the enriched rationals, we make the convention that if $N \geq 0$ is the smallest parabolic level, then
    \[
        [\varepsilon_1 b_1, \varepsilon_2 b_2, \ldots]_- 
        = \begin{cases}
            0 & \text{ if } N = 0, \\
            [\varepsilon_1 b_1, \varepsilon_2 b_2, \ldots, \varepsilon_N b_N]_- & \text{ if } N \geq 1.
        \end{cases}
    \]
    The continuity and surjectivity of $\bar{\eta}$ are straightforward.
\end{proof}

\subsubsection{The dynamical embedding}

The third embedding to consider is the \emph{dynamical embedding} $\ord: \Irrat \to \Omega$ mentioned in the introduction.
It has a strong motivation arising from the renormalization theory of critical circle maps and neutral quadratic polynomials.

Given integers $k,l \geq 1$, let $\tilde{\Omega}_{k,l}$ denote the set of functions from $\{0,1\ldots,k\}^{l+1}$ from $\T$.
We say that two functions $f, g$ in $\tilde{\Omega}_{k,l}$ are equivalent if there is an orientation-preserving self-homeomorphism $h$ of $\T$ such that $f = h \circ g$.
Denote by $\Omega_{k,l}$ the corresponding set of equivalence classes. 
Equivalently, $\Omega_{k,l}$ is the set of all partitions of $\{0,1,\ldots,k\}^{l+1}$ into $s \leq (k+1)^{l+1}$ blocks (corresponding to the fibers of $f$) together with a cyclic order on the blocks.
The set $\Omega_{k,l}$ is finite and we will equip it with the discrete topology.
We will consider the product space
\[
    \Omega = \prod_{k, l \geq 1} \Omega_{k,l}
\]
which is a Cantor set.
Given an irrational $\theta \in \Irrat$, 
we denote by $\ord_{k,l}(\theta) \in \Omega_{k,l}$ the ambient isotopy class of the embedding 
\[
    (j_0,\ldots,j_{l}) \mapsto 
    \left( \rotate_{\theta}^{[0]} \right)^{j_0}\circ \left( \rotate_{\theta}^{[1]} \right)^{j_1} \circ \ldots \circ \left( \rotate_{\theta}^{[l]} \right)^{j_{l}} (1).
\]
This gives us a locally constant map
\[
    \ord_{k,l}: \Irrat \to \Omega_{k,l} \qquad \text{ for all } k,l \geq 1,
\]
which altogether induces the embedding
\[
    \ord: \Irrat \to \Omega, \qquad
    \ord(\theta) = \{\ord_{k,l}(\theta)\}_{k,l}.
\]

\begin{proposition}
\label{prop:compact-ord}
    The embedding $\ord$ admits an $\bar{\iota}$-extension $\overline{\ord}: \TheCpt \to \Omega$, which is again an embedding.
\end{proposition}

\begin{proof}
    The proof is split into two steps, namely the continuity and well-definedness of $\overline{\ord}$, and the injectivity of $\overline{\ord}$.
    Since $\TheCpt$ is compact and $\Omega$ is Hausdorff, these two steps are enough to prove that $\overline{\ord}$ is an embedding.

    For $\ttheta \in \TheCpt$ and $M,l \in \N$, define
    \[
        \bar{U}_{M,l}(\ttheta) := \left\{ \seq{(\varepsilon_n, \abar_n)}_{n \geq 1} \in \TheCpt \: : \: 
        \begin{array}{c}
             \textnormal{for } n \in \{1,\ldots, l\}, \:\varepsilon_n = \varepsilon_n(\ttheta),  \\
             \abar_n \geq M \textnormal{ if } \abar_n = \infty, \textnormal{ and}  \\
             \abar_n = \abar_n(\ttheta) \textnormal{ if } \abar_n < \infty 
        \end{array}
        \right\}.
    \]
    The set $\bar{U}_{M,l}(\ttheta)$ is a clopen neighborhood of $\ttheta$ in $\TheCpt$.
    Let $U_{M,l}(\ttheta)$ be the open subset of $\Irrat$ given by
    \[
        U_{M,l}(\ttheta) = \bar{\iota}^{-1}\left( \bar{U}_{M,l}(\ttheta) \right).
    \]
    Denote 
    \[
        m_{l}(\ttheta) := \max\{ \abar_j \: : \: \abar_j < \infty, 1\leq j \leq l\}.
    \]
    
    \noindent \underline{Claim:} For any $k$ and $l \geq 1$, there exists a sufficiently high positive integer $M$ depending only on $k$, $l$, and $m_l(\ttheta)$ such that $\ord_{k,l}$ is constant in $U_{M,l}(\ttheta)$. 
    This constant will be denoted by $\overline{\ord}_{k,l}(\ttheta)$. 

    \begin{proof}
        Consider two distinct tuples $(i_s)_{0\leq s \leq l}$ and $(j_s)_{0\leq s \leq l}$ in $\{0,\ldots,k\}^{l+1}$, and write $w_s = j_s - i_s$. 
        Since we are dealing with irrational numbers $\theta \in U_{M,l}(\ttheta)$, we conclude that the collision of the $(i_s)_s$-iterate and the $(j_s)_s$-iterate can be written arithmetically as
        \[
            \rotate_{\theta}^{i_0 q_{[0]}(\theta) + \ldots + i_l q_{[l]}(\theta)}(1) = \rotate_{\theta}^{j_0 q_{[0]}(\theta) + \ldots + j_l q_{[l]}(\theta)}(1)
            \qquad \textnormal{iff} \qquad
            P(\theta) := \sum_{s=0}^m w_s q_{[s]}(\theta) = 0.
        \]
        Our goal is reduced to showing that in $U_{M,l}(\ttheta)$, the function $P(\theta)$ is either non-vanishing or constantly zero.
        If $\abar_j < \infty$ for all $j \in \{1,\ldots,l\}$, then $P(\theta)$ is constant in $U_{M,l}(\ttheta)$ and there is nothing further to be done.
        So let us suppose otherwise.

        Observe that for all $t \in \{1,\ldots, l\}$,
        \begin{align*}
            \left| \sum_{s=0}^{t-1} w_s q_{[s]}(\theta) \right| 
            \leq k \left| \sum_{s=0}^{t-1} q_{[s]}(\theta) \right| \leq k q_{[t]}.
        \end{align*}
        To keep track of the indices that give us infinity, we write
        \[
            \{ s \in \{1,\ldots,l\} \: : \: \abar_s = \infty \} = \{t_1, t_2,\ldots, t_\nu\}
        \]
        where $t_j < t_{j+1}$.
        The recurrence relation allows us to write
        \begin{align*}
            \sum_{s=t_{\nu}-1}^{m} w_s q_{[s]}(\theta) 
            &= X_{\nu} q_{[t_\nu]}(\theta) + Y_{\nu} q_{[t_\nu-1]}(\theta).
        \end{align*}
        for some integers $X_{\nu}$ and $Y_{\nu}$ independent of $\theta \in U_{M,l}(\ttheta)$ whose absolute values are bounded above by some constant $C_\nu > 0$ depending only on $k$, $l$, and $m_l(\ttheta)$.
        Observe that 
        \begin{align*}
            |P(\theta)| 
            &\geq \left| \sum_{s=t_{\nu} -1}^{m} w_s q_{[s]}(\theta) \right| - \left| \sum_{s=0}^{t_{\nu} -2} w_s q_{[s]}(\theta) \right| \\
            &\geq |X_{\nu}| q_{[t_{\nu}]}(\theta) - |Y_{\nu}| q_{[t_{\nu}-1]}(\theta) - k q_{[t_{\nu}-1]}(\theta) \\
            &\geq \Big| |X_\nu| (\abar_{t_\nu} (\theta)-1) - |Y_\nu| - k \Big| \: q_{[t_{\nu}-1]}(\theta).
        \end{align*}
        We will demand that $M \geq C_\nu + k + 2$.
        This would mean that if $U_{\nu} \neq 0$, then 
        \[
            \Big| |X_\nu| (\abar_{t_\nu} (\theta)-1) - |Y_\nu| - k \Big| \geq \abar_{t_\nu}(\theta) - 1 - C - k 
            \geq 1,
        \]
        which means that $P(\theta) \neq 0$ for all $\theta \in U_{M,l}(\ttheta)$.
        Let us suppose otherwise that $X_{\nu}=0$.
        In this case, we can go further down to index $t_{\nu -1}$ and write
        \[
            \displaystyle \sum_{s=t_{\nu-1} -1}^m w_s q_{[s]}(\theta) = X_{\nu-1} q_{[t_{\nu-1}]}(\theta) + Y_{\nu-1} q_{[t_{\nu-1}-1]}(\theta),
        \]
        where again, $X_{\nu-1}$ and $Y_{\nu-1}$ are integers that are independent of $\theta \in U_{M,l}(\theta)$ whose absolute values are bounded above by some constant $C_{\nu-1} > 0$ depending only on $k$, $l$, and $m_l(\ttheta)$.
        By demanding again that $M \geq C_{\nu-1} + k + 2$, if $X_{\nu-1} \neq 0$, we have that $P(\theta)$ is non-vanishing, and if otherwise, we can go down to the next index and repeat the argument.

        If the inductive argument above survives, we arrive at the situation where 
        \[
            \displaystyle \sum_{s=t_1-1}^m w_s q_{[s]}(\theta) = Y_{1} q_{[t_{1}-1]}(\theta)
        \]
        where $Y_1$ is an integer independent of $\theta \in U_{M,l}(\theta)$.
        In this final case, the whole sum $P(\theta)$ would have to be a constant independent of $\theta \in U_{M,l}(\theta)$, and we are done.
    \end{proof}
    
    The claim above implies that $\overline{\ord}(\ttheta) = \left\{ \overline{\ord}_{k,l}(\ttheta) \right\}_{k,l \geq 1}$ is a well-defined continuous extension of $\ord$.
    Next, we will show that $\overline{\ord}$ is injective.
    Pick two distinct elements $\ttheta = \seq{ (\varepsilon_n, \abar_n }_{n \geq 1}$ and $\ttheta' = \seq{ (\varepsilon'_n, \abar'_n }_{n \geq 1}$ of $\TheCpt$.
    There are $2$ cases.
    \vspace{0.1in}

    \noindent \underline{Case 1:} $\varepsilon_l \neq \varepsilon'_l$ for some minimal $l \geq 1$. \\[0.05in]
    Let $\theta, \theta' \in \Irrat$ be irrationals sufficiently close to $\ttheta$, $\ttheta'$ respectively such that $\ord_{1,l}(\theta) = \overline{\ord}_{1,l}(\ttheta)$ and $\ord_{1,l}(\theta') = \overline{\ord}_{1,l}(\ttheta')$.
    In this case, clearly the order of the points $\left\{ 1, \rotate_{\theta}^{[l]}(1) \right\}$ is different from the order of $\left\{ 1, \rotate_{\theta'}^{[l]}(1) \right\}$.
    \vspace{0.1in}

    \noindent \underline{Case 2:} $\varepsilon_n = \varepsilon'_n$ for all $n \geq 1$, and $\abar_l \neq \abar'_l$ for some minimal $l \geq 1$. \\[0.05in]
    Without loss of generality, assume that $\abar_l < \abar'_l$ and let $k = \abar_l$.
    Let $\theta$ and $\theta'$ be irrationals sufficiently close to $\ttheta$ and $\ttheta'$ such that $\ord_{k+1,l}(\theta) = \overline{\ord}_{k+1,l}(\ttheta)$ and $\ord_{k+1,l}(\theta') = \overline{\ord}_{k+1,l}(\ttheta')$.
    Then, the order of points $\left\{ \rotate_{\theta}^{q_{[l-2]}(\theta)+j q_{[l-1]}(\theta)}(1) \right\}_{0 \leq j \leq k+1}$ is different from the order of points $\left\{\rotate_{\theta'}^{q_{[l-2](\theta')}+j q_{[l-1]}(\theta')}(1) \right\}_{0 \leq j \leq k+1}$.
\end{proof}

\subsection{Universal properties of $\TheCpt$}
\label{ss:universality}

The following result is a general theorem in topological dynamics.

\begin{theorem}[Minimal dynamical compactification]
\label{thm:minimal-compactification}
    Consider a topological dynamical system $f: X \to X$ and an embedding $h: X \to Y$ to a compact Hausdorff space $Y$.
    There exists a unique minimal compactification $(\hat{\iota}, \hat{X}, \hat{f})$ of $(X,f)$ such that $h$ admits an $\hat{\iota}$-extension $\hat{h}$ in the following sense.
    Let $(\tilde{\iota}, \tilde{X}, \tilde{f})$ be any compactification of $(X,f)$ such that $h$ admits an $\tilde{\iota}$-extension $\tilde{h}$.
    There exists a continuous surjective map $\phi: \tilde{X} \to \hat{X}$ such that
    \begin{align*}
        \phi \circ \tilde{\iota} = \hat{\iota}, \qquad
        \phi \circ \tilde{f} = \hat{f} \circ \phi, \qquad
        \tilde{h} = \hat{h} \circ \phi.
        \tag{$\dagger$}
    \end{align*}
\end{theorem}

The reference embedding $h$ is important.
Without it, the theorem is false.
This result is likely part of the folklore of topological dynamics. However, in the absence of an explicit reference in the standard literature, we include its proof here for the reader's convenience. 

\begin{proof}
    Consider the compact sequence space $Y^{\N}$, equipped with the infinite product topology.
    Let us embed $X$ into $Y^{\N}$ via the map
    \[
        \Phi: X \to Y^{\N}, \qquad \Phi(\theta) 
        = \seq{ h \circ f^n(x) }_{n \geq 0}. 
    \]
    Since $h$ is an embedding and $Y$ is Hausdorff, then $\Phi$ is a homeomorphism onto its image, i.e. again an embedding.
    It induces a conjugacy between $f: X \to X$ and the shift map on the image $\Phi(X)$.
    Let us set 
    \begin{itemize}
        \item $\hat{X}$ to be the closure of $\Phi(X)$ as a subset of $Y^{\N}$,
        \item $\hat{\iota}: X \to \hat{Y}$ to be equal to $\hat{\iota}(x) = \Phi(x)$,
        \item $\hat{f}: \hat{X} \to \hat{X}$ to be the shift map, and
        \item $\hat{h} : \hat{X} \to Y$ to be the projection to the first coordinate: $\hat{h} ( \seq{ x_n }_{n\geq 0} ) = x_0$.
    \end{itemize}
    It is clear that $(\hat{\iota}, \hat{X}, \hat{f})$ is a compactification of $(X, f)$ and $\hat{h}$ is an $\hat{\iota}$-extension of $h$.

    Next, let us prove the universal property.
    Pick an arbitrary compactification $(\tilde{\iota}, \tilde{X}, \tilde{f})$ such that $h$ admits an $\tilde{\iota}$-extension $\tilde{h}$.
    Consider the map
    \[
        \phi : \tilde{X} \to Y^{\N}, \qquad 
        \phi(z) = \seq{ \tilde{h} \circ \tilde{f}^n(z) }_{n \geq 0}. 
    \]
    Observe that $\phi \circ \tilde{\iota} = \Phi$ and so the image $\phi(\tilde{\iota}(X))$ is dense in $\hat{X}$.
    Since $\tilde{\iota}(X)$ is dense in $\tilde{X}$ and $\phi$ is continuous, then the image of $\phi$ has to be equal to $\hat{X}$.
    Therefore, $\phi$ induces a continuous surjective map $\phi: \tilde{X} \to \hat{X}$.
    It is straightforward to check that the three equations in ($\dagger$) hold.
\end{proof}

\begin{theorem}
\label{thm:universal-eta}
    The unique minimal compactification of $(\Irrat, \gauss)$ that respects the segment embedding $\eta: \Irrat \to [0,1]$ is homeomorphic to $(\bar{\iota}, \TheCpt, \bar{\gauss})$.
\end{theorem}

\begin{proof}
    Consider the minimal compactification $(\hat{\iota}, \hat{\Irrat}, \hat{\gauss})$ of $(\Irrat, \gauss)$ respecting $\eta$ constructed in Theorem \ref{thm:minimal-compactification}. Let $\hat{\eta}$ be the $\hat{\iota}$-extension of $\eta$.
    From Proposition \ref{prop:compact-eta}, we know that the compactification $(\bar{\iota}, \TheCpt, \bar{\gauss})$ respects $\eta$ with $\bar{\iota}$-extension $\bar{\eta}$.
    Consider the continuous surjective map $\phi: \TheCpt \to \hat{\Irrat}$ from Theorem \ref{thm:minimal-compactification}.
    It can be written explicitly as
    \[
        \phi: \TheCpt \to \hat{\Irrat} \subset [0,1]^{\N}, \qquad \phi(\ttheta) = \seq{ \bar{\eta} \circ \bar{\gauss}^n(\ttheta) }_{n \geq 0}.
    \]
    More explicitly, for every element $\ttheta = \seq{ (\varepsilon_n,\abar_n) }_{n\geq 1}$ of $\TheCpt$, we have $\phi(\ttheta) = \seq{ x_n }_{n\geq 0}$ where 
    for all $n \geq 1$,
    \[
        x_{n-1} = \frac{1-\varepsilon_{n}}{2} + [\varepsilon_{n} b_{n}(\ttheta), \varepsilon_{n+1} b_{n+1}(\ttheta), \ldots]_-.
    \]
    We need to show that $\phi$ is a homeomorphism.
    Since $\TheCpt$ is compact, $\hat{\Irrat}$ is Hausdorff, and $\phi$ is surjective, it remains to show that $\phi$ is injective.
    
    Suppose for a contradiction that $\phi$ is not injective. 
    There exist two distinct elements 
    $\ttheta = \seq{ (\varepsilon_n,\abar_n) }_{n\geq 1}$ 
    and 
    $\ttheta' = \seq{ (\varepsilon'_n, \abar'_n ) }_{n\geq 1}$ of $\TheCpt$ such that $\seq{ x_n }_{n \geq 0} = \phi(\ttheta)$ is equal to $\seq{ x'_n }_{n\geq 0} = \phi(\ttheta')$.
    After shifting, we will assume without loss of generality that $(\varepsilon_1, \abar_1) \neq (\varepsilon'_1, \abar'_1)$. 
    There are two cases.
    \begin{itemize}
        \item Suppose $\varepsilon_1 \neq \varepsilon'_1$. 
        Without loss of generality, suppose $\varepsilon_1 = +$ and $\varepsilon_2 = -$.
        Since $x_0=x'_0$, then both $x_0$ and $x'_0$ must be equal to $\frac{1}{2}$.
        Therefore, $\abar_1 = \abar'_1=2$, $(\varepsilon_2, \abar_2) = (-,\infty)$, and $(\varepsilon'_2, \abar'_2) = (+,\infty)$.
        But this would imply that $x_1 = 1$ and $x'_1 = 0$, which are not equal.
        \item Suppose $\varepsilon_1 = \varepsilon'_1$ and $\abar_1 \neq \abar'_1$.
        Without loss of generality, suppose $\varepsilon_1 = +$ and $\abar_1 < \abar'_1$.
        Since $\abar_1 < \infty$, then $x_0$ is a positive real number.
        Since $x_0 = x'_0$, then both $x_0$ and $x'_0$ must be equal to $\frac{1}{k}$ for some integer $k \geq 3$.
        Therefore, $\abar_1 = k-1$, $\abar'_1 = k$, $(\varepsilon_2, \abar_2) = (+,\infty)$, and $(\varepsilon'_2, \abar'_2) = (-,\infty)$.
        But this would again imply that $x_1 = 0$ and $x'_1 = 1$, which are not equal. \qedhere
    \end{itemize}
\end{proof}

\begin{theorem}
\label{thm:universal-ord}
    The unique minimal compactification of $(\Irrat, \gauss)$ that respects the dynamical embedding $\ord$ is homeomorphic to $(\bar{\iota}, \TheCpt, \bar{\gauss})$.
\end{theorem}

\begin{proof}
    Consider the minimal compactification $(\hat{\iota}, \hat{\Irrat}, \hat{\gauss})$ of $(\Irrat, \gauss)$ respecting $\ord$ constructed in Theorem \ref{thm:minimal-compactification}. 
    Let $\hat{\ord}$ be the $\hat{\iota}$-extension of $\ord$.
    From Proposition \ref{prop:compact-ord}, we know that the compactification $(\bar{\iota}, \TheCpt, \bar{\gauss})$ respects $\ord$ with $\bar{\iota}$-extension $\bar{\ord}$.
    Consider the continuous surjective map $\phi: \TheCpt \to \hat{\Irrat}$ from Theorem \ref{thm:minimal-compactification}.
    It can be written explicitly as
    \[
        \phi: \TheCpt \to \hat{\Irrat} \subset \Omega^{\N}, \qquad \phi(\ttheta) = \seq{ \overline{\ord} \circ \bar{\gauss}^n(\ttheta) }_{n \geq 0}.
    \]
    To prove that $\phi$ is a homeomorphism, we just need to show that $\phi$ is injective.
    But this follows directly from the fact that $\overline{\ord}$ is an embedding (Proposition \ref{prop:compact-ord}).
\end{proof}

The three theorems above imply that that taking $\ord$ into account is equivalent to taking $\eta$ into account.

\begin{corollary}
    Every compactification of $(\Irrat, \gauss)$ that respects $\eta: \Irrat \to [0,1]$ also respects $\ord: \Irrat \to \Omega$, and vice versa.
\end{corollary}

\begin{proof}
    Let $(\tilde{\iota}, \tilde{\Irrat}, \tilde{\gauss})$ be a compactification such that $\eta$ admits an $\tilde{\iota}$-extension $\tilde{\eta}$.
    By Theorems \ref{thm:minimal-compactification} and \ref{thm:universal-eta}, there exists a continuous surjective map $\phi: \tilde{\Irrat} \to \TheCpt$ such that
    \[
        \phi \circ \tilde{\iota} = \bar{\iota}, \quad
        \phi \circ \tilde{\gauss} = \bar{\gauss} \circ \phi, \quad
        \tilde{\eta} = \bar{\eta} \circ \phi.
    \]
    Let us define $\widetilde{\ord}: \tilde{\Irrat}\to \Omega$ by $\widetilde{\ord} = \overline{\ord} \circ \phi$.
    It satisfies 
    \[
        \widetilde{\ord} \circ \tilde{\iota} = \overline{\ord} \circ \phi \circ \tilde{\iota} = \bar{\ord} \circ \bar{\iota} = \ord.
    \]
    Hence, $\widetilde{\ord}$ is indeed an $\tilde{\iota}$-extension of $\ord$.
    The converse is analogous.
\end{proof}

Altogether the results above imply Theorem \ref{main-thm-2}.
In contrast to the previous two theorems, we have:

\begin{proposition}
\label{prop:multiplier-failure}
    The unique minimal compactification of $(\Irrat, \gauss)$ that respects the modular embedding $\mu: \Irrat \to \T$ is \underline{not} homeomorphic to $(\bar{\iota}, \TheCpt, \bar{\gauss})$.
    In particular, the map
    \[
        \TheCpt \to \T^{\N}, \qquad \ttheta \mapsto \seq{ \bar{\mu}(\shift^n(\ttheta))}_{n\geq 0}
    \]
    is not injective.
\end{proposition}

\begin{proof}
    Here is one example.
    Let $\ttheta \in \TheCpt$ be of the form $\seq{(\varepsilon_1, \infty), (\varepsilon_2, \infty), (\varepsilon_3, \infty), \ldots}$.
    Regardless of the value of $\varepsilon_n$'s, we have that for all $n \geq 0$, $\bar{\mu}(\shift^n(\ttheta)) \equiv 0$ (mod $1$).
\end{proof}


\subsection{The continuant group}
\label{ss:continuant-1}

For $\ttheta = \seq{ (\varepsilon_n,\abar_n) }_{n\geq 1} \in \TheCpt$, denote the set of parabolic levels by
\[
    \mathcal{P}_{\ttheta} := \{n \geq 0 \: : \: \abar_{n+1} = \infty \}.
\]
We will also denote $\mathcal{P}^*_{\ttheta} := \mathcal{P}_{\ttheta} \cup \{-1\}$.

\begin{definition}
\label{def:time-group}
For $\ttheta = \seq{ (\varepsilon_n,\abar_n) }_{n\geq 1} \in \TheCpt$, we define an additive abelian group $\Kont_{\ttheta}$, called the \emph{continuant group of $\ttheta$}, such that it is generated by the sequence $\{\qq_{[n]}\}_{n \geq 0}$ subject to the relations
\begin{align}
\label{eqn:relation-continuant}
    \qq_{[n+1]} = b_{n+1}(\ttheta) \, \qq_{[n]} - \varepsilon_{n} \varepsilon_{n+1}  \, \qq_{[n-1]} 
    \quad
    \textnormal{ for } n \in \N \backslash \mathcal{P}_{\ttheta},
\end{align}
where $\qq_{[-1]} = 0$ is the identity element.
For $n \geq 0$, the element $\qq_{[n]} = \qq_{[n]}(\ttheta) \in \Kont_{\ttheta}$ will be referred to as the \emph{$n$\textsuperscript{th} return time} of $\ttheta$.
\end{definition}

The relation expresses $\qq_{[n+1]}$, for every non-parabolic level $n$, in terms of $\qq_{[n]}$ and $\qq_{[n-1]}$.
These generators can be eliminated, leading to:

\begin{lemma}
    \label{lem:free-group-1}
    For $\ttheta \in \TheCpt$, the continuant group $\mathcal{K}_{\ttheta}$ is a free abelian group with basis $\{\qq_{[n+1]}(\ttheta) \}_{n \in \mathcal{P}^*_{\ttheta}}$. 
\end{lemma}

In particular,
if $\ttheta$ is irrational, then $\qq_{[0]} \mapsto 1$ induces an isomorphism between $\Kont_{\ttheta}$ and the group of integers $\Z$ that sends each $\qq_{[n]}$ to the natural number $q_{[n]}$ coming from \S\ref{ss:more-notation}.

\subsubsection{Conversion}
\label{sss:conversion-1}

In practice, sometimes it is more convenient to consider a different set of generators $\{\qq_n = \qq_n(\ttheta)\}_{n\geq 0}$ for $\Kont_{\theta}$.
It mimics the denominator continuants of the regular continued fraction expansion and it is defined together with an increasing sequence of integers $\{t_n\}_{n \geq 0}$ and a sequence $\{a_n\}_{n \geq 1}$ of elements of $\overline{\N}$ as follows.
First, set 
\[
t_0 = 0, \qquad a_1 = \abar_1, \quad \text{and} \quad \qq_{0} = \qq_{[0]}.
\]
The rest is defined inductively as follows. For $n \geq 0$,
\begin{itemize}
    \item if $\varepsilon_{n+1}\varepsilon_{n+2} = -1$, then 
    \[
        t_{n+1} = t_n +1, \quad a_{t_n+2} = \abar_{n+2}, \quad \text{and} \quad \qq_{t_{n+1}} = \qq_{[n+1]};
    \]
    \item if $\varepsilon_{n+1}\varepsilon_{n+2} = +1$, then 
    \begin{gather*}
        t_{n+1} = t_n +2, \qquad  a_{t_n+2} = 1, \qquad a_{t_n+3} = \abar_{n+2}-1, \\
        \qq_{t_n+1} = \qq_{[n+1]} - \qq_{[n]}, \qquad  \text{and} \qquad \qq_{t_{n+1}} = \qq_{[n+1]}.
    \end{gather*}
\end{itemize}
We will call $\qq_n = \qq_n(\ttheta)$ the $n$\textsuperscript{th} \emph{slow return time} of $\ttheta$.
The group relations become slightly simpler, namely
\[
    \qq_{n+1} = a_{n+1} \qq_n + \qq_{n-1}  \qquad \text{ for } n \geq 0 \text{ with } a_{n+1} < \infty,
\]
where $\qq_{-1} = 0$ is set to be the identity for convenience.

\subsubsection{The time order}
\label{sss:time-semigroup}

\begin{proposition}[Time order]
\label{prop:chronological-order-01}
    For every $\ttheta \in \TheCpt$, the continuant group $\Kont_{\ttheta}$ admits a unique translation-invariant total order $<$, called the time order of $\Kont_{\ttheta}$, with the property that for all $n \geq 0$,
    \begin{enumerate}
        \item \label{item-T1} $0 < \qq_{[n]} < \qq_{[n+1]}$,
        \item \label{item-T2} for $n \in \mathcal{P}_{\ttheta}$, we have $k \qq_{[n]} < \qq_{[n+1]}$ for all $k \geq 1$.
    \end{enumerate}
\end{proposition}

In the proof, we will use the following classical result.

\begin{lemma}
\label{lem:lexicograpihc}
    Let $(S,<)$ be a totally ordered set and let $G$ be the free abelian group with basis $S$.
    There exists a unique translation-invariant total order $<$ on $G$, called the \emph{lexicographic order} induced by $S$, such that for any element $x = \sum_{j=1}^k c_j s_j$ of $G$ where $c_1,\ldots, c_k \in \Z \backslash \{0\}$ and $s_1,\ldots,s_k \in S$ with $s_1<\ldots<s_k$, then $x > 0$ if and only if $c_k > 0$.
\end{lemma}

\begin{proof}[Proof of Proposition \ref{prop:chronological-order-01}]
    Recall from Lemma \ref{lem:free-group-1} that $\mathcal{K}_{\ttheta}$ is the free abelian group generated by the basis $\{\qq_{[n+1]}\}_{n \in \mathcal{P}^*_{\ttheta}}$.
    The standard order of integers induces an order on $\mathcal{P}^*_{\ttheta}$.
    Let $<$ be the lexicographic order on $\mathcal{K}_{\ttheta}$ induced by $\{\qq_{[n+1]}\}_{n \in \mathcal{P}^*_{\ttheta}}$.
    It is clear that $<$ satisfies property \ref{item-T2}.
    
    To prove property \ref{item-T1}, we need that $\qq_{[n]} < \qq_{[n+1]}$ for all $n \geq 0$.
    Fix an integer $N$ such that $N-1 \in \mathcal{P}^*_{\ttheta}$.
    If $N$ is in $\mathcal{P}_{\ttheta}$ too, then being lexicographic directly implies that $\qq_{[N]} < \qq_{[N+1]}$.
    Else, let $\mathcal{I} \subset \N$ be the maximal sequence of consecutive non-parabolic levels starting from $N$.
    If $\mathcal{I}$ has a maximum, we set $\mathcal{I}' = \mathcal{I} \cup \{\max \mathcal{I}\}$; else, we set $\mathcal{I}' = \mathcal{I}$.
    Let us define an sequence of positive integers $\{\alpha_t\}_{t \in \{N-1\} \cup \mathcal{I}'}$ recursively by
    \[
        \alpha_{N-1} = 0, \qquad \alpha_{N} = 1, \qquad
        \alpha_t = b_{t} \alpha_{t-1} - \varepsilon_{t-1} \varepsilon_{t} \alpha_{t-2} \textnormal{ for } t > N.
    \]
    This sequence is increasing by virtue of Lemma \ref{lem:p[n]-q[n]} (5).
    The relation (\ref{eqn:relation-continuant}) implies that for $t \in \mathcal{I}'$, $\qq_{[t]} - \alpha_t \qq_{[N]}$ is either $0$ if $N=0$, or a linear combination of $\qq_{[n]}$, $n < N$ if $N \geq 1$. 
    This implies \ref{item-T1}.

    Let $\blacktriangleleft$ be another total order satisfying \ref{item-T1} and \ref{item-T2}.
    Consider any lexicographically positive element $x = \sum_{j=1}^k c_j \qq_{[n_j+1]}$ where $c_1,\ldots,c_k \in \Z$ with $c_k > 0$ and $n_1 < \ldots < n_k$ is an increasing sequence in $\mathcal{P}^*_{\ttheta}$.
    If $k=1$, then $x= c_1 \qq_{[n_1+1]}$ and property \ref{item-T1} would imply that $0 \blacktriangleleft x$.
    Otherwise, $n_k$ is a parabolic level and properties \ref{item-T1} and \ref{item-T2} imply that
    \[
        x \blacktriangleright c_k \qq_{[n_k+1]} - \left( \sum_{j=1}^{k-1} |c_j| \right) \qq_{[n_k]} \blacktriangleright 0.
    \]
    Therefore, since the positive cone of $\blacktriangleleft$ contains the positive cone of $<$, then $\blacktriangleleft$ is the same order as $<$.
\end{proof}

The positive cone of $<$ is the semigroup
\[
    \Time_{\ttheta} := \{P \in \Kont_{\ttheta} \: : \: P > 0 \},
\] 
which we refer to as the \emph{time semigroup} of $\ttheta$.
Similar to Lemma \ref{lem:free-group-1}, $\Time_{\theta}$ has a set of generators of the form
\[
    \{\qq_{0}\} \cup \bigcup_{a_{n+1}=\infty} \{\qq_{n+1} - k \qq_n \}_{k \geq 1}. 
\]

\section{The continuant group and topological cascades}
\label{sec:natural-extension}

Consider the bi-infinite sequence space
\[
    \TheBiCpt := \overline{\Sigma}^{\Z} = 
    \left\{ \seq{ (\varepsilon_n,\abar_n) }_{n\in \Z} \: : \: 
    (\varepsilon_{n},\abar_n) \in \overline{\Sigma} \text{ for all } n \right\},
\]
equipped with the infinite product topology. The projection map
\[
    \textnormal{proj}: \TheBiCpt \to \TheCpt, \quad 
    \seq{ \ldots, (\varepsilon_{-1}, \abar_{-1}), (\varepsilon_0, \abar_0); (\varepsilon_1, \abar_1), \ldots } 
    \mapsto 
    \seq{ (\varepsilon_1, \abar_1), (\varepsilon_2, \abar_2), \ldots }
\]
is continuous and the shift map $\shift: \TheBiCpt \to \TheBiCpt$ given by
\[
\shift( \seq{ \ldots, (\varepsilon_{-1}, \abar_{-1}), (\varepsilon_0, \abar_0); (\varepsilon_1, \abar_1), \ldots } ) = \seq{ \ldots, (\varepsilon_{0}, \abar_{0}), (\varepsilon_1, \abar_1); (\varepsilon_2, \abar_2), \ldots }
\]
is a homeomorphism.
Indeed, $\shift: \TheBiCpt \to \TheBiCpt$ is the natural extension of $\shift: \TheCpt \to \TheCpt$.
Consider also the following shift-invariant subset of $\TheBiCpt$:
\[
    \TheBi := \left\{ \seq{ (\varepsilon_n, \abar_n) }_{n \in \Z} \in \TheBiCpt \: : \: \abar_n < \infty \text{ for all }n \right\}.
\]

\begin{definition}
    An element $\tttheta$ of $\TheBiCpt$ is called \emph{irrational} if it is contained in $\TheBi$, \emph{enriched rational} if otherwise.
\end{definition}

In this final section, we discuss the \emph{continuant group} $\Kbold_{\tttheta}$ associated to an element $\tttheta$ of $\TheBiCpt$ which takes the role of continuants associated to an irrational.
We show that $\Kbold_{\tttheta}$ can be equipped with two natural orders called \emph{time order} $<$ and \emph{space order} $\lhd$.
We will prove Theorem \ref{main-thm-3} on the existence of a unique group action $A_{\tttheta} = (A_{\tttheta}^P)_{P \in \Kbold_{\tttheta}}$ by $\Kbold_{\tttheta}$ on $\R$, called an \emph{invertible topological cascade}, which respects the space order. 
Lastly, we discuss how for every irrational element $\tttheta$, we can construct $A_{\tttheta}$ to consist of translations and naturally connects with a bi-infinite tower of sector renormalizations of irrational rotations.

\subsection{The continuant group}

For every $\tttheta = \seq{ (\varepsilon_n, \abar_n) }_{n \in \Z} \in \TheBiCpt$, we will denote the set of parabolic levels by
\[
    \mathbf{P}_{\tttheta} := \{ n \in \Z \: : \: \abar_{n+1} = \infty \},
\]
and for $n \in \Z \backslash \mathbf{P}_{\tttheta}$, denote
\[
    b_{n+1}(\tttheta) := b_{n+1} = \abar_{n+1} + \frac{1+\varepsilon_{n+1} \varepsilon_{n+2}}{2}.
\]

\begin{definition}
For every $\tttheta = \seq{ (\varepsilon_n, \abar_n) }_{n \in \Z} \in \TheBiCpt$, we again define an additive abelian group $\Kbold_{\tttheta}$, called the (\emph{bi-infinite}) \emph{continuant group} of $\tttheta$, 
such that it is generated by the sequence $\{Q_{[n]}\}_{n\in\Z}$ and subject to the relations
\[
    Q_{[n+1]} = b_{n+1} Q_{[n]} - \varepsilon_{n} \varepsilon_{n+1}  Q_{[n-1]} 
    \qquad \text{ for } n \in \Z \backslash \mathbf{P}_{\tttheta}. 
\]
For every $n \in \Z$, $Q_{[n]}= Q_{[n]}(\tttheta)$ is called the \emph{$n$\textsuperscript{th} return time} of $\tttheta$.
\end{definition}

Below, we will define another generating set $\{Q_n=Q_n(\tttheta)\}_{n\in\Z}$ for $\Kbold_{\tttheta}$; each $Q_n$ will be called the \emph{$n$\textsuperscript{th} slow return time of} $\tttheta$.
To define them, we will employ a strictly increasing sequence of integers $\{t_n = t_n(\tttheta)\}_{n \in \Z}$ and a sequence $\{a_n = a_n(\tttheta)\}_{n\in\Z}$ of elements of $\overline{\N}$ as follows.

First, we set $t_0 = 0$. 
For all $n \geq 1$, $t_n$ and $t_{-n}$ are defined recursively by
\[
    t_n = \begin{cases}
        t_{n-1}+1 & \text{ if } \varepsilon_{n} \varepsilon_{n+1} = -,\\
        t_{n-1}+2 & \text{ if } \varepsilon_{n} \varepsilon_{n+1} = +,
    \end{cases}
\]
and 
\[
    t_{-n} = \begin{cases}
        t_{-n+1}-1 & \text{ if } \varepsilon_{-n+1} \varepsilon_{-n+2} = -,\\
        t_{-n+1}-2 & \text{ if } \varepsilon_{-n+1} \varepsilon_{-n+2} = +.
    \end{cases}
\]
For all $n \in \Z$,
\begin{itemize}
    \item if $\varepsilon_n \varepsilon_{n+1} = -$, set $a_{t_n + 1} = \abar_{n+1}$;
    \item if $\varepsilon_n \varepsilon_{n+1} = +$, set $a_{t_n} = 1$ and $a_{t_n + 1} = \abar_{n+1}-1$.
\end{itemize}
(Here, we make the convention that $\infty + k  = \infty$ for any integer $k$.)
For all $n \in \Z$, we set
\[
    Q_{t_n} := Q_{[n]},
\]
and whenever $\varepsilon_n \varepsilon_{n+1} = +$, we set 
\[
    Q_{t_n-1} := Q_{[n]}-Q_{[n-1]}.
\]
Denote $\mathbf{P}'_{\tttheta} = \{n \in \Z \: : \: a_{n+1} = \infty \}$, the slow variant of $\mathbf{P}$.
The corresponding set of relations for this new generating set is simpler:
\[
    Q_{m+1} = a_{m+1} Q_{m} + Q_{m-1} \qquad \textnormal{ when } m \notin \mathbf{P}'_{\tttheta}.
\]

The conversion above can be reversed assuming that the initial orientation has been specified.

\begin{proposition}
    The map
    \[
        \TheBiCpt \to \overline{\N}^{\Z}, \quad \tttheta \mapsto \{a_n(\tttheta)\}_{n\in\Z}
    \]
    is a double covering map where every fiber is of the form 
    \[
        \{ \seq{(\varepsilon_n, \abar_n)}_{n\in\Z}, \seq{(-\varepsilon_n, \abar_n)}_{n\in\Z}\}.
    \]
\end{proposition}

\begin{proof}
    The continuity is an elementary routine exercise: the preimage of a cylinder set of $\overline{\N}^{\N}$ is a union of cylinder sets of $\TheBiCpt$.
    The map is surjective because the conversion above can be reversed.
    From the algorithm above, the $a_n$'s determine $\delta_n = \varepsilon_n \varepsilon_{n+1} \in \{-,+\}$, the $t_n$'s, and the $\abar_n$'s.
    However, $\varepsilon_1$ cannot be determined from the $a_n$'s, resulting in the map having degree two.
\end{proof}

For the rest of this section, we will fix
\[
    \tttheta = \seq{ (\varepsilon_n, \abar_n) }_{n \in \Z} \in \TheCpt.
\]
Denote $\Kbold = \Kbold_{\tttheta}$, $\mathbf{P} = \mathbf{P}_{\tttheta}$, $\mathbf{P}' = \mathbf{P}'_{\tttheta}$, $a_n = a_n(\tttheta)$, $Q_{[n]} = Q_{[n]}(\tttheta)$, and $Q_n = Q_{n}(\tttheta)$.

Similar to Lemma \ref{lem:free-group-1}, we have:

\begin{lemma}
\label{lem:free-group-2}
    For any $k \in \Z$, $\Kbold$ is a free abelian group with basis
    \begin{align}
    \label{eqn:basis-abelian-group}
        \{Q_{k-1}, Q_k \} 
        \; \cup \; 
        \{Q_{n+1}\}_{n \in \mathbf{P}', n \geq k} 
        \; \cup \; 
        \{Q_{n-1}\}_{n \in \mathbf{P}', n < k}.
    \end{align}
    In particular, $\Kbold$ has rank $|\mathbf{P}| + 2$.
\end{lemma}

\begin{proof}
    Fix $k \in \Z$. 
    Let $\mathbf{F}$ be the free abelian group with basis (\ref{eqn:basis-abelian-group}), and let $\iota: \mathbf{F} \to \Kbold$ be the homomorphism that sends each basis element to itself.
    Define the function $\tau: \{Q_n\}_{n\in\Z} \to \mathbf{F}$ inductively as follows.
\begin{itemize}
    \item If $Q_n$ is in (\ref{eqn:basis-abelian-group}), we set $\tau(Q_n) := Q_n$.
    \item For $n\geq k$ with $n \notin \mathbf{P}'$, set $\tau(Q_{n+1}):= a_{n+1}\tau(Q_{n})+\tau(Q_{n-1})$.
    \item For $n \leq k-1$ with $n \notin \mathbf{P}'$, set $\tau(Q_{n-1}) := \tau(Q_{n+1}) - a_{n+1}\tau(Q_{n})$.
\end{itemize}
Naturally, $\tau$ extends to a homomorphism from $\Kbold$ to $\mathbf{F}$. 
Both $\iota\circ\tau$ and $\tau\circ\iota$ are the identity on the generators.
Thus, $\iota$ is an isomorphism.
\end{proof}

\subsection{The time order}
\label{ss:chronological-order}

\begin{proposition}
\label{prop:chronological-order-02}
    The continuant group $\Kbold$ admits a unique translation-invariant total order $<$, called the time order of $\Kbold$, with the property that for all $n \in \Z$,
    \begin{enumerate}[label = \textnormal{(T\arabic*)}]
        \item\label{time-1} $0 < Q_{n}$,
        \item\label{time-2} if $n \in \mathbf{P}'$, then $k Q_{n} < Q_{n+1}$ for all $k \geq 1$.
    \end{enumerate}
\end{proposition}

\begin{proof}
    Following criteria \ref{time-1} and \ref{time-2}, consider the sub-semigroup $\Tbold$ of $\Kbold$ generated by $\{Q_n\}_{n \in \Z}$ and $\{Q_{n+1}-kQ_n\}_{n \in \mathbf{P}',k \geq 1}$.
    To prove the existence and uniqueness of the time order, we need to show that $\Tbold$ has to be the positive cone, i.e.
    \begin{enumerate}
        \item[(a)] $\Kbold = \Tbold \cup \{0\} \cup -\Tbold$, and
        \item[(b)] the subsets $\Tbold$, $\{0\}$, and $-\Tbold$ are disjoint.
    \end{enumerate}

    To prove (a), we need the following properties.
    \vspace{0.1in}

    \noindent \underline{Claim 1:} For every $m,n \in \Z$ with $m > n$, $Q_{m}-Q_n$ is in $\Tbold$.

    \begin{proof}
        If $n \in \mathbf{P}'$, then $Q_{n+1}-Q_n$ is trivially in $\Tbold$.
        Otherwise, since both $Q_n$ and $Q_{n-1}$ are in $\Tbold$, then so is $Q_{n+1}-Q_n = (a_{n+1}-1)Q_n + Q_{n-1}$.
        Then, $Q_m - Q_n$ is in $\Tbold$ because it is a finite sum of elements of $\Tbold$ namely $Q_{t+1}-Q_t$ for $t \in \{n, \ldots, m-1\}$.
    \end{proof}

    \noindent \underline{Claim 2:}
    For all $n, c, d\in\Z$, $cQ_n+dQ_{n-1}$ is in $\Tbold \cup \{0\} \cup -\Tbold$.

    \begin{proof} 
    Let $X = c Q_n - d Q_{n-1}$. 
    We will perform induction on $c+d$ simultaneously for all $n$. 
    The case when $c=0$ or $d=0$ or $cd <0$ is trivial. Let us assume $c,d > 0$.
    \begin{itemize}
        \item If $c \geq d$, then $X = c(Q_n-Q_{n-1}) + (c-d)Q_{n-1} \in \Tbold$ by Claim 1.
        \item If $n-1\in\mathbf{P}'$, then $X = (c-1)Q_n + (Q_n - dQ_{n-1})\in\Tbold$.
        \item Otherwise, $c < d$ and $n-1\notin\mathbf{P}'$. We can write
        \[
          X = - (c'Q_{n-1} - d'Q_{n-2}), \qquad \text{ where } \; c'=d-a_nc,\; d'= c .
        \]
        If $c'\leq 0$, then $X$ is clearly in $\Tbold$. 
        Else,
        $0 < c'+d'=d-(a_n-1)c < c+d$, and the induction hypothesis applies.
        \qedhere
    \end{itemize}
\end{proof}

    Pick any non-zero element $X = c_{k-1} Q_{k-1} + c_{k} Q_{k} + \ldots + c_l Q_l$ for some $k, l \in \Z$ with $k \leq l$ and integer coefficients $c_{k-1},\ldots, c_l$.
    By Lemma \ref{lem:free-group-2}, there exist unique integers $\delta_{k-1}$, $\delta_k$, and $\delta_{n+1}$ for $n  \in \mathbf{P}'$, $n \geq  k$ such that
    \[
        X = \delta_{k-1} Q_{k-1} + \delta_k Q_k + \sum_{n  \in \mathbf{P}',n \geq  k} \delta_{n+1} Q_{n+1}.
    \]
    If $\delta_{n+1} = 0$ for all $n \in \mathbf{P}'$ with $n \geq k$, then $X = \delta_{k-1} Q_{k-1} + \delta_k Q_k$ and by Claim 2, $X$ has to be in $\Tbold \cup -\Tbold$.
    Otherwise, let $n_1, n_2, \ldots, n_s$ be the elements of the set $\{ n \in \mathbf{P}'+1 \: : \: \delta_{n} \neq 0, \, n > k \}$, written in increasing order.
    By replacing $X$ with $-X$ if needed, assume $\delta_{n_s} > 0$.
    Write $\sigma_j = \textnormal{sgn}(\delta_j) \in \{+1,-1\}$ and $\delta = |\delta_{k-1}| + |\delta_k| + |\delta_{n_1}| + \ldots + |\delta_{n_{s-1}}| \geq 0$. 
    By Claim 1, we can write $X$ as a sum of elements of $\Tbold$:
        \begin{align*}
            X 
            &= (\delta_{n_s}-1) Q_{n_s} + (Q_{n_s} - \delta Q_{n_s-1}) + \sum_{j \in \{k-1,k,n_1,\ldots, n_{s-1}\}} |\delta_{j}| (Q_{n_s-1} + \sigma_{j} Q_{j}).
        \end{align*}
    Thus, $X$ is in $\Tbold$, and we are done with (a).

    To prove (b), we will use the group $\R(\!(x^{-1})\!)$ of formal series having real coefficients and a finite number of positive-power terms.
    For any non-zero element $y = \sum_{k \leq N} c_k x^k$ of $\R(\!(x^{-1})\!)$, we define the \emph{degree} of $y$ to be $\deg y = \max \{ k \: : \: c_k\neq 0 \}$ and the \emph{value} of $y$ to be $\textnormal{val}(y):=c_{\deg y}$.
    \vspace{0.1in}

    \noindent \underline{Claim 3:} There exists a homomorphism $\phi: \Kbold \to \R(\!(x^{-1})\!)$ such that
    \begin{enumerate}
        \item[(i)] $\textnormal{val} \phi(Q_m) > 0$ for all $m \in \Z$, and
        \item[(ii)] $\textnormal{val} \phi(Q_{m+1} - k Q_{m}) > 0$ for all $m \in \mathbf{P}'.$
    \end{enumerate}

    \begin{proof}
    For any $n\in\mathbf{P}'$, define the sequence of integers $\{c_{n,m} \}_{m\in \Z}$ by
    \[
        c_{n,m+1} := \begin{cases}
            0 & \textnormal{ if } m < n \textnormal{ or } m >n \textnormal{ and } m \in \mathbf{P}' , \\
            1 & \textnormal{ if } m=n, \\
            a_{m+1} c_{n,m} + c_{n,m-1} & \textnormal{ if } m > n \textnormal{ and } m \notin \mathbf{P}'.
        \end{cases}
    \]
    If $\mathbf{P}'$ is unbounded from below, we set $\mathbf{m} := 0$ and $\rho_m = 0$ for all $m \in \Z$. 
    Else, let $\mathbf{m} = 0$ and $\mathbf{U} = \Z$ if $\mathbf{P}'$ is empty, $\mathbf{m} = \min \mathbf{P}'$ and $\mathbf{U} = \{n\in \Z \: : \: n \leq \mathbf{m} \}$ if otherwise.    
    Define the sequence of real numbers $\{\rho_m\}_{m \in \Z}$ as follows.
    For $m \in \mathbf{U}$, we define the irrational number
    \[
        r_m :=[a_m,a_{m-1},a_{m-2},\dots]_+ \in (0,1)
    \] 
    coming from the past regular continued fraction.
    It satisfies $r_m = 1/(a_m + r_{m-1})$. 
    For $m\in \Z$, we set
    \[
        \rho_m := \begin{cases}
            \prod_{j=m+1}^{\mathbf{m}} r_j & \textnormal{ if } m < \mathbf{m}, \\
            1 & \textnormal{ if } m = \mathbf{m}, \\
            0 & \textnormal{ if } m > \mathbf{m} \text{ and } m-1 \in \mathbf{P}', \\
            a_m \rho_{m-1} + \rho_{m-2} & \textnormal{ if } m > \mathbf{m} \text{ and } m-1 \notin \mathbf{P}'. 
        \end{cases}
    \]
    By design, $\rho_m > 0$ for all $m \in \mathbf{U}$ and 
    \begin{align}
    \label{eqn:rho-chain}
        \rho_{m+1} = a_{m+1} \rho_{m} + \rho_{m-1} \qquad \text{ for all }
        m \notin \mathbf{P}'.
    \end{align}
    For $m \in \Z$, we set
    \[
        \phi(Q_m) = \rho_m x^{\mathbf{m}-1} + \sum_{n \in \mathbf{P}', n\leq m} c_{n,m} x^n.
    \]
    
    Equation (\ref{eqn:rho-chain}) and the way the coefficients $c_{n,m}$ are defined imply that $\phi(Q_{m+1}) = a_{m+1} \phi(Q_{m}) + \phi(Q_{m-1})$ whenever $m \notin \mathbf{P}'$.
    Therefore, $\phi$ extends to a well-defined homomorphism from $\Kbold$ to $\R(\!(t^{-1})\!)$.
    Fix $m \in \Z$.
    If there exists a largest integer $n$ in $\mathbf{P}'$ such that $n < m$, then $\phi(Q_m)$ has degree $n$ and positive value $c_{n,m} > 0$; otherwise, $\deg \phi(Q_m) = \mathbf{m}-1$ and $\text{val} \phi(Q_m) = \rho_m >0$.
    If $m \in \mathbf{P}'$, the leading term of $\phi(Q_{m+1}-kQ_{m})$ is $x^{m+1}$ with coefficient $c_{m+1,m+1} - kc_{m+1,m} = 1 > 0$.
    Therefore, properties (i) and (ii) are satisfied.
    \end{proof}
    
    Let $\phi$ be from Claim 3. 
    Elements of $\phi(\Tbold)$ have positive values whereas elements of $\phi(\Tbold)$ have negative values. 
    Therefore, $\Tbold$ and $-\Tbold$ are disjoint, and (b) follows.
\end{proof}

The positive cone of the time order $<$ is the semigroup
\[
    \Tbold = \Tbold_{\tttheta} := \{P \in \Kbold_{\ttheta} \: : \: P > 0 \},
\] 
which we will refer to as the \emph{time semigroup} of $\tttheta$.

For every $n \in \Z$, we denote
\begin{itemize}
    \item $\Kbold^{\geq n} := $ subgroup of $\Kbold$ generated by $\{Q_k\}_{k \geq n}$;
    \item $\Kbold^{\leq n} := $ subgroup of $\Kbold$ generated by $\{Q_k\}_{k \leq n}$.
\end{itemize}
The following lemma follows directly from Lemma \ref{lem:free-group-2}.

\begin{lemma}
\label{lem:intersecting-archimedean}
    For $n \in \Z$, the intersection of $\Kbold^{\leq n}$ and $\Kbold^{\geq n-1}$ is the subgroup of $\Kbold$ generated by $Q_n$ and $Q_{n-1}$.
\end{lemma}

Observe that $\Kbold^{\leq n} \cap \Tbold$ is a sub-semigroup of $\Tbold$ consisting of elements of $\Tbold$ in the $<$-Archimedean class of $Q_k$ for some $k \leq n$.
It has the following convexity property with respect to $<$; the proof is elementary.

\begin{lemma}
\label{lem:archimedean-convexity}
    For $n \in \Z$ and $P, Q \in \Tbold$, 
    \[
        P < Q \quad \text{and} \quad Q \in \Kbold^{\leq n}
        \quad \xRightarrow[]{\qquad} \quad
        P \in \Kbold^{\leq n}.
    \]
\end{lemma}

\subsection{The space order}
\label{ss:left-right-order}

\begin{proposition}
\label{prop:left-right-order}
    The continuant group $\Kbold$ admits a unique translation invariant total order $\lhd$, called the space order, that satisfies the following properties.
    \begin{enumerate}[label=\textnormal{(S\arabic*)}]
        \item\label{space-1} $\varepsilon_1 Q_n \lhd 0$ if $n$ is even, and $0 \lhd \varepsilon_1 Q_n$ if $n$ is odd.
        \item\label{space-2} For $n \in \mathbf{P}'$, $-k Q_n$ is always $\lhd$-between $0$ and $Q_{n-1}$ for all $k \geq 1$.
    \end{enumerate}
\end{proposition}

The dynamical meaning of $\lhd$ will be clarified in Theorem \ref{thm:real-ordering}.

\begin{figure}
    \centering
    \begin{tikzpicture}[scale=0.95]
        \draw [gray] (-7.5,0) -- (5.5,0);
        \node [black] at (-7,0) {$\bullet$};
        \node [black] at (-5,0) {$\bullet$};
        \node [black] at (-3.5,0) {$\bullet$};
        \node [black] at (-2,0) {$\bullet$};
        \node [black] at (1,0) {$\bullet$};
        \node [black] at (3,0) {$\bullet$};
        \node [black] at (5,0) {$\bullet$};
    
        \node [black] at (-7,-0.4) {\scalebox{0.85}{$Q_{n}$}};
        \node [black] at (-5,-0.4) {\scalebox{0.85}{$0$}};
        \node [black] at (-3.5,-0.4) {\scalebox{0.85}{$Q_{n+1}$}};
        \node [black] at (-2,-0.4) {\scalebox{0.85}{$Q_{n+1}-Q_n$}};
        \node [black] at (1,-0.4) {\scalebox{0.85}{$Q_{n-1}+2Q_n$}};
        \node [black] at (3,-0.4) {\scalebox{0.85}{$Q_{n-1}+Q_n$}};
        \node [black] at (5,-0.4) {\scalebox{0.85}{$Q_{n-1}$}};

        \draw [blue,-latex] (4.8,0.5) .. controls (4.3,1) and (3.7,1) .. (3.2,0.5);
        \draw [blue,-latex] (2.8,0.5) .. controls (2.3,0.9) and (1.7,0.9) .. (1.2,0.5);
        \draw [blue,-latex] (0.8,0.5) .. controls (0.6,0.75) and (0.2,0.75) .. (0,0.5);
        \node [blue] at (-0.5,0.5) {$\ldots$};
        \draw [blue,-latex] (-1,0.5) .. controls (-1.2,0.75) and (-1.6,0.75) .. (-1.8,0.5);
        \draw [blue,-latex] (-2.2,0.5) .. controls (-2.5,0.85) and (-3,0.85) .. (-3.3,0.5);
        \node [blue] at (2,1.1) {\scalebox{0.85}{$+Q_n$}};
    \end{tikzpicture}
    \caption{The $\lhd$-ordering of $0$, $Q_{n}$, $Q_{n+1}$, and the first few translates of $Q_{n-1}$ by $Q_n$.}
    \label{fig:space-order}
\end{figure}
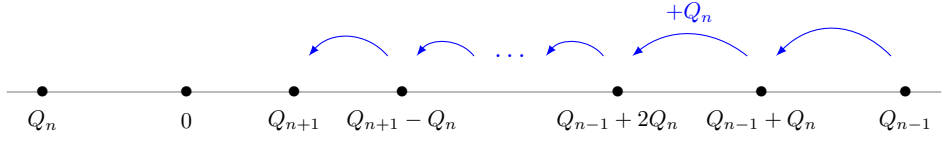

\begin{proof}
    Let $Q'_n = (-1)^{n+1} \varepsilon_1 Q_n$. 
    It satisfies the relation $Q'_{n-1} = a_{n+1} Q'_{n} + Q'_{n+1}$ for all $n \notin \mathbf{P}'$. 
    Let $\Upsilon$ be the semigroup generated by
    \[
        \{Q'_n\}_{n\in\Z} \cup \bigcup_{n \in \mathbf{P}'} \{Q'_{n-1} - k Q'_n \}_{k \geq 1}.
    \]
    With almost the same proof as in Proposition \ref{prop:chronological-order-02}, the continuant group $\Kbold$ is indeed equal to the disjoint union of $\Upsilon$, $\{0\}$, and $-\Upsilon$.
    This induces a unique translation-invariant order $\lhd$ on $\Kbold$ such that $\Upsilon$ is the corresponding positive cone.
\end{proof}

Properties \ref{space-1} and \ref{space-2} immediately imply the following lemma; see Figure \ref{fig:space-order}.

\begin{lemma}
    \label{lem:further-space-orderings}
    For $n \in \Z$, 
    \begin{enumerate}
        \item if $a_{n+1} < \infty$, the following set is written in $\lhd$-order (or reversed):
        \[
            0, \enspace Q_{n+1}=Q_{n-1} + a_{n+1}Q_n, 
            \enspace Q_{n-1}+(a_{n+1}-1)Q_n, \enspace \ldots,
            \enspace Q_{n-1}+Q_n,
            \enspace Q_{n-1};
        \]
        \item if $a_{n+1} = \infty$, the following set is written in $\lhd$-order (or reversed):
        \begin{align*}
            0, \enspace Q_{n+1}, 
            \enspace Q_{n+1}-Q_n, \enspace Q_{n+1}-2Q_n,, \enspace \ldots \qquad \qquad\\
            \qquad \qquad \ldots, \enspace Q_{n-1}+2Q_n, \enspace Q_{n-1}+Q_n,
            \enspace Q_{n-1};
        \end{align*}
        \item $\varepsilon_{n+1} Q_{[n]} \lhd 0$.
    \end{enumerate}
\end{lemma}

Using $\lhd$, we can define an abstract \emph{norm} $\| \cdot \|$ on $\Kbold$ as follows.
For $P = 0$, we will simply set $\| 0 \|:=0$.
For $P \neq 0$, we will introduce the notation $\| P \|$ to denote the unique element in $\{P,-P\}$ such that $-\| P \| \lhd \| P \|$.
The lemma below is elementary.

\begin{lemma}
\label{lem:elementary}
    $\| \cdot \|$ satisfies the triangle inequality:
    \[
        \| P + Q \| \trianglelefteq \| P \| + \| Q \|.
    \]
    For all $n \in \Z$,
    \[
        0 \trianglelefteq \| Q_{n} \| + \| Q_{n+1} \| \trianglelefteq \| Q_{n-1} \|.
    \]
    In particular, for every $n \in \Z$ and $m \geq 0$, we have
    \[
        \mathsf{F}_{m+1} \| Q_{n} \| \trianglelefteq \| Q_{n-m} \|,
    \]
    where $\mathsf{F}_m$ is the $m$\textsuperscript{th} Fibonacci number.
\end{lemma}

\begin{lemma}
\label{lem:limit-left-right-order}
    We have 
    \[
        \bigcup_{n \in \Z} \left\{ P \in \Kbold \: : \: \| P \| \trianglelefteq \| Q_n \| \right\} = \Kbold
    \]
    and
    \[
        \bigcap_{n \in \Z} \left\{ P \in \Kbold \: : \: \| P \| \trianglelefteq \| Q_n \| \right\} = \{0\}.
    \]
\end{lemma}

\begin{proof}
    Pick $P \in \Kbold \backslash \{0 \}$.
    Let us write $P = \sum_{j=1}^k c_j Q_{m_j}$ for some integers $c_k$ and $m_k$.
    Let $C := \sum_{j=1}^k |c_j|$, $M := \min_j m_j$, and let $m \geq 0$ be the smallest natural number such that $C \leq \mathsf{F}_{m+1}$.
    By Lemma \ref{lem:elementary},
    \[
        \| P \| \trianglelefteq \sum_{j=1}^k |c_j| \| Q_{m_j} \|
        \trianglelefteq C \| Q_{M} \| \trianglelefteq \| Q_{M-m} \|.
    \]
    Since $P$ is arbitrary, the inequality above implies the first desired equation.
    
    Next, let us denote $O = \{ P \in \Kbold \: : \: \| P \| \trianglelefteq \| Q_n \| \text{ for all } n \in \Z \}$.
    It follows from Lemma \ref{lem:elementary} that $O$ is a subgroup of $\Kbold$ and that it is $\lhd$-convex.
    To proceed, we will introduce a new order $\blacktriangleleft$ obtained by reversing $\lhd$ inside of $O$ as follows:
    \begin{itemize}
        \item if $P \not \in O$, we write $P \blacktriangleleft 0 $ if and only if $P \lhd 0$;
        \item if $P \in O \backslash \{0\}$, we write $P \blacktriangleleft 0$ if and only if $0 \lhd P$.
    \end{itemize}
    Observe that each $Q_n$ cannot lie inside of $O$.
    For this reason, it is straightforward to check that the new order $\blacktriangleleft$ still satisfies properties \ref{space-1} and \ref{space-2} of Proposition \ref{prop:left-right-order}.
    By the uniqueness of $\lhd$, the order $\blacktriangleleft$ must be equal to $\lhd$.
    Therefore, $O$ is the singleton $\{0\}$ and the second desired equation follows.
\end{proof}

\begin{lemma}[Small space $\Rightarrow$ large index]
\label{lem:small-space}
    For $P \in \Kbold$ and $n \in \Z$, 
    \[
    \| P \| \trianglelefteq \| Q_{n} \|
        \quad \xRightarrow[]{\qquad} \quad
    P \in \Kbold^{\geq n}.
    \]
\end{lemma}

\begin{proof}
    Clearly, $P$ is contained in $\Kbold^{\geq k}$ for some $k \in \Z$.
    Suppose for a contradiction that the maximal $k_*$ such that $P \in \Kbold^{\geq k_*}$ exists and satisfies $k_* \leq n-1$.
    Then, $\Kbold^{\geq k_*} \neq \Kbold^{\geq k_* + 1}$ and $a_{k_*+2} = \infty$.
    We can write $P = c Q_{k_*} + P'$ where $c \in \Z$ is non-zero and $P' \in \Kbold^{\geq k_*+1}$.

    We know that $\| P' \| \trianglelefteq C \| Q_{k_*+1} \|$ for some positive integer $C$.
    We also know that $(C+1) \| Q_{k_*+1} \| \lhd \| Q_{k_*} \|$. 
    Thus,
    \[
        \| P \| 
        \trianglerighteq 
        \| Q_{k_*} \| - \| P'\| 
        \trianglerighteq
        \| Q_{k_*} \| - C \| Q_{k_* + 1} \| 
        \rhd
        \| Q_{k_*+1} \| 
        \trianglerighteq
        \| Q_{n} \|. \qedhere
    \]
\end{proof}

The next result asserts that small space implies large time.

\begin{proposition}[Proper discontinuity]
\label{prop:proper-discontinuity}
    Let $P \in \Kbold$ and $n \in \Z$ be such that
    \[
        P > 0 \qquad \text{and} \qquad \| P \| \lhd \| Q_{n-1}\|.
    \]
    Then,
    \begin{enumerate}
        \item $P \geq Q_n$;
        \item if $P$ is $\lhd$-between $0$ and $Q_{n-1}$, then $P \geq Q_{n-1} + Q_n$.
    \end{enumerate}
\end{proposition}

For translational cascades with periodic combinatorics, this type of property first appeared in \cite[Section 2.13]{DL23}.

\begin{proof}
    Lemma \ref{lem:small-space} implies that $P$ is in $\Kbold^{\geq n-1}$.
    Let us make an additional assumption that $P < Q_n + Q_{n-1}$; otherwise, the lemma is trivial.
    This implies that $P$ is in $\Kbold^{\leq n}$ by virtue of Lemma \ref{lem:archimedean-convexity}.
    Altogether, by Lemma \ref{lem:intersecting-archimedean}, there exist integers $c$ and $d$ such that
    \[
        P = c Q_{n-1} + d Q_n.
    \]

    First, we claim that $cd \geq 0$.
    If otherwise, then $c Q_{n-1}$ and $d Q_n$ have the same $\lhd$-sign and so
    \[
        \| P \| = |c| \|Q_{n-1}\| + |d| \| Q_{n}\| \trianglerighteq \| Q_{n-1} \|,
    \]
    which is a contradiction to the hypothesis.

    Secondly, we claim that $d \geq 1$.
    Suppose otherwise.
    If $d=0$, $P > 0$ implies $c > 0$ as well as $\| P \| \trianglerighteq c \| Q_{n-1} \|$, which violates the hypothesis.
    If $d<0$, then the inequality $cd \geq 0$ implies $c \leq 0$ and altogether they violate $P > 0$.

    The two claims above imply that $c \geq 0$ and $d \geq 1$.
    Therefore, $P \geq Q_n$.
    If $P$ is $\lhd$-between $0$ and $Q_{n-1}$, then we cannot have $c = 0$ because of the alternating nature of $\lhd$, resulting in $P \geq Q_{n-1} + Q_n$.
\end{proof}

\subsection{Invertible topological cascades}
\label{ss:topological-cascade}

The next result states that $\Kbold$ can be realized as a unique invertible dynamical system on $\R$.

\begin{theorem}
\label{thm:topological-cascade}
    There exists an invertible dynamical system
    \[
        A = A_{\tttheta} = \{ A^P : \R \to \R \}_{P \in \Kbold}
    \]
    parametrized by $\Kbold$ with the following properties.
    \begin{enumerate}
        \item Each map $A^P$ is a monotonically increasing self-homeomorphism of $\R$.
        \item The map $(\Kbold, \lhd) \to (\R,<)$, $P \mapsto A^P(0)$ is order-preserving and injective.
        \item The orbit $\{A^P(0)\}_{P \in \Kbold}$ is dense in $\R$.
    \end{enumerate}
    Moreover, $A$ is unique up to topological conjugacy.
\end{theorem}

This settles items (1)--(3) of Theorem \ref{main-thm-3}.
The topological dynamical system $A_{\tttheta}$ described above is called an (\emph{invertible}) \emph{topological cascade} with \emph{combinatorics} $\tttheta$. 
In the next subsection, we demonstrate an explicit natural construction of $A_{\tttheta}$ when $\tttheta$ is irrational and prove item (4) of Theorem \ref{main-thm-3}.

\begin{proof}
    The ordered set $(\Kbold, \lhd)$ is countable and unbounded in the sense that there is no minimal or maximal $\lhd$-positive element (Lemma \ref{lem:limit-left-right-order}).
    By Cantor's isomorphism theorem in order theory (cf. \cite[Theorem 4.9]{HJ99}), there is an order-preserving bijection $\psi$ from $(\Kbold, \lhd)$ to $(\Q, <)$;
    we will assume that $\psi(0)=0$.
    Then, $\psi$ induces a group operation $\oplus$ on $\Q$ given by $\psi(P+Q) = \psi(P) \oplus \psi(Q)$.
    The standard order $<$ on $\Q$ is invariant under $\oplus$.

    The Dedekind completion of $(\Q,<)$ is order-isomorphic to $(\R,<)$.
    From the definition of Dedekind cuts, it is straightforward that the group action of $(\Q,\oplus)$ on $\Q$ itself extends to a group action by order-preserving bijective self-maps of $\R$.
    Via the map $\psi$, the existence of $A$ with the desired properties (1)--(3) then follows.

    Suppose $B=\{B^P: \R \to \R\}_{P \in \Kbold}$ is another dynamical system satisfying (1)--(3).
    The image of the map $\phi: \Kbold \to \R$, $P \to B^P(0)$ is a dense subset $\mathbb{S}$ of $\R$ equipped with a group operation $\odot$ compatible with the usual order $<$ such that $B^{P+Q}(0) = B^P(0) \odot B^Q(0)$.
    The map $\Phi = \phi \circ \psi^{-1} : (\Q,\oplus,<) \to (\mathbb{S},\odot, <)$ is an isomorphism of ordered groups. 
    By density, there exists a unique extension of $\Phi$ to a monotonically increasing self-homeomorphism of $\R$. 
    This extension is a conjugacy between $A$ and $B$.
    Thus, uniqueness holds.
\end{proof}

Let us fix $A = A_{\tttheta}$ described above.
The dynamics of $A$ can be described in terms of dynamical tilings.
Recall that for any two points $x$ and $y$ on the real line, $[x,y]$ denotes the closed real interval with endpoints $x$ and $y$.

For $P \in \Kbold$, denote the real number
\[
    V_P := A^P(0).
\]
For $n \in \Z$, denote the real interval
\[
    I_n := [V_{-Q_n-Q_{n-1}}, V_{-Q_n}].
\]

\begin{lemma}
    For $n \in \Z$, we have
    \[
        I_n \subset I_{n-1}, \qquad I_n \Subset I_{n-2}, \qquad I_n \Subset I_{n-1} \textnormal{ if and only if } a_{n+1} \geq 2.
    \]
    Moreover,
    \[
        \bigcap_{n\in\Z} I_n = \{0\}, \qquad \bigcup_{n \in \Z} I_n = \R.
    \]
\end{lemma}

\begin{proof}
    The inclusion is straightforward from the defining properties of $\lhd$.
    The two equalities follow from Lemma \ref{lem:limit-left-right-order}.
\end{proof}

\begin{proposition}[Return times]
    For $n \in \Z$, the first return map of $A$ back to $I_n$ is the commuting pair
    \begin{align*}
        A^{Q_n+Q_{n-1}}: [V_{-Q_n-Q_{n-1}},V_{-2Q_n-Q_{n-1}}] \to [V_0,V_{-Q_n}], \\
        A^{Q_n} : [V_{-2Q_n-Q_{n-1}},V_{-Q_n}] \to [V_{-Q_n-Q_{n-1}},V_0].
    \end{align*}
\end{proposition}

\begin{proof}
    The commuting pair is well-defined.
    The fact that $Q_n + Q_{n-1}$ and $Q_n$ are indeed the first return times (with respect to $<$) follows directly from Proposition \ref{prop:proper-discontinuity}. 
\end{proof}

\begin{proposition}[Parabolic fixed points]
\label{prop:parabolic-fixed-point}
    For every $n \in \mathbf{P}'$, there exists a unique fixed point $\beta_n$ of $A^{Q_n}$ on $I_n$, and it satisfies
    \[
        \lim_{k\to\infty} V_{-Q_{n+1} + k Q_n} 
        = \lim_{k\to\infty} V_{-Q_{n-1} - k Q_n} = \beta_n.
    \]
    Moreover, $\beta_n$ is a fixed point of $A^{Q_m}$ for all $m \geq n$.
\end{proposition}

\begin{proof}
    Without loss of generality, assume that $0 \lhd Q_n$.
    By Figure \ref{fig:space-order} and the density of the points $V_P$'s, in order to show the existence of $\beta_n$, we just need to show that the set
    \[
        \{ P \in \Kbold \: : \: k Q_n \lhd P \lhd -Q_{n-1} - k Q_n \textnormal{ for all } k \geq 1 \}
    \]
    is empty.
    
    Suppose for a contradiction that there exists some $P \in \Kbold$ such that
    \begin{align}
    \label{eqn:k-nesting}
        k \|Q_n \| \lhd P \lhd \|Q_{n-1}\| - k \|Q_n\|.
    \end{align}
    for all $k \geq 1$.
    By Lemma \ref{lem:small-space}, since $\| P \| \trianglelefteq \| Q_{n-1}\|$, then $P \in \Kbold^{\geq n-1}$.
    Since $a_{n+1} = \infty$, we can write $P$ as 
    \[
        P = P' + c Q_{n-1}
    \]
    for some unique $P' \in \Kbold^{\geq n}$ and $c \in \Z$.
    Note that there exists some integer $C>0$ such that $\| P' \| \trianglelefteq C \|Q_n\|$.
    There are three cases.
    \begin{itemize}
        \item Suppose $c=0$. Then, $\| P \| \trianglelefteq C \|Q_n\|$, which violates (\ref{eqn:k-nesting}) for $k = C$.
        \item Suppose $|c| \geq 2$. Then, $\| P \| \trianglerighteq 2 \| Q_{n-1} \| - C \| Q_n \| \rhd \| Q_{n-1} \|$, which again violates (\ref{eqn:k-nesting}).
        \item Suppose $|c| =1 $. Then, $\| P \| \trianglerighteq \| Q_{n-1} \| - C \| Q_n \|$, but this again violates (\ref{eqn:k-nesting}) for $k = C$.
    \end{itemize}
    In all cases above, we have arrived at a contradiction.

    The point $\beta_n$ is fixed under $A^{Q_m}$ when $m > n$ because the image under $A^{Q_m}$ of $[V_{-Q_{n+1}+k Q_n}, V_{-Q_{n-1}-kQ_n}]$ is contained in $[V_{-Q_{n+1}+ (k-1) Q_n}, V_{-Q_{n-1}-(k-1)Q_n}]$ and both intervals converge to $\{\beta_n\}$ as $k \to \infty$.

    It is clear that on the interval $[ V_{-Q_{n+1}},V_{-Q_n-Q_{n-1}}]$, $\beta_n$ is the unique fixed point of $A^{Q_n}$.
    On the rest of $I_n$, namely on $[V_{-Q_n}, V_{-Q_{n+1}}]$, every point escapes $I_n$ after two iterates of $A^{Q_n}$ and thus is not a fixed point of $A^{Q_n}$.
\end{proof}

The proposition above tells us that for $n \in \mathbf{P}'$, the orbit $\{A^P(\beta_n)\}_{P \in \Kbold}$ is precisely equal to $\{A^P(\beta_n)\}_{0 \leq P < Q_n}$, which is discrete.
In this enriched rational case, the cascade $A=A_{\tttheta}$ is topologically transitive, but it is not minimal because of the existence of periodic orbits.
In the irrational case, $A$ is minimal.

\subsection{Cascade of translations}
\label{ss:cascade}
 
Consider irrational subset $\TheBi$ of $\TheBiCpt$.
We have the following commutative diagram.
\begin{center}
{\large
\begin{tikzcd}[row sep=large, column sep=large]
    {\TheBi} \arrow[d, "\mathfrak{X} \circ \textnormal{proj}"'] \arrow[r, hook] \arrow[loop left, "\shift"] 
    & {\TheBiCpt} \arrow[d, "\textnormal{proj}"] \arrow[loop right, "\shift"]\\
    {\Irrat} \arrow[r, "\overline{\iota}"'] \arrow[loop left, "\gauss"]
    & {\TheCpt} \arrow[loop right, "\shift"]
\end{tikzcd}
}
\end{center}

From now on, we will fix an element $\tttheta = \seq{(\varepsilon_n, \abar_n)}_{n \in \Z}$ of $\TheBi$.
It naturally induces a bi-infinite sequence of irrationals 
\[
\theta_{n} = \mathfrak{X}\circ \textnormal{proj} \circ \shift^n(\tttheta) \in \Irrat, \quad n \in \Z,
\]
where $\gauss(\theta_n) = \theta_{n+1}$ for all $n$.
The bi-infinite tower of sector renormalizations 
\[
\{ \rotate_{\theta_n}: \overline{\D} \to \overline{\D} \}_{n \in \Z}
\]
can be put into a single dynamical plane as follows.

\begin{definition}
    We define the (\emph{invertible}) \emph{cascade of translations} associated to $\tttheta$ to be the group $F=F_{\tttheta}$ of translations generated by
    $\{F^{[n]}\}_{n \in \Z}$ where
    \[
        F^{[n]}(z) := z - \varepsilon_{n+1} l_{[n]} \quad \text{ and } l_{[n]}= l_{[n]}(\tttheta) := \begin{cases}
        |\theta_0 \ldots \theta_n| & \text{ if } n \geq 0, \\
        1 & \text{ if } n = -1, \\
        |\theta_{-1} \ldots \theta_{n+1}|^{-1} & \text{ if } n \leq -2.
    \end{cases}
    \]
\end{definition}

Let $F = F_{\tttheta}$ be the cascade of translations associated to $\tttheta$.
For $n \in \Z$, denote $b_n = b_n(\tttheta)$.
Let $\Kbold=\Kbold_{\tttheta}$ be the continuant group for $\tttheta$ defined previously.

\begin{lemma}
\label{lem:FQn}
    For all $n \in \Z$, we have
    \[
        F^{[n]} = (F^{[n-1]})^{b_n} \circ (F^{[n-2]})^{-\varepsilon_{n-1}\varepsilon_n}.
    \]
\end{lemma}

\begin{proof}
    From Lemma \ref{lem:r-shift}, we have $\theta_n = \varepsilon_n b_n-\frac{1}{\theta_{n-1}}$.
    Therefore,
    \begin{align*}
        \varepsilon_{n+1} |\theta_{n-1}\theta_n| = |\theta_{n-1}| \left( \varepsilon_n b_n-\frac{1}{\theta_{n-1}} \right) = b_n \cdot \varepsilon_n |\theta_{n-1}| - \varepsilon_{n-1} \varepsilon_n \cdot \varepsilon_{n-1}.
    \end{align*}
    By multiplying both sides by $l_{[n-2]}$, we obtain
    \[
        \varepsilon_{n+1} l_{[n]} 
        = b_n (\varepsilon_{n} l_{[n-1]}) - \varepsilon_{n-1} \varepsilon_{n} (\varepsilon_{n-1} l_{[n-2]}).
    \]
    This implies the desired equation.
\end{proof}

\begin{lemma}
\label{lem:parametrization}
    As abelian groups, $\Kbold$ is isomorphic to $F$ via the identification $Q_{[n]} \mapsto F^{[n]}$ for all $n \in \Z$.
\end{lemma}

\begin{proof}
    According to the previous lemma, there is a well-defined surjective homomorphism 
    \[
    \phi: \Kbold \to F, \quad P \mapsto F^P,
    \]
    where $F^{Q_{[n]}} = F^{[n]}$ for every $n \in \Z$.
    It remains to check that $\phi$ is injective, i.e. $F$ acts freely on $\R$.
    
    Fix $P \in \Kbold$. 
    It can be written as an integral combination of $Q_{[n]}, Q_{[n+1]}, \ldots, Q_{[n+k]}$ for some integers $n \in \Z$ and $k \geq 0$.
    By Lemma \ref{lem:free-group-2}, we can rewrite $P$ as
    $c_0 Q_{[n]} + c_1 Q_{[n+1]}$ for some integer coefficients $c_0, c_1$.
    Hence,
    \[
        F^P(z) = z - \varepsilon_{n+1} c_0 l_{[n]} - \varepsilon_{n+2} c_1 l_{[n+1]} = z -  \left( \varepsilon_{n+1} c_0 + \varepsilon_{n+2} c_1 |\theta_{n+1}| \right) l_{[n]}.
    \]
    Since $\theta_{n+1}$ is irrational, the equation $ \varepsilon_{n+1} c_0 + \varepsilon_{n+2} c_1 |\theta_{n+1}| = 0$ only has trivial solution $c_0=c_1 = 0$.
    This implies that $F^P$ is the identity map if and only if $P = 0$.
    Therefore, $\phi$ is injective.
\end{proof}

As a consequence of this proposition, we can parametrize $F$ in terms of elements of $\Kbold$ and write 
\[
    F = (F^P)_{P \in \Kbold} \qquad \text{where} \qquad
    F^{Q_{[n]}} \equiv F^{[n]}.
\]
From now on, we will denote
\[
    V_P := F^P(0) \qquad \text{ for } P \in \Kbold.
\]

Recall the time order $<$ and the space order $\lhd$ of $\Kbold$.
Recall the notion of topological cascades from \S\ref{ss:topological-cascade}.

\begin{theorem}
\label{thm:real-ordering}
    The group of translations $F$ is a topological cascade with combinatorics $\tttheta$.
    Every other topological cascade with combinatorics $\tttheta$ that consists of translations must be linearly conjugate to $F$.
\end{theorem}

This theorem, together with Theorem \ref{thm:topological-cascade}, implies Theorem \ref{main-thm-3}.

\begin{proof}
    As a translation, each map $F^P$, $P \in \Kbold$ is a monotonically increasing self-homeomorphism of $\R$.
    By Lemma \ref{lem:parametrization}, the map $\Kbold \to \R$, $P \mapsto V_P$ is injective.
    To prove that $F$ is a topological cascade, we need to show the following properties:
    \begin{enumerate}
        \item[(i)] For $P, R \in \Kbold$, $P \lhd R$ if and only if $V_P < V_R$.
        \item[(ii)] $\{V_P\}_{P \in \Kbold}$ is dense in $\R$.
    \end{enumerate}

    Let us define a strict total order $\blacktriangleleft$ on $\Kbold$ where $P_1 \blacktriangleleft P_2$ if and only if $V_{P_1} < V_{P_2}$.
    To prove (i), we just need to show that $\blacktriangleleft$ satisfies property \ref{space-1} defining $\lhd$ (property \ref{space-2} is vacuous).
    Recall the conversion $Q_{t_n} = Q_{[n]}$.
    Since $V_{Q_0} = -\varepsilon_1 l_{[0]}$, then it is clear that $\varepsilon_1 Q_0 \blacktriangleleft 0$.
    Property \ref{space-1} follows from the following observations.
    \begin{itemize}
        \item If $\varepsilon_n \varepsilon_{n+1} = -1$, then $Q_{t_{n-1}}$ and $Q_{t_n}$ have opposite $\blacktriangleleft$-signs:
        \[
        \frac{V_{Q_{t_{n-1}}}}{V_{Q_{t_n}}} 
        = \frac{V_{Q_{[n-1]}}}{V_{Q_{[n]}}}
        = - \frac{l_{[n-1]}}{l_{[n]}} < 0.
        \]
        \item If $\varepsilon_n \varepsilon_{n+1} = +1$, then $Q_{t_{n-2}}$ have the same $\blacktriangleleft$-sign as $Q_{t_n}$ which is opposite to $Q_{t_{n-1}}$:
        \begin{align*}
            \frac{V_{Q_{t_{n-2}}}}{V_{Q_{t_{n}}}} &= \frac{V_{Q_{[n-1]}}}{V_{Q_{[n]}}} = \frac{l_{[n-1]}}{l_{[n]}} > 0, \\
            \frac{V_{Q_{t_{n-1}}}}{V_{Q_{t_n}}} &= \frac{V_{Q_{[n]}}-V_{Q_{[n-1]}}}{V_{Q_{[n]}}} = 1- \frac{l_{[n-1]}}{l_{[n]}} = 1- \frac{1}{|\theta_n|} < 0.
        \end{align*}
    \end{itemize}

    Property (ii) follows directly from (i), the lemma below, and the fact that $|V_{-Q_n}| \to 0$ as $n \to \infty$.

    Lastly, let us prove the uniqueness part.
    Suppose $G$ is another topological cascade with the same combinatorics $\tttheta$ consisting of translations.
    Let $X_F = \{V_P = F^P(0)\}_{P \in \Kbold}$ and $X_G = \{G^P(0) \}_{P \in \Kbold}$.
    Both $X_F$ and $X_G$ are dense subgroups of $\R$, and the map $V_P \mapsto G^P(0)$ induces an order-preserving isomorphism $\phi$ between $X_F$ and $X_G$.
    By density, $\phi$ extends to an orientation-preserving self-homeomorphism of $\R$ that conjugates $F$ and $G$ and is additive: $\phi(x+y) =\phi(x)+\phi(y)$.
    Hence, $\phi$ must be linear.
\end{proof}

\begin{lemma}[Renormalization tiling]
    For $n \in \Z$, the collection $\mathcal{L}_n$ of intervals
    \[
        \{ [V_{-Q_{n}-P},V_{-P}]\}_{0 \leq P < Q_n+Q_{n-1}} 
        \cup
        \{ [V_{-P},V_{-Q_n-Q_{n-1}-P}]\}_{0 \leq P < Q_n}
    \]
    forms a tiling of $\R$. 
\end{lemma}

\begin{proof}
    This is a direct consequence of Proposition \ref{prop:proper-discontinuity}.
\end{proof}

For $n \in \Z$,
denote 
\[
    J_n := [V_{-Q_{[n]}}, V_{-\check{Q}_{[n]}}] 
    \quad
    \text{ where }
    \quad
    \check{Q}_{[n]} := Q_{[n]} - \varepsilon_n \varepsilon_{n+1} Q_{[n-1]}.
\]

\begin{lemma}
    For all $n \in \Z$, 
    \begin{enumerate}
        \item $0 \in J_n \Subset J_{n-1}$;
        \item $\{Q_{[n]}, \check{Q}{_{[n]}\} =  \{Q_{m},Q_m+Q_{m-1}}\}$ for some $m = m(n) \in \Z$.
    \end{enumerate}
\end{lemma}

\begin{proof}
    We have
    $V_{-Q_{[n]}} = \varepsilon_{n+1} l_{[n]}$ and $V_{-\check{Q}_{[n]}} = -\varepsilon_{n+1} (l_{[n-1]} - l_{[n]})$.
    Clearly, $V_{-Q_{[n]}}$ and $V_{-\check{Q}_{[n]}}$ have opposite signs and so $0$ is in the interior of $J_n$. 
    The containment $J_n \Subset J_{n-1}$ follows from:
    \[
        \max_{x \in \partial J_n} |x| = l_{[n-1]} - l_{[n]} < l_{[n-1]} =  \min_{x \in \partial J_{n-1}} |x|.
    \]

    Let us write $Q_{[n]} = Q_{t_n}$.
    To show property (2), there are two cases.
    \begin{itemize}
        \item Suppose $\varepsilon_n \varepsilon_{n+1} = -$.
        Then, $\check{Q}_{[n]} = Q_{t_n}+ Q_{t_n-1}$.
        \item Suppose $\varepsilon_n \varepsilon_{n+1} = +$.
        Then, $\check{Q}_{[n]} = Q_{t_n-1}$ and $Q_{[n]} = Q_{t_n-1} + Q_{t_n-2}$. \qedhere
    \end{itemize}
\end{proof}

For $n \in \Z$, the interval $J_n$ can be partitioned in two ways.
Firstly, we have $J_{n} = J_{n,0} \cup J_{n,1}$ where
    \[
        J_{n,0} = [V_{-\check{Q}_{[n]}},0] \quad \text{ and } \quad J_{n,1} = [0, V_{-Q_{[n]}}],
    \]
Secondly, we have $J_{n} = J'_{n,0} \cup J'_{n,1}$ where
\[
    J_{n,j}' := F^{-Q_{[n]} + j \varepsilon_n \varepsilon_{n+1} Q_{[n-1]}}(J_{n,j}), \qquad  j \in \{0,1\}.
\]
By the two lemmas above, the commuting pair
\[
    \mathcal{F}_n = \left( F^{Q_{[n]}}: J'_{n,0} \to J_{n,0}, \enspace F^{\check{Q}_{[n]}}: J'_{n,1} \to J_{n,1} \right)
\]
is the first return map of $F$ back to the interval $J_{n,0} \cup J_{n,1}$.

We will now outline how the translational cascade $F$ is related to the bi-infinite tower
\[
\{ \rotate_{\theta_n}: \overline{\D} \to \overline{\D} \}_{n \in \Z}.
\]
Roughly, for all $n \in \Z$,
\[
    \rotate_{\theta_n} = F^{[n]} \textnormal{ mod } F^{[n-1]}.
\]

Denote by $-\overline{\UHP}$ the closed lower half plane. 
For any closed interval $J \subset \mathbb{R}$, denote
\[
W(J) := \{ z \in -\overline{\UHP} \: : \: \real(z) \in J \}.
\]
For $n \in \Z$ and $j \in \{0,1\}$, denote
\[
    W_n = W(J_n), \qquad W_{n,j} = W(J_{n,j}), \qquad W_{n,j}' = W(J'_{n,j}).
\]
Each $W_n$ is a strip of horizontal width $l_{[n-1]}$. The second part of the lemma above implies that we have a bi-infinite nest of half-strips
\[
    \ldots \ \subset \ W_2 \ \subset \ W_1 \ \subset \ W_0 \ \subset \ W_{-1} \ \subset \ \ldots
\]
whose union is $-\overline{\UHP}$ and whose nested intersection is the negative imaginary axis.
The commuting pair $\mathcal{F}_n$ mentioned before extends to a commuting pair on $W_n$:
    \[
        \mathcal{F}_n = \left( F^{Q_{[n]}}: W'_{n,0} \to W_{n,0}, \enspace F^{\check{Q}_{[n]}}: W'_{n,1} \to W_{n,1} \right).
    \]
See Figure \ref{fig:cascade} for an illustration.

The next proposition draws the bridge between cascades and sector renormalization.
It essentially follows from the equation $\varepsilon_{n+1} l_{[n]} / l_{[n-1]} = \theta_n$.

\begin{figure}    
    \centering
    \begin{tikzpicture}[scale=1.15]
        \draw[black,thick] (-5,3.5) -- (5,3.5);
        \draw[red,thick] (0,3.5) -- (0,1);
        
        \draw[blue,thick] (-4.23,3.5) -- (-4.23,1);
        \draw[blue,thick] (2.62,3.5) -- (2.62,1);
        \node [blue, font=\bfseries] at (-4.23,3.75) {\small $V_{-\check{Q}_{[0]}}$};
        \node [blue, font=\bfseries] at (2.7,3.75) {\small $V_{-Q_{[0]}}$};
        \node [blue, font=\bfseries] at (-2,0.4) {\small $W_{0,0}$};
        \node [blue, font=\bfseries] at (1.3,0.4) {\small $W_{0,1}$};
        \draw[gray!70!blue] (-4.15,0.8) -- (-4.15,0.6) -- (-0.1,0.6) -- (-0.1,0.8);
        \draw[gray!70!blue] (0.1,0.8) -- (0.1,0.6) -- (2.5,0.6) -- (2.5,0.8);

        \draw[green!50!black,thick] (-1.61,3.5) -- (-1.61,1);
        \draw[green!50!black,thick] (1,3.5) -- (1,1);
        \node [green!50!black, font=\bfseries] at (-1.61,3.75) {\small $V_{-\check{Q}_{[1]}}$};
        \node [green!50!black, font=\bfseries] at (1.1,3.75) {\small $V_{-Q_{[1]}}$};
        \node [green!50!black, font=\bfseries] at (-0.8,2) {\small $W_{1,0}$};
        \node [green!50!black, font=\bfseries] at (0.55,2) {\small $W_{1,1}$};
        
        \node [red, font=\bfseries] at (0,3.75) {\small $0$};
        \draw[gray!60!black, -latex] (1,4.1) .. controls (0.35,4.7) and (-1,4.7) .. (-1.6,4.1);
        \draw[gray!60!black, -latex] (2.62,4.2) .. controls (0.6,5.5) and (-2,5.5) .. (-4.23,4.2);
        \node [gray!40!black, font=\bfseries] at (-0.7,4.8) {\small $F^{Q_{[0]}}$};
        \node [gray!40!black, font=\bfseries] at (1,5.3) {\small $F^{Q_{[-1]}}$};
\end{tikzpicture}
    \caption{Cascade $F=F_{\tttheta}$ with golden mean combinatorics $\tttheta= \seq{ \ldots, (+,2); (+,2), (+,2),\ldots }$}
    \label{fig:cascade}
\end{figure}

\begin{proposition}
    For $n \in \Z$, the map
\[
    \phi_n : W_n \to \overline{\D}, \quad \phi_n(z) = e^{-2\pi i z/ l_{[n-1]}}
\]
    satisfies the following properties.
\begin{enumerate}
    \item $\phi_n$ conformally sends the interior of $W_n$ to the open unit disk $\D$ minus the radial slit $\{ \arg z = -2\pi \theta_n \}$, and sends each of the horizontal sides of $W_n$ to this slit.
    \item $\phi_n$ projects the commuting pair $\mathcal{F}_n$ to the rotation $\rotate_{\theta_n}: \overline{\D} \to \overline{\D}$, and
    \item For $k \geq n+1$, $\phi_n$ maps $W_k$ conformally onto the sector $S_{\theta_n}^{k-n}$ and the commuting pair $\mathcal{F}_k$ projects to the first return map of $\rotate_{\theta_n}$ back to $S_{\theta_n}^{k-n}$.
\end{enumerate}
\end{proposition}

Lastly, we remark that Theorem \ref{thm:topological-cascade} allows for a generalization of the observation above. Namely, if $\tttheta \in \TheBiCpt$ is an enriched rational, there exists a bi-infinite sequence of circle homeomorphisms $h_n$ with rotation number $\theta_n$ such that
\[
    h_n = A^{Q_{[n]}} \text{ mod } A^{Q_{[n-1]}}.
\]

\bibliographystyle{alpha}
 
{\small \bibliography{bibliography}}

@article{AC17,
author = {Avila, Artur and Cheraghi, Davoud},
year = {2017},
pages = {2005--2062},
title = {Statistical properties of quadratic polynomials with a neutral fixed point},
volume = {20},
number = {8},
journal = {J. Eur. Math. Soc}
}

@article{BC12,
author = {Buff, Xavier and Ch\'eritat, Arnaud},
year = {2012},
pages = {673--746},
title = {Quadratic {Julia} sets with positive area},
volume = {176},
journal = {Ann. Math.}
}

@article{CC15,
author = {Cheraghi, Davoud and Cheritat, Arnaud},
year = {2015},
pages = {677--742},
title = {A proof of the {Marmi-Moussa-Yoccoz} conjecture for rotation numbers of high type},
volume = {202},
journal = {Invent. Math.}
}

@article{Che13,
author = {Cheraghi, Davoud},
year = {2013},
pages = {999--1035},
title = {Typical orbits of quadratic polynomials with a neutral fixed point: {Brjuno} type},
volume = {332},
number = {3},
journal = {Commun. Math. Phys.}
}

@article{Che19,
author = {Cheraghi, Davoud},
year = {2019},
pages = {59--138},
title = {Typical orbits of quadratic polynomials with a neutral fixed point: {Non-Brjuno} type},
volume = {52},
number = {1},
journal = {Ann. Sci. \'{E}c. Norm. Sup.}
}

@article{Che25,
author = {Cheraghi, Davoud},
year = {2025},
pages = {1321--1383},
title = {Topology of irrationally indifferent attractors},
volume = {6},
number = {58},
journal = {Ann. Sci. \'{E}c. Norm. Sup.}
}

@article{dF99,
  title={Asymptotic rigidity of scaling ratios for critical circle mappings},
  author={de Faria, Edson},
  journal={Ergod. Theory Dyn. Syst.},
  year={1999},
  volume={19},
  number={4},
  pages={995--1035}
}

@article{dFdM2,
  title={Rigidity of critical circle mappings {II}},
  author={de Faria, Edson and de Melo, Welington},
  journal={J. Amer. Math. Soc.},
  year={2000},
  volume={13},
  pages={343--370}
}

@book{dMvS93, 
    author = {de Melo, Welington and van Strien, Sebastian},
    publisher = {Springer-Verlag},
    title = {One-Dimensional Dynamics},
    year = {1993},
}

@incollection{D87,
author = {Douady, Adrien},
title = {Disques de {S}iegel et anneaux de {H}erman},
booktitle = {S\'eminaire Bourbaki : volume 1986/87, expos\'es 669-685},
journal = {Ast\'{e}risque},
series = {Ast\'erisque},
publisher = {Soci\'et\'e math\'ematique de France},
number = {152-153},
year = {1987},
pages = {4, 151--172},
issn = {0303-1179,2492-5926}
}

@article{DLS,
author = {Dudko, Dzmitry and Lyubich, Mikhail and Selinger, Nikita},
year = {2020},
pages = {653--733},
title = {Pacman renormalization and self-similarity of the {Mandelbrot} set near {Siegel} parameters},
volume = {33},
number = {3},
journal = {J. Amer. Math. Soc.}
}

@misc{DL22,
author = {Dudko, Dzmitry and Lyubich, Mikhail},
year = {2022},
title = {Uniform a priori bounds for neutral renormalization},
publisher = {arXiv},
eprint = {2210.09280},
doi = {10.48550/arXiv.2210.09280},
note = {\href{https://arxiv.org/abs/2210.09280}{arXiv.2210.09280}}
}

@article{DL23,
author = {Dudko, Dzmitry and Lyubich, Mikhail},
year = {2023},
title = {{Local connectivity of the Mandelbrot set at some satellite parameters of bounded type}},
journal = {Geom. Funct. Anal.},
volume = {33},
pages={912--1047}
}

@misc{DLL,
author = {Dudko, Dzmitry and Lim, Willie Rush and Lyubich, Mikhail},
year = {2026},
title = {Rigidity of the attractor of neutral quadratic polynomials},
publisher = {arXiv},
eprint = {2609.37383},
doi = {10.48550/arXiv.2609.37383},
note = {\href{https://arxiv.org/abs/2609.37383}{arXiv.2609.37383}}
}

@misc{DL26,
author = {Dudko, Dzmitry and Lyubich, Mikhail},
year = {2026},
title = {Uniform a priori bounds for neutral renormalization. {Variation I: Sector Renormalization}},
note = {Manuscript}
}

@book{HJ99,
  title={Introduction to Set Theory},
  author={Hrbacek, Karel and Jech, Thomas},
  edition={3rd},
  year={1999},
  publisher={Marcel Dekker},
  isbn={9780824779153}
}

@article{Hur,
mylabel = {H1889},
author = {Hurwitz, Adolf},
year = {1889},
pages = {367--405},
title = {\"{U}ber eine besondere {A}rt der {K}ettenbruch-{E}ntwicklung reeller Gr\"o{\ss}en},
volume = {12},
journal = {Acta Math.}
}

@misc{IS,
author = {Inou, Hiroyuki and Shishikura, Mitsuhiro},
publisher = {arXiv},
year = {2006},
title = {The renormalization for parabolic fixed points and their perturbation},
note = {\href{https://www.math.kyoto-u.ac.jp/~mitsu/pararenorm/ParabolicRenormalization.pdf}{Preprint}}
}

@book{Khin,
    author = {Khinchin, Aleksandr Y.},
    title = {Continued fractions},
    publisher = {Dover Publications, Inc.},
    year = {1997},
    address = {Mineola, NY},
    note = {With a preface by B. V. Gnedenko.
Translated from the third (1961) Russian edition. Reprint of the 1964 translation.}
}

@Inbook{Lan88,
author={Lanford, Oscar E.},
title={Renormalization Group Methods for Circle Mappings},
bookTitle={Nonlinear Evolution and Chaotic Phenomena},
year={1988},
publisher={Springer US},
pages={25--36},
}

@article{Lim24,
author = {Lim, Willie Rush},
year = {2026},
title = {Hyperbolicity of renormalization of critical quasicircle maps},
journal = {Commun. Math. Phys},
volume = {407},
number = {88}
}

@misc{Lim26,
author = {Lim, Willie Rush},
year = {2026},
title = {Lebesgue measure of the postcritical set of neutral quadratic polynomials},
publisher = {arXiv},
eprint = {2609.35563},
doi = {10.48550/arXiv.2609.35563},
note = {\href{https://arxiv.org/abs/2609.35563}{arXiv.2609.35563}}
}

@article{McM98,
author = {McMullen, Curtis},
year = {1998},
pages = {247--292},
title = {Self-similarity of {Siegel} disks and {Hausdorff} dimension of {Julia} sets},
volume = {180},
number = {2},
journal = {Acta Math.}
}

@article{Mi1873,
    author = {Minnigerode, Bernhard},
    title = {\"{U}ber eine neue Methode, die Pellsche Gleichung aufzul\"{o}sen},
    journal = {Nachr. G\"ottingen},
    year = {1873},
    pages = {619--653}
}

@article{Na81,
author = {Nakada, Hitoshi},
year = {1981},
pages = {399--426},
title = {Metrical theory for a class of continued fraction transformations and their natural extension},
volume = {4},
number = {2},
journal = {Tokyo J. Math.}
}

@article{Si42,
year = {1942},
volume = {43},
number = {4},
pages = {607--612},
author = {Siegel, Carl Ludwig},
title = {Iteration of analytic functions},
journal = {Ann. Math.}
}

@article{SY24,
    author = {Shishikura, Mitsuhiro and Yang, Fei},
    title = {The high type quadratic {Siegel} disks are {Jordan} domains},
    journal = {J. Eur. Math. Soc.},
    year = {2024},
    volume = {27},
    number = {11},
    pages = {4501--4562}
}

@article{Y99,
  title={Complex bounds for renormalization of critical circle maps},
  author={Yampolsky, Michael},
  journal={Ergod. Theory Dyn. Syst.},
  year={1999},
  volume={19},
  number={1},
  pages={227--257}
}

@article{Y01,
  title={The attractor of renormalization and rigidity of towers of critical circle maps},
  author={Yampolsky, Michael},
  journal={Commun. Math. Phys.},
  volume = {218},
  number = {3},
  year = {2001},
  pages = {537--568}
}

@article{Y02,
  title={Hyperbolicity of renormalization of critical circle maps},
  author={Yampolsky, Michael},
  journal={Publ. Math. Inst. Hautes {\'E}tudes Sci.},
  volume = {96},
  year={2002},
  pages = {1--41}
}

@article{Y03,
  title={Renormalization horseshoe for critical circle maps},
  author={Yampolsky, Michael},
  journal={Commun. Math. Phys.},
  volume = {240},
  year={2003},
  number={1--2},
  pages = {75--96}
}

@incollection{Yoc95,
     author = {Yoccoz, Jean-Christophe},
     title = {Th\'eor\`eme de {Siegel,} nombres de {Bruno} et polyn\^omes quadratiques},
     booktitle = {Petits diviseurs en dimension 1},
     series = {Ast\'erisque},
     pages = {1--88},
     publisher = {Soci\'et\'e math\'ematique de France},
     number = {231},
     year = {1995}
}

\end{document}